%% file: main.tex
\documentclass[10pt]{article}
\usepackage[utf8]{inputenc}
\usepackage[english]{babel}
\usepackage{mathrsfs}
\usepackage{amsthm}
\usepackage{enumitem}
\usepackage{amsmath}
\usepackage{mathtools}
\usepackage{amssymb}

\newcommand{\Xtildell}{(\widetilde{X},\widetilde{d},\widetilde{\ll},\widetilde{\leq},\widetilde{\tau}) }

\newcommand{\Xll}{(X,d,\ll,\leq,\tau) }
\def\confXll{(X,d,\ll_{\Omega},\leq_{\Omega},\tau_{\Omega}) }

\usepackage{tikz}
\usetikzlibrary{decorations.markings}
\usepackage{pgfplots}
\pgfplotsset{compat=1.18}

\usepackage{dirtytalk} 
\usepackage{xcolor} 
\usepackage[letterpaper,top=3cm,bottom=3cm,left=3cm,right=3cm,marginparwidth=0cm]{geometry}
\usepackage{comment}

\usepackage[normalem]{ulem}

\usepackage{authblk} 
\title{Conformal transformations of metric spaces\\ and Lorentzian pre-length spaces}
\author{
Miguel Manzano$^1$\thanks{{\tt m.manzano.rod@gmail.com}}\ ,
Karim Mosani$^1$\thanks{{\tt karim.mosani@univie.ac.at}}\ ,
Clemens S{\"a}mann$^1$\thanks{{\tt clemens.saemann@univie.ac.at}}\ \ 
and Omar Zoghlami$^1$\thanks{{\tt omar.zoghlami@univie.ac.at}} \\ 
$^1$ Faculty of Mathematics, University of Vienna, \\
Oskar-Morgenstern-Platz 1, 1090 Vienna, Austria.
}
\date{\today}

\usepackage[
  backend=biber,
  style=alphabetic,
  maxnames=99,
  giveninits=true,  
  sorting=nyt       
]{biblatex}
\AtEveryBibitem{%
  \clearfield{issn}%
}
\renewbibmacro*{in:}{
  \ifboolexpr{
    test {\iffieldequalstr{eprinttype}{arxiv}} 
    or
    test {\ifentrytype{article}}               
  }
    {}
    {\printtext{\bibstring{in}\intitlepunct}}}

\usepackage{csquotes}
\usepackage{hyperref}
\hypersetup{
  colorlinks   = true, 
  urlcolor     = blue, 
  linkcolor    = blue, 
  citecolor   = olive 
}

\usepackage{zref-clever}
\let\oldnewtheorem\newtheorem

\RenewDocumentCommand{\newtheorem}{momo}{
  \IfValueTF{#2}{
    \AddToHook{env/#1/begin}{
      \zcsetup{countertype={#2=#1}}}
      \zcRefTypeSetup{#1}{
Name-sg = #3 ,
      }
    \oldnewtheorem{#1}[#2]{#3}
  }{
    \AddToHook{env/#1/begin}{
      \zcsetup{countertype={#1=#1}}}
    \zcRefTypeSetup{#1}{
Name-sg = #3 ,
      }
    \IfValueTF{#4}{
      \oldnewtheorem{#1}{#3}[#4]
    }{
      \oldnewtheorem{#1}{#3}
    }
  }
}

\newcommand{\cref}[1]{\zcref{#1}}
\newcommand{\Cref}[1]{\zcref[S]{#1}}

\input{macros}

\input{theorem_enviroments}

\usepackage{setspace}
\usepackage{tikz}
\usetikzlibrary{calc}

\usepackage{tocloft}

\begin{document}
        
\maketitle

\vspace{-0.7cm}

\begin{abstract}
    We introduce conformal transformations in the synthetic setting of metric spaces and Lorentzian (pre-)length spaces. Our main focus lies on the Lorentzian case, where, motivated by the need to extend classical notions to spaces of low regularity, we provide the first consistent notion of conformal length, and analyse its fundamental properties. We prove that the conformal time separation function $\tau_{\Omega}$ (and the causal structure it induces) yields a Lorentzian pre-length structure if the original space is intrinsic and strongly causal. This allows us to construct a notion of conformal transformation between spaces within this class, yielding an equivalence relation. As applications, we show that the conformal length functional agrees with the standard conformal length of (strongly causal) spacetimes. We also prove conformal invariance of angles and causality conditions, give a characterisation of global hyperbolicity via finiteness of $\tau_{\Omega}$ for all conformal factors, and establish the behaviour of the Lorentzian Hausdorff measure defined in \cite{McCann_2022} under conformal changes. Moreover, we apply the same methods to the metric case, which is of interest in its own right. This is exemplified by proving an analog of the Nomizu--Ozeki theorem for metric length spaces, which has the advantage that the resulting complete space is conformally related to the original space.
\vskip 1em

\medskip

\end{abstract}

\newpage
\begingroup
\tableofcontents
\endgroup

\section{Introduction}\label{Introduction}
\hypertarget{intro-target}{}

 Conformal transformations play a fundamental role in Einstein's general theory of relativity. While these transformations alter lengths of curves induced by the spacetime metric tensor and also alter the curvature quantities, they preserve the light cone structure at each point --- that is, the distinction between timelike, spacelike, and null vectors, as well as future and past, remains unchanged under such rescalings of the spacetime metric. This property makes conformal transformations an invaluable tool for studying aspects of spacetime that depend only on causality. Moreover, the conformal structure (or light cone structure) and the spacetime volume completely determine the Lorentzian manifold \cite{HawkingKingMcCarthy1976, Mal:77}. This is elevated to a fundamental principle in causal set theory, an approach to quantum gravity \cite{BLMS:87}, encapsulated in the slogan ``number + order = geometry", see e.g.\ \cite{Sur:25, Bra:25b}.

 In geometry, the study of conformal transformations dates back to the classical works of Cotton~\cite{cotton1899varietes}, Schouten~\cite{schouten1921konforme}, and Weyl~\cite{weyl1918reine,Weyl1921}.
 In general relativity, their significance stems from the seminal contributions of Penrose~\cite{penrose1963asymptotic,penrose1964conformal}, who provided the first geometric description of the asymptotic structure of spacetime. Conformal methods remain a vibrant area of research, see for instance, the proceedings \cite{Frauendiener:2002mm}, the monograph by Valiente Kroon \cite{Kro:16}, and the references therein, with applications like the construction of conformal extensions of exact solutions of the Einstein field equations, the analysis of asymptotic spacetime behaviour, the structure of spacelike and null infinity, and the geometry of the light cone.
 
 We say that there exists a conformal transformation from one spacetime to another if the spacetimes are conformally related, i.e., if one can be obtained from the other by a smooth, positive rescaling of the metric tensor. The precise definition is as follows:
\begin{definition}[Conformal relation]\label{definition:conformal relation}
      Two spacetimes $(\mathcal{M}_1,g_1)$ and $(\mathcal{M}_2,g_2)$ are said to be conformally related if there exists a diffeomorphism $\iota: \mathcal{M}_1 \to \mathcal{M}_2$ and a positive smooth map $\Omega: \mathcal{M}_2\to \mathbb{R}_{>0}$, called {\em conformal factor}, such that $\iota^{*}\left(\Omega^{-2}g_2\right)=g_1$.
    \end{definition}

The aim of this article is to provide a synthetic definition of a conformal relation, that is, to define and study the notion of conformal relation without any smooth or manifold structure. In particular, we achieve this within the framework of metric spaces and Lorentzian pre-length spaces.

In general relativity, spacetimes are modelled as connected smooth manifolds endowed with a Lorentzian metric tensor that is traditionally assumed to be sufficiently smooth (at least of class $C^2$) so that fundamental geometric constructions --- such as the Levi-Civita connection, curvature tensors, and geodesic equations --- are well-defined in the classical sense. However, many physically relevant spacetimes possess metric tensors of lower regularity. Examples include impulsive gravitational waves \cite{Podolsky_2022} (see also §20 of \cite{Griffiths_2009}), and models obtained by matching spacetime regions, such as those describing stellar collapse \cite{Hogen_2004, Samad_2023}. These and similar examples have motivated a systematic investigation of general relativity under weakened regularity assumptions on the metric tensor.

Beyond these specific models, the singularity theorems represent one of the cornerstones of general relativity, demonstrating that under physically reasonable energy and causality conditions, spacetimes can become geodesically incomplete as a result of gravitational collapse (see, e.g.\ \cite{Senovilla_1998, Sen:22, Ste:23} for extensive reviews of the singularity theorems). The classical formulations --- developed by Penrose and later by Hawking --- were established for smooth Lorentzian manifolds \cite{Hawking_1973}. This leaves open the question of whether they imply the existence of an incomplete causal geodesic (hence a singularity) or just the breakdown of the regularity of the spacetime metric tensor. To answer this question and fill this gap, one needed to prove the singularity theorems for metrics of low regularity --- a need that was already spelled out explicitly in \cite{Hawking_1973}. This line of thought also leads to the natural question of whether spacetimes can be extended beyond their singularities when the regularity assumptions on the metric tensor are suitably relaxed.

Two complementary approaches have emerged in the study of low-regularity spacetimes. The first investigates manifolds with Lorentzian metric tensors of limited smoothness or cone structures. A foundational contribution in this direction is due to Minguzzi \cite{Min:15} and  Kunzinger--Steinbauer--Stojković--Vickers \cite{Kunzinger_2013}, who showed that many standard results of causality theory remain valid even if the regularity of the metric tensor is reduced to $C^{1,1}$. For cone structures we refer to \cite{FS:12, BS:18, Min:19a} and references therein. The second approach takes a synthetic or metric-space perspective: rather than relying on a differentiable manifold and a smooth metric tensor, it seeks to capture the causal and geometric structure of spacetime using purely metric notions. This approach is grounded in the classical result of Hawking, King, and McCarthy \cite{HawkingKingMcCarthy1976}, which establishes that, for a strongly causal spacetime, the time-separation function $\tau$ --- also known as the Lorentzian distance, in analogy with the distance function in metric spaces --- completely determines the metric tensor. This line of inquiry is inspired by advances in metric geometry, where curvature is understood via comparison properties of geodesic triangles (as in Alexandrov and CAT($K$)-spaces) rather than through smooth differential structure. Following this analogy, Kunzinger and the third author introduced the framework of Lorentzian Length Spaces \cite{Kunzinger_2018} as a Lorentzian counterpart of classical length spaces \cite{Burago_2001}. This allows one to study spacetime geometry in a setting where both the metric tensor and the manifold structure are absent, yet causal and geometric information is retained in a synthetic form, see \cite{Sae:24} for a recent review of this approach.

In recent years, substantial progress has been achieved on both fronts. Regarding the first approach, the classical singularity theorems have been extended to metrics of lower and lower regularity, thereby closing the gap mentioned above. Most recently, Calisti et al.\ \cite{Calisti_2025} established a version of Hawking’s singularity theorem under the assumption that the Lorentzian metric is merely Lipschitz continuous (cf.\ the discussion therein for the development up to the Lipschitz Hawking theorem and see also \cite{Ste:23}). Earlier, in a landmark result, Sbierski proved that the maximally extended analytic Schwarzschild spacetime is inextendible beyond its central singularity even under $C^0$ regularity of the metric tensor \cite{Sbierski_2018b, Sbierski_2018}, and more recently extended this approach considerably \cite{Sbi:22, Sbi:25}.

 As for the second approach, building on the foundational work of \cite{Kunzinger_2018}, Beran et al.\ \cite{Beran_2023} recently established a splitting theorem for globally hyperbolic Lorentzian length spaces, extending classical results to the synthetic setting. Furthermore, Beran, Kunzinger and Rott \cite{Beran_2024} introduced and analysed several notions of sectional curvature bounds for Lorentzian pre-length spaces, proving their equivalence under mild conditions. Moreover, causality conditions and the causal ladder have been established in \cite{ACS:20} by Aké Hau, Cabrera Pacheco and Solis.

 Complementing these advances, tools from optimal transport have been found to be essential for studying general relativity in the synthetic setup. For example, McCann \cite{McC:20} and independently, Mondino and Suhr \cite{Mondino_2022} developed an optimal transport formulation of (lower) timelike Ricci curvature bounds. Mondino--Suhr also defined upper timelike Ricci curvature bounds, which enabled them to define purely synthetic Einstein equations for Lorentzian length spaces. Additionally, versions of the Hawking and Penrose singularity theorems have also been derived by Cavalletti and Mondino \cite{Cavalletti_2024}, and in collaboration with Manini \cite{CMM:24, Cavalletti_2025}, using techniques from optimal transport theory. These are just a few selected results, for a more comprehensive view see the recent reviews \cite{CM:22, Sae:24, McC:25, Bra:25}. These developments indicate that this synthetic approach to Lorentzian geometry via Lorentzian length spaces (or their variants) is a natural and powerful way of investigating curvature, causality, and singularities within a non-smooth context.

 Within this broader research program, one of the key directions concerns the study of black hole spacetimes (see \S 12.1 of \cite{Wald_1984} for the classical definition of a black hole). Although substantial progress has been made in extending causality theory to low-regularity or synthetic settings --- as discussed in the previous paragraphs --- the concept of black holes has, so far, received comparatively little attention within the Lorentzian length space framework, even in situations where such an extension seems to appear feasible in principle \cite{Hau_2022, Burgos_2023}. As the classical definition of a black hole is given in terms of asymptotic flatness, a first step towards a synthetic notion of black holes is therefore to define asymptotic flatness.
 
 This would have far-reaching implications beyond black holes, for instance for the analysis of
 timelike and null infinity, the formulation of the weak and strong cosmic censorship conjectures
 (see \cite[§11.1, 11.2, 12.1]{Wald_1984}), and the theory of gravitational radiation \cite{penrose1963asymptotic}. Consequently, to even get started, we require a synthetic notion of conformal transformation and conformal equivalence, which we introduce in this article. Before we do this we comment briefly on other work.

There is another approach to define conformal transformations via the \emph{causal speed} of causal curves (introduced in \cite{BBCGMORS:24}) in upcoming work by Braun and Sálamo Candal \cite{Bra:25c}, which works in the \emph{measured} Lorentzian synthetic setting. Also, there has been a previous attempt to define a synthetic notion of conformal transformations by Ebrahimi, Vatandoost, and Pourkhandani \cite{ebrahimi23}. While this work represents a valuable step toward a synthetic conformal framework, it encounters fundamental challenges: the proposed conformal length is not additive, and the supremum of this length over all causal curves between two fixed points fails to satisfy the reverse triangle inequality, preventing it from being a time separation function. We illustrate these problems using an example based on a subset of two-dimensional Minkowski spacetime in \Cref{section3.1}. Owing to these foundational inconsistencies, the definitions in \cite{ebrahimi23} do not yield a coherent notion of synthetic conformal transformation within Lorentzian (pre-)length spaces. In what follows, we adopt a more cautious and structurally grounded framework that preserves consistency with classical Lorentzian geometry.

Finally, we apply the same methods used in the Lorentzian case to the metric case. Thereby we get, in a straightforward manner, analogous results for metric (length) spaces, that were (to the best of our knowledge) not available in the literature in this form. In particular, our approach allows us to prove an analogue of the Nomizu--Ozeki theorem from Riemannian geometry, i.e., we show that on any locally compact metric length space there exists a complete metric that not only induces the same topology but is conformally related to the original one, see \Cref{thm:make_a_space_complete}.

The structure of the paper is as follows. In \Cref{Preliminaries}, we provide preliminary definitions relevant to our construction, in particular we recall the basics of Lorentzian (pre-)length spaces. In \Cref{Conformal transformations of Lorentzian pre-length spaces}, after discussing the issues with the previous approach of \cite{ebrahimi23} in \Cref{section3.1}, we introduce the notions of \textit{conformal variation} and \textit{conformal length} for a given conformal factor (\Cref{definition:conformal-factor-variation-length}) and establish their fundamental properties. In \Cref{A class of spaces closed under conformal transformations}, we define a map $\tau_{\Omega}: X \times X \to \mathbb{R}$, where $X$ is a Lorentzian pre-length space, as the supremum of the conformal length of causal curves between two points, and show that it satisfies the reverse triangle inequality. We then prove that the class of intrinsic and (quasi-)strongly causal Lorentzian pre-length space is closed under conformal transformations by showing that $\tau_{\Omega}$ is lower semi-continuous (and hence a time separation function in the sense of Definition 2.8 in \cite{Kunzinger_2018}) and intrinsic whenever the original space is intrinsic. In \Cref{Uniqueness of the conformal time separation function}, we discuss under which conditions distinct conformal factors give rise to distinct conformal time separation functions. In \Cref{Conformal transformations as an equivalence relation}, we define a conformal relation between intrinsic (quasi-)strongly causal Lorentzian pre-length spaces (\Cref{def:conf_rel_LLS}) and prove that it forms an equivalence relation. In \Cref{thm-con-ang-pre} we show that under conformal changes angles between timelike curves of the same time orientation are preserved. As applications, we establish results connecting the synthetic framework to classical Lorentzian geometry in \Cref{sectionapplications}. In \Cref{Connection with smooth spacetimes}, we show that for strongly causal spacetimes, our synthetic notion of conformal relation coincides with the classical smooth one (\Cref{definition:conformal relation}) by proving that the conformal length of a causal curve agrees with the length induced by the conformally transformed metric tensor. In \Cref{subsec-cc}, we exhibit the conformal invariance of the causality conditions. In \Cref{Global hyperbolicity and finiteness of Lorentzian distance}, we establish a synthetic analogue of the classical result linking global hyperbolicity with the finiteness of the conformal time separation function, showing that an intrinsic, non--totally imprisoning, quasi-strongly causal Lorentzian pre-length space is globally hyperbolic if and only if the conformal time separation function associated with any conformal factor is finite. In \Cref{Lorentzian Hausdorff measure}, we derive the relation between the $s$-dimensional Lorentzian Hausdorff measure (\cite{McCann_2022}) of a Borel set $E \subset X$ (with $X$ a Lorentzian pre-length space) before and after a conformal transformation of $X$. In \Cref{Conformal Length on Metric Spaces}, we apply the same methods as in \Cref{Conformal transformations of Lorentzian pre-length spaces} in the case of metric spaces. In \Cref{basic_properties_metric_case}, we introduce the notion of conformal variation and conformal length for a given conformal factor (\Cref{conformal variation and length in ms}), and establish their fundamental properties. In \Cref{section5.2}, we define the induced conformal distance as a map $d_{\Omega}: X \times X \to \mathbb{R}$, where $X$ is a metric space, as the infimum of the conformal length of causal curves between two points, and establish its essential properties. As applications, in \Cref{section 5.3}, we prove a synthetic version of Nomizu--Ozeki's theorem, and in \Cref{section 5.4}, we derive the relation between Hausdorff measures (of a given Borel set) induced by comformally related metrics (just like in \Cref{Lorentzian Hausdorff measure}).

\section{Preliminaries}\label{Preliminaries}

In this section, we review the fundamental concepts and results to be used throughout the paper. We divide the section into two parts: the first covers a number of basic notions of metric geometry and length spaces, while the second focuses on Lorentzian pre-length spaces. For further details, we refer to \cite{Burago_2001} and \cite{Kunzinger_2018}.

\subsection{Metric geometry and length spaces}

We next recall the well-known concepts of metric space, length structure, length space, variational length and metric speed.

\begin{definition}[Metric space]
    Let $X$ be a set and $d\colon X\times X\longrightarrow [0,\infty)$ a function such that, for all $x,y,z\in X$: $(i)$  $d(x,y)=0$ if and only if $x=y$, $(ii)$ $d(x,y)=d(y,x)$, and $(iii)$ $d(x,y)\leq d(x,z)+d(z,y)$. Then, \(d\) is called {\em distance} and $(X,d)$ is called a {\em metric space} .
\end{definition}

\begin{definition}[Length structure]\label{def-len-str}
Let $X$ be a Hausdorff topological space. A length structure $(X,A,L)$ on $X$ consists of a subset $A$ of all (admissible) paths in $X$ together with a length functional $L: A \longrightarrow [0,+\infty]$ such that the following holds.
\begin{enumerate}[label=(\roman*)]\itemsep0cm
    \item 
    If $\gamma : [a,b] \to X$ is admissible and $a \le c < d \le b$, then 
    $\gamma|_{[c,d]}$ 
    is also admissible.

    \item 
    If $\gamma_1\colon[a,c]\rightarrow X$ and $\gamma_2\colon [c,d]\rightarrow X$ are admissible for 
    $a < c < b$, so is their concatenation $\gamma_1\gamma_2 : [a,b] \to X$.

    \item 
    For $A \ni \gamma : [a,b] \to X$ and a homeomorphism $\varphi : [c,d] \to [a,b]$ of the form $\varphi(t) = \alpha t + \beta$ ($\alpha, \beta \in \mathbb{R}$), the composition $\gamma \circ \varphi : [c,d] \to X$ is also an admissible path.

    \item 
    If $\gamma : [a,b] \to X$ is admissible, then 
    \(
    L(\gamma) = L(\gamma|_{[a,c]}) + L(\gamma|_{[c,b]})
    \)
    for any $c \in [a,b]$.

    \item 
    If $\gamma : [a,b] \to X$ is admissible and of finite length,
    $t \mapsto L(\gamma|_{[a,t]})$ is continuous and the length of all constant paths vanishes.

    \item
    For $\gamma, \varphi$ as in (A3), 
    $L(\gamma \circ \varphi) = L(\gamma)$.

    \item 
    For any neighbourhood $U$ of ${x\in X}$, the length of paths connecting $x$ with a point in the complement of $U$ is uniformly bounded away from zero, i.e.\ 
    \(
        \inf \{ L(\gamma) : \gamma(a) = x, \gamma(b) \in X \setminus U \} > 0.
    \)
\end{enumerate}
\end{definition}

\begin{definition}[Variational length]
Let $\gamma : [a,b] \to X$ be a path in a metric space $(X,d)$, and $\sigma = \{ t_0=a, \ldots, t_n=b \}$ a partition of $[a,b]$. The {\em variation $V^d_\sigma(\gamma)$ of $\gamma$ with respect to $\sigma$}, and the {\em variational length $L^d(\gamma)$ of $\gamma$} are defined as
\begin{align*}
V^d_\sigma(\gamma) &\coloneqq \sum_{i=0}^{n-1} d\big(\gamma(t_i), \gamma(t_{i+1})\big),\qquad\qquad  L^d(\gamma):=\sup \big\{V_{\sigma}^d(\gamma):\sigma\textup{ is a partition of }[a,b]\big\}.
\end{align*}
\end{definition}

\begin{definition}[Length space]
A metric space $(X, d)$ is called a {\em length space} if 
$d=d_L$, where 
\(
    d_L(x,y)\coloneqq \inf\{L(\gamma)\  \colon \ \gamma \in A, \gamma(a)=x, \gamma(b)=y\}
\)
for any $x,y\in X$ and $(X, A, L)$ is the length structure coming from the metric $d$ (i.e., $L=L^d$). In such case, $d$ is called {\em intrinsic} or {\em length metric}.
\end{definition}

\begin{definition}[Metric speed] \label{def:metric_speed}
Let $(X,d)$ be a metric space and consider a curve $\gamma:I\longrightarrow X$. The metric speed $v_{\gamma}(t)$ of $\gamma$ at $t\in I$ is defined as 
\[
v_{\gamma}(t)\coloneqq \lim_{\varepsilon\to0}\frac {d\big(\gamma(t),\gamma(t+\varepsilon)\big)}{\vert\varepsilon\vert},
\]
if it exists.
\end{definition}

\subsection{Lorentzian pre-length spaces}\label{lorentzianprelengthspacesubsection}
The fundamental concept in synthetic Lorentzian geometry is that of a Lorentzian pre-length space, which we present next.
They are the analogues of metric spaces, while Lorentzian length spaces are the analogues of metric length spaces. These notions have been introduced in \cite{Kunzinger_2018} (following earlier works of Busemann \cite{Bus:67} and Kronheimer--Penrose \cite{KP:67}). Nowadays, there are different variants of the basic axiomatization of these spaces, cf.\ \cite{McC:24, BMcC:23, BBCGMORS:24, MS:25} (and other approaches like e.g.\ bounded Lorentzian metric spaces \cite{MS:24, BMS:25}). We opted for sticking to the original setting of \cite{Kunzinger_2018} as, in particular, the choice of a background metric $d$ inducing the topology does not matter except for the causality condition of non-total imprisoning (see \Cref{subsec-cc}) and the Lorentzian Hausdorff measures (see \Cref{Lorentzian Hausdorff measure}).

\begin{definition}[Lorentzian pre-length space]\label{definition:LpLS}
    Let $X$ be a set and $\ll,\leq$ two transitive relations on $X$ with $\leq$ reflexive and $\ll$ contained in $\leq$. Let $d$ be a metric on $X$ and $\tau:X\times X\longrightarrow[0,+\infty]$ be a map satisfying the following properties:
    \begin{itemize}\itemsep0cm
        \item[$(i)$] $\tau$ is lower semicontinuous with respect to the metric topology induced by $d$;
        \item[$(ii)$] $\tau$ satisfies the reverse triangle inequality: $\quad \tau(x,z)\geq \tau(x,y)+\tau(y,z)\quad \forall x\leq y\leq z$;
        \item[$(iii)$] $\tau(x,y)=0$ if $x\not\leq y$, and $\tau(x,y)>0$ if and only if $x\ll y$. 
    \end{itemize}
    \(\tau\) is called {\em time separation function} and $\Xll$ is called {\em Lorentzian pre-length space}.
\end{definition}
Any Lorentzian pre-length space fulfils the so-called \emph{push-up property}, that is, for any $x,y,z\in X$ with $x\leq y\ll z$ or $x\ll y\leq z$ it holds that $x\ll y$. Moreover, for every $x\in X$ either $\tau(x,x)=0$ or $\tau(x,x)=+\infty$, and if $\tau(x,y)\in(0,+\infty)$ then $\tau(y,x)=0$. The causal structure of $\Xll$ allows us to define the timelike and causal future/past of a given point as follows.
\begin{definition}[Timelike/causal future of a point]
    \label{definition:future-past-of-points}
        Let $\Xll$ be a Lorentzian pre-length space. For any $x,y\in X$ we define the sets
        \begin{enumerate}[label=(\roman*)]\itemsep0cm
            \item $I^{+}(x):=\Big\{z\in X \, \colon x\ll z\Big\}, \, \, \, I^{-}(x):=\Big\{z\in X \, \colon z\ll x\Big\}$;
            \item $J^{+}(x):=\Big\{z\in X \, \colon x\leq z\Big\}, \, \,  \,  J^{-}(x):=\Big\{z\in X \, \colon z\leq x\Big\}$;
            \item $I(x,y):=I^+(x)\cap I^-(y), \, \, \, J(x,y):=J^+(x)\cap J^-(y)$.
        \end{enumerate}
    We call $I^{+}(x)$ (resp. $I^{-}(x)$) chronological future (resp. past), $J^{+}(x)$ (resp. $J^{-}(x)$) causal future (resp. past), and $I(x,y)$ (resp. $J(x,y)$) timelike (resp. causal) diamond.
\end{definition}
From the causal relations $\ll,\leq$ and the time separation function $\tau$, we can define timelike and causal curves as well as the so-called variational $\tau$-length of a causal curve as follows. 

\begin{definition}[Causal and timelike curves]
\label{definition:causal-timelike-curves}
    Let $\Xll$ be a Lorentzian pre-length space and  $I\subseteq\mathbb{R}$ an interval. A non-constant curve $\gamma:I\longrightarrow X$ is called causal (resp.\ timelike) if $\gamma(t)\leq \gamma(s)$ (resp.\ $\gamma(t)\ll \gamma(s)$) for all $t,s\in I$ such that $t< s$. 
\end{definition}
\begin{definition}[Variation of causal curves, $\tau$-length]
    \label{definition:tau-length}
    Let \((X, \ll, \leq, \tau, d)\) be a Lorentzian pre-length space. For any causal curve \(\gamma \colon [a,b] \to X\) and any partition \(\sigma=\{t_i\}_{i=0}^{m}\) of \([a,b]\), we define the variation with respect to \(\sigma\), denoted by $V_{\sigma}(\gamma)$ as
    \[
        V_{\sigma}(\gamma) \coloneqq \sum_{i=0}^{m-1} \tau(\gamma(t_i), \gamma(t_{i+1})).
    \]
    We then define the \(\tau\)-length of \(\gamma\) by
    \[
        L^\tau(\gamma) \coloneqq \inf \Big\{ V_{\sigma}(\gamma) \ \colon \ \sigma\ \text{is a partition of} \ [a,b] \Big\}.
    \]
\end{definition}
Given a causal curve $\gamma$, 
variations are monotone decreasing with respect to inclusions of partitions, while the $\tau$-length satisfies additivity and reparametrization invariance\footnote{A reparametrization of a causal curve $\gamma:[a,b]\longrightarrow X$ is another causal curve $\eta:[c,d]\longrightarrow X$ such that $\gamma=\eta\circ\phi$, where $\phi:[a,b]\longrightarrow[c,d]$ is a continuous and strictly monotonically increasing map.}. If the time separation function between causally-related points coincides with the supremum of $\tau$-lengths over all possible causal curves and vanishes otherwise, the space $\Xll$ is called intrinsic. The precise definition is as follows.
\begin{definition}[Intrinsicness]
    \label{definition:intrisicness}
    A Lorentzian pre-length space $\Xll$ is intrinsic if
    \[
        \tau(x,y) :=
            \sup \Big\{ \{0\} \cup\left\{ L^{\tau}(\gamma) \ \colon \ \gamma \ \text{is a causal curve from \(x\) to \(y\)}  \right\} \Big\}.
    \]
\end{definition}
In the context of synthetic Lorentzian geometry, it is usually convenient to enforce certain causality conditions, e.g.\ to exclude pathological causal structures or to guarantee the existence of causal curves between causally-related points. In this regard, the following notions will play a role later on.
\begin{definition}[Causal path connectedness]
    \label{definition:path-connectedness}
     A Lorentzian pre-length space $\Xll$ is said to be causally path connected if for all $x,y\in X$ with $x\ll y$ (resp.\ $x< y$) there is a timelike (resp.\ causal) curve from $x$ to $y$. 
\end{definition}

\begin{definition}[Quasi-strong causality]
    \label{defintion:quasi_strong_causality}
    A Lorentzian pre-length space $\Xll$ is said to be quasi-strongly causal 
    if, for all $x\in X$ and for every neighbourhood $U$ of $x$, there is a neighbourhood $V\subseteq U$ of $x$ such that every causal curve $\gamma:[a,b]\longrightarrow X$ with $\gamma(a),\gamma(b)\in V$ verifies $\gamma([a,b])\subseteq U$. 
\end{definition}

\begin{remark}\label{remark:quasi-strong-is-smooth-strong}
    In the literature, a spacetime is typically called strongly causal if it fulfils \Cref{defintion:quasi_strong_causality}, see e.g.\ \cite{Oneill_1983,Wald_1984}. In the context of Lorentzian pre-length spaces, however, strong causality usually requires that the Alexandrov topology coincides with the topology induced by $d$ (i.e., if timelike diamonds form a subbasis). While \Cref{defintion:quasi_strong_causality} agrees with the standard spacetime notion, it is less restrictive than the corresponding condition for Lorentzian pre-length spaces. We therefore use the terminology \emph{quasi-strong causality} to emphasise this point, unless we are considering an actual spacetime, in which case we simply refer to this property as strong causality. 
\end{remark}

Given a spacetime $(M,g)$, i.e.\ a connected, smooth, time-oriented manifold with a Lorentzian metric $g$ of signature $(-,+,...,+)$, one can construct a causal space $(M,\ll,\leq)$ from the standard causal relations \cite{Kunzinger_2018}, namely $x\ll y$ (resp.\ $x\leq y$) if there is a future-directed timelike (resp.\ causal) curve from $x$ to $y$. One then defines the $g$-length of a causal curve $\gamma:[a,b]\longrightarrow M$ by 
\begin{align}
\label{def:g-length}
L_g(\gamma):=\int_{a}^b\sqrt{-g(\dot{\gamma},\dot{\gamma})dt}.
\end{align}
The map $\tau_g:M\times M\longrightarrow[0,+\infty]$ defined by
\begin{align}
    \label{def:time-sep-funct-spacetimes}
    \tau_g(x,y) \coloneqq
        \sup \Big\{\{0\} \cup \big\{ L_g(\gamma) \ \colon \ \gamma \ \text{is a causal curve from \(x\) to \(y\)}  \big\} \Big\}.
\end{align}
defines a time separation function \cite{Oneill_1983}, and, in particular, is lower semicontinuous with respect to the metric $d^h$ associated to any Riemannian metric $h$ on $M$ (which induces the manifold topology). Thus, $(M,d^h,\ll,\leq,\tau)$ is a Lorentzian pre-length space induced by $(M,g)$. Observe that, from the construction of the causal relations, it immediately follows that $(M,d^h,\ll,\leq,\tau)$ is causally path connected. 

\input{lorentzian_definition}

\input{metric_definition}

\section{Outlook}\label{Conclusions}  
     In the smooth setting of spacetimes, the classical result of Beem (see \cite[Thm.\ 8]{Beem_1976}) states that, under certain assumptions, a spacetime (that may or may not be timelike or null geodesically complete) can be conformally transformed into a timelike and null geodesically complete spacetime. This provides a mechanism to construct models of causally geodesically complete spacetimes (devoid of spacetime singularities), while key causal structures, such as the existence of a black hole horizon, are preserved. The resulting black hole spacetimes are called regular black hole spacetimes (see for e.g.\ \cite{Bambi_2016}). Such spacetimes violate one of the conditions of the pattern singularity theorem (see \cite[Thm.\ 6.1]{Senovilla_1998}), for instance the energy condition. The synthetic version of Beem's theorem, which would also be the Lorentzian analogue of \Cref{thm:make_a_space_complete} will extend these ideas to the non-smooth setting of Lorentzian pre-length spaces offering a rigorous framework for analysing completeness and ``singularity-free" structures beyond the smooth category.  

    As also mentioned in the \hyperref[Introduction]{Introduction}, one may seek to extend the conformal and asymptotic analysis of spacetimes to the synthetic setting. In the smooth category, the notions of conformal relation and conformal transformations (see \Cref{definition:conformal relation}) allow one to define a conformal completion, which in turn induces a well-defined conformal boundary $\mathscr{I}$. A spacetime is said to be asymptotically flat if  $\mathscr{I}$ satisfies the following three properties: 1) there exists an open neighbourhood of $\mathscr{I}$ in which Einstein's vacuum equations hold; 
    2) $\mathscr{I}$ can be decomposed into two (not necessarily disjoint) components, $\mathscr{I}^+$ and $\mathscr{I}^-$, such that every inextendible future-directed null geodesic originates at $\mathscr{I}^-$ and terminates at $\mathscr{I}^+$; and 
    3) $\mathscr{I}$ is a null hypersurface. With the recent development of tools from optimal transport --- namely, the optimal transport formulation of the Einstein vacuum equations (see Appendix B in \cite{Mondino_2022}) --- and the synthetic notion of null hypersurfaces (see \cite{Cavalletti_2025}), it seems conceivable to formulate a synthetic version of asymptotic flatness. Such a formulation would provide a framework to study asymptotic structure and causal properties of spacetimes beyond the smooth category, potentially extending classical results to Lorentzian pre-length spaces.

    \section*{Acknowledgements}
    The authors are grateful to Tobias Beran for helpful discussions and Michael Kunzinger for valuable comments on the draft of this paper.
    
    This research was funded in whole or in part by the Austrian Science Fund (FWF) [grants DOI \href{https://www.fwf.ac.at/en/research-radar/10.55776/EFP6}{10.55776/EFP6} and \href{https://www.fwf.ac.at/en/research-radar/10.55776/STA32}{10.55776/STA32}]. For open access purposes, the authors have applied a CC BY public copyright license to any author-accepted manuscript version arising from this submission.
    \printbibliography

\end{document}

%% file: macros.tex
\newcommand{\R}{\mathbb{R}}
\newcommand{\N}{\mathbb{N}}
\newcommand{\abs}[1]{\left\vert #1 \right\vert}
\newcommand{\conftau}{\tau_\Omega}

\newcounter{mnotecount}

\newcommand{\ma}{\measuredangle}

%% file: theorem_enviroments.tex
\newtheorem{theorem}{Theorem}[section]
\newtheorem{proposition}[theorem]{Proposition}
\newtheorem{lemma}[theorem]{Lemma}
\theoremstyle{definition}
\newtheorem{definition}[theorem]{Definition}
\newtheorem{remark}[theorem]{Remark}
\newtheorem{corollary}[theorem]{Corollary}
\newtheorem{example}[theorem]{Example}

%% file: lorentzian_definition.tex
\section{Conformal transformations of Lorentzian pre-length spaces}\label{Conformal transformations of Lorentzian pre-length spaces}

In this section, we introduce a notion of conformal transformation for Lorentzian pre-length spaces. We begin by examining the shortcomings in the earlier approach of \cite{ebrahimi23}, thereby motivating the need for a consistent and well-founded formulation. We then introduce the concepts of conformal variation with respect to partitions and conformal length, and analyse some of its basic properties. We continue by defining a conformal time separation function $\tau_{\Omega}$, and studying its semi-continuity and causal properties. This, in particular, allows us to show that the class of intrinsic and quasi-strongly causal Lorentzian pre-length spaces is closed under conformal transformations. The section concludes with the proof that the conformal relation defined in this setting is an equivalence relation.

\setcounter{subsection}{0}
\subsection{Previous approach to synthetic conformal transformations}\label{section3.1}

As mentioned in the \hyperlink{intro-target}{Introduction}, the earliest attempt to define a notion of conformal transformation within the synthetic Lorentzian framework was presented in \cite{ebrahimi23}. Given a Lorentzian pre-length space $\Xll$, a positive continuous function $\Omega:X\to (0, +\infty)$, and any (future-directed) causal curve $\gamma:I\subseteq\mathbb{R}\to X$, the conformal length functional is defined in \cite{ebrahimi23} as
    \[
        \overline{L}_{\Omega}(\gamma) \coloneqq \inf \Bigg\{ \sum_{i=0}^{m-1} \Omega\big(\gamma(t_i)\big)\Omega\big(\gamma(t_{i+1})\big) \, \tau(\gamma(t_i), \gamma(t_{i+1})) \ \colon \ \sigma:=\{t_i\}_{i=1}^{m} \, \text{is a partition of} \ I \Bigg\}.
    \]
This definition, however, presents at least two important drawbacks, namely that $\overline{L}_{\Omega}$ is \textit{not} additive and that the function $\overline{\tau}_{\Omega}$ given by
\begin{align*}
    \overline{\tau}_{\Omega}(p,q) \coloneqq
        \sup \big\{ \bar{L}_\Omega(\gamma) \ \colon \ \gamma \ \text{is a future-directed causal curve from \(p\) to \(q\)}  \big\}\cup\{0\},
\end{align*}
does \textit{not} satisfy the reverse triangle inequality. 
We illustrate this with the following example of a subset of Minkowski space.

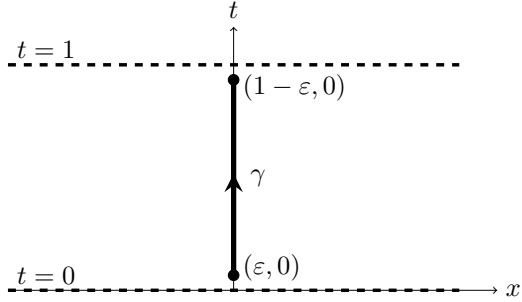
\begin{figure}[ht]
    \centering
\begin{tikzpicture}[scale=1, decoration={markings, 
                            mark= at position 0.52 with {\arrow{stealth}}}]
  \def\xmin{-3}
  \def\xmax{3}
  \def\ymin{0}
  \def\ymax{3}
  \def\ylo{0}
  \def\yhi{3}


  \draw[dashed, line width=1.3pt] (\xmin,\ylo) -- (\xmax,\ylo);
  \draw[dashed, line width=1.3pt] (\xmin,\yhi) -- (\xmax,\yhi);

  \draw[->, line width=0.1pt] (\xmin,0) -- (3.5,0) node[right] {$x$};
  \draw[->, line width=0.1pt] (0,\ymin) -- (0,3.5) node[above] {$t$};

  \node[right] at (-3,3.2) {$t=1$};
  \node[right] at (-3,0.2) {$t=0$};
  \node[right] at (0.1,1.5) {$\gamma$};
  \node[right] at (0,0.3) {$(\varepsilon,0)$};
  \node[right] at (0,2.7) {$(1-\varepsilon,0)$};

  \draw[line width=2pt, black, postaction = decorate] (0,0.2) -- (0,2.8);

  \filldraw[black] (0,0.2) circle (2pt);
  \filldraw[black] (0,2.8) circle (2pt);
\end{tikzpicture}

\caption{The subset $\big(M=(0,1)\times \mathbb{R},g=-dt^2+dx^2\big)$ of Minkowski space, and the timelike curve $\gamma: [\varepsilon,1-\varepsilon] \longrightarrow M$ given by $\gamma(s):= \big(s,0\big)$ 
with $\varepsilon\in(0,1/2)$.}
\label{Figure:Mk:example}
\end{figure}

\begin{example}\label{example:iranian-length} \normalfont
Consider the subset $\big(M=(0,1)\times \mathbb{R},g=-dt^2+dx^2\big)$ of Minkowski space (see \Cref{Figure:Mk:example}), and let $\Omega:M\longrightarrow(0,+\infty)$ be the function defined as $\Omega(t,x):=t(1-t)$. The timelike/causal relations $\ll,\leq$ induced from Minkowski space, together with the euclidean metric $d$ of $\mathbb{R}^2$ and the Minkowskian time separation function $\tau\big((t,x),(t',x')\big)=\sqrt{((t'-t)_+^2-(x'-x)^2)_+}$ (where $(.)_+:=\max(0,.)$ denotes the positive part) yield a Lorentzian pre-length space $(M,d,\ll,\leq,\tau)$.  
For the (future-directed) timelike curve $\gamma: [\varepsilon,1-\varepsilon] \longrightarrow M$ given by $\gamma(s):= \big(s,0\big)$ with $\varepsilon\in(0,1/2)$, the definition of $\overline{L}_{\Omega}$ as an infimum entails
\[
    \overline{L}_{\Omega}(\gamma)\leq \Omega\big(\gamma(\varepsilon)\big)\Omega\big(\gamma(1-\varepsilon))\big) \, \tau\big(\gamma(\varepsilon), \gamma(1-\varepsilon)\big)=\varepsilon^2(1-\varepsilon)^2(1-2\varepsilon).
\]

Moreover, we have that $\min_{s\in[\varepsilon,1-\varepsilon]} (\Omega\circ\gamma)(s)=\Omega(\varepsilon) =\varepsilon(1-\varepsilon)$, so $\varepsilon^2(1-\varepsilon)^2(1-2\varepsilon)$ is also a lower bound for the conformal length of $\gamma$. Thus $\overline{L}_{\Omega}(\gamma)=\varepsilon^2(1-\varepsilon)^2(1-2\varepsilon)$, which means that $\overline{L}_{\Omega}(\gamma\vert_{[\varepsilon,1-\varepsilon]})$ vanishes in the limit $\varepsilon\to 0$. Hence, $\overline{L}_{\Omega}$ is not additive as extending \(\gamma\) decreases its length. 

To prove that $\overline\tau_{\Omega}$ does not satisfy the reverse triangle inequality, first notice that any causal curve $\beta$ with same endpoints as $\gamma$ satisfies $\overline{L}_{\Omega}(\beta)\leq \Omega(\varepsilon)\Omega(1-\varepsilon)(1-2\varepsilon)=\overline{L}_{\Omega}(\gamma)$, hence $\gamma$ is $\overline{L}_{\Omega}$-maximizing. Now, suppose that $\overline\tau_{\Omega}$ verifies the reverse triangle inequality and for a fixed \(\delta \in (0,1/2)\) consider any \(\varepsilon < \delta\). Then,
\begin{align*}    
    \varepsilon^2(1-\varepsilon)^2(1-2\varepsilon)&=\overline\tau_{\Omega}\big((\varepsilon,0),(1-\varepsilon,0)\big)\\
    &\geq  \overline\tau_{\Omega}\big((\varepsilon,0),(\delta,0)\big)
    + \overline\tau_{\Omega}\big((\delta,0),(1-\delta,0)\big)
    + \overline\tau_{\Omega}\big((1-\delta,0),(1-\varepsilon,0)\big)\\
    &\geq \overline\tau_{\Omega}\big((\delta,0),(1-\delta,0)\big)=\overline{L}_{\Omega}\big(\gamma\vert_{[\delta,1-\delta]}\big)=\delta^2(1-\delta)^2(1-2\delta).
    \end{align*} 
In the limit $\varepsilon\to0$ we obtain $0\geq \delta^2(1-\delta)^2(1-2\delta)$, which is a contradiction. Thus, $\overline\tau_{\Omega}$ does not satisfy the reverse triangle inequality.
\end{example}

While the approach in \cite{ebrahimi23} might seem not completely unreasonable as a first guess, \Cref{example:iranian-length} highlights the need for a notion of \emph{conformal length functional} with more suitable properties, additivity and compatibility with the smooth case in particular. This issue is addressed in the next section.

\subsection{Lorentzian conformal length: definition and properties}\label{Lorentzian conformal length: definition and properties}

The definition of \textit{conformal length} that we propose next is analogous to that of the standard length in a Lorentzian pre-length space. 
In particular, it is a variational length, defined as the infimum of variations of a given curve over all partitions. The corresponding notion of \textit{conformal variation} with respect to a partition will therefore be analysed in order to verify that it is sufficiently well-behaved so as to ensure desirable properties of the conformal length (e.g.\ additivity).

\begin{definition}[Conformal factor, variation, length]\label{definition:conformal-factor-variation-length} 
    Let \(\Xll\) be a Lorentzian pre-length space and \(\Omega \colon X \to (0,+\infty)\) a positive continuous function. For any causal curve \(\gamma \colon [a,b] \to X\) and any partition \(\sigma\) of \([a,b]\), we define the {\em conformal variation of \(\gamma\) with respect to the conformal factor \(\Omega\)} as
    \[
        V_{\Omega,\sigma}(\gamma) \coloneqq \sum_{i=0}^{m-1} \max_{t \in [t_i, t_{i+1}]} \Omega(\gamma(t)) \, \tau(\gamma(t_i), \gamma(t_{i+1})).
    \]
    We then define the {\em conformal length of \(\gamma\) with respect to the conformal factor \(\Omega\)} as
    \[
        L_\Omega^\tau(\gamma) \coloneqq \inf \Big\{ V_{\Omega,\sigma}(\gamma) \ \colon \ \sigma\ \text{is a partition of} \ [a,b] \Big\}.
    \]
\end{definition}
The previous definition of $L_{\Omega}^{\tau}$ is motivated by the usual expression of the conformal length of a curve in a smooth spacetime as an integral; it is basically a discretisation of such integral. In fact, for strongly causal spacetimes our notion agrees with the classical notion of conformal length, as we will see in \Cref{Connection with smooth spacetimes}. The presence of the maximum function in the definition is necessary for natural properties of the conformal variations to hold, such as monotone decrease with respect to inclusion of partitions. This property is proven in the following lemma.
\begin{lemma}\label{lemma:variations_decreasing}
    If \(\sigma\) and \(\eta\) are partitions of \([a,b]\) such that \(\sigma \subset \eta\), then \(V_{\Omega,\sigma}(\gamma) \geq V_{\Omega, \eta}(\gamma)\). As a consequence, if we fix any partition \(\sigma\) of \([a,b]\), we have that
    \[
        L_\Omega^\tau(\gamma) = \inf\Big\{V_{\Omega, \eta}(\gamma)\ \colon \ \eta\ \text{is a partition of} \ [a,b], \ \sigma \subset \eta \Big\}.
    \]
\end{lemma}
\begin{proof}
    It suffices to prove the result for the case \(\eta = \sigma \cup \{s\}\), for some \(s \in[a,b] \setminus \sigma\). Indeed, if \(\eta\) had more than one extra point than \(\sigma\), as partitions are finite we can just iteratively apply the previous case until we include all points of \(\eta\). As \(\sigma = \{t_i\}_{i=0}^m\) is a partition of \([a,b]\) and \(s \not \in \sigma\), there exists a unique \(i \in \{0, \dots, m-1\}\) such that \(s \in (t_i, t_{i+1})\). We then see that by the reverse triangular inequality of \(\tau\) we get
    \[
        \begin{split}
            \max_{t \in [t_i, t_{i+1}]} \Omega(\gamma(t)) \, \tau(\gamma(t_i), \gamma(t_{i+1}))&\geq  
            \max_{t \in [t_i, t_{i+1}]}\Omega(\gamma(t)) \,\tau(\gamma(t_i), \gamma(s)) + \max_{t \in [t_i, t_{i+1}]}\Omega(\gamma(t)) \,\tau(\gamma(s), \gamma(t_{i+1}))\\
            &\geq \max_{t \in [t_i, s]}\Omega(\gamma(t)) \,\tau(\gamma(t_i), \gamma(s)) + \max_{t \in [s, t_{i+1}]}\Omega(\gamma(t)) \,\tau(\gamma(s), \gamma(t_{i+1})).
        \end{split}
    \]
    Since all of the other terms required to compute the variations \(V_{\Omega, \sigma}\) and \(V_{\Omega, \eta}\) are the same, we immediately get that \( V_{\Omega, \sigma}(\gamma) \geq V_{\Omega, \eta}(\gamma)\) which proves the first claim.
    
    The second claim immediately follows from noticing that the inequality \say{\(\leq\)} holds by definition and the inequality \say{\(\geq \)} is implied by the fact that for any partition \(\theta\), we can always consider the partition \(\theta \cup \sigma\) which will contain \(\sigma\) and will give a smaller variation by the first claim.
\end{proof}

Later on we will use the notion of conformal length to construct a conformal time separation function $\tau_{\Omega}$ on an intrinsic and quasi-strongly causal Lorentzian pre-length space $\Xll$ (see \Cref{lorentzianprelengthspacesubsection}). This requires that $L_\Omega^\tau $ satisfies a number of properties that we prove in the next proposition.
\begin{proposition}\label{prop:basic_properties_conformal_length}
    The functional \(L_\Omega^\tau \) satisfies the following properties:
    \begin{enumerate}[label=(\roman*)]
        \item For any positive number \(\delta >0\), it holds that
            \[
                L_\Omega^\tau (\gamma) = \inf \left\{ V_{\Omega,\sigma}(\gamma) \ \colon \ \sigma\ \text{is a partition of} \ [a,b] \ \land \ \abs{\sigma} < \delta \right\},
            \]
            where for a partition \(\sigma = \{t_i\}_{i=0}^m\) we define its modulus as
            \(
                \abs{\sigma} \coloneqq \max_{i=0, \dots, m-1} \abs{t_{i+1} - t_i};
            \)
        \item \label{item:additivity} \(L_\Omega^\tau \) is additive, i.e.\ for any causal curve \(\gamma \colon [a,b] \to X\) and any \(s \in [a,b]\), it holds
            \[
             L_\Omega^\tau (\gamma) = L_\Omega^\tau (\gamma\vert_{[a,s]}) + L_\Omega^\tau (\gamma\vert_{ [s,b]});
            \]
        \item \( L_\Omega^\tau \) is invariant under reparametrizations of causal curves, namely if \(\gamma \colon [a,b] \to X\) is a causal curve and \(\sigma \colon [c,d] \to [a,b]\) is an increasing Lipschitz homeomorphism (so that \(\gamma \circ \sigma\) is still causal), then \(L_\Omega^\tau (\gamma) = L_\Omega^\tau (\gamma \circ \sigma)\);
    \end{enumerate}
\end{proposition}
\begin{proof}
    We prove each point separately:
    \begin{enumerate}[label=(\roman*)]
        \item By definition, \(L_\Omega^\tau \) is equal to the infimum over all partitions of \([a,b]\) of the variations \(V_{\Omega, \sigma}\). In particular, this includes as well partitions which have modulus less than \(\delta\), hence the inequality \say{\(\leq\)} is trivial. About the other inequality, for any \(\varepsilon > 0\) we can consider a partition \(\sigma = \{t_i\}_{i=0}^m\)  such that \(V_{\Omega,\sigma}(\gamma) \leq L_\Omega^\tau (\gamma) + \varepsilon\). We split each interval \([t_i, t_{i+1}]\) of the partition into subintervals which have modulus less than \(\delta\) as follows: let \(M\in\mathbb{N}\) be such that \(\abs{\sigma}/M < \delta\). For each \(i \in \{0, \dots, m-1\}\) we define the points
        \[
            t_j^i \coloneqq t_i + \frac{t_{i+1} - t_i}{M}\cdot j, \quad \text{for} \ j \in \{0, \dots, M\}.
        \]
        The collection of points \(\eta = \{t_j^i\}_{i=0,j=0}^{m-1,M}\) is again a partition of \([a,b]\). Moreover, \(\abs{\eta} < \delta\): indeed, it is clear that for any \(i \in \{0, \dots,m-1\}\) and any \(j \in \{0, \dots, M-1\} \) we have that 
        \[
            \abs{t^i_{j+1} - t_j^i} = \frac{t_{i+1}-t_i}{M} \leq \frac{\abs{\sigma}}{M} < \delta.
        \]
        If instead we look at any point of the form \(t_M^i\), for \(i \in \{0, \dots, m-1\}\), we know that \(t_M^i = t_0^{i+1}\), which is either equal to \(b\) so there is nothing to be proved, or we can apply the previous reasoning. The partition \(\eta\) contains \(\sigma\) by construction, so by \Cref{lemma:variations_decreasing} we immediately get that 
        \[
            V_{\Omega, \eta}(\gamma) \leq V_{\Omega,\sigma}(\gamma) \leq L_\Omega^\tau (\gamma) + \varepsilon.
        \]
        As we have proven that for any \(\varepsilon > 0\) there is a partition \(\eta\) with modulus less than \(\delta\) such that the variation is less than \( L_\Omega^\tau (\gamma) + \varepsilon\), we have that the inequality \say{\(\geq\)} holds and thus the two infimum are equal.
        \item Again, we prove that both the inequalities \say{\(\leq\)} and \say{\(\geq\)} hold. For \say{\( \leq \)}, take any two partitions \(\sigma\) and \(\eta\) of \([a,s]\) and \([s,b]\) respectively. Clearly, the set \(\theta \coloneqq \sigma \cup \eta\) is a partition of the full interval \([a,b]\). From the definition it is immediate to see that
        \[
            L_\Omega^\tau (\gamma) \leq V_{\Omega, \theta}(\gamma) = V_{\Omega, \sigma}(\gamma\vert_{ [a,s]}) + V_{\Omega, \eta}(\gamma\vert_{ [s,b]}).
        \]
        Taking the infimum over all \(\sigma\) and \(\eta\) yields the inequality \say{\(\leq\)}. For the opposite one, we start with the partition \(\sigma = \{a,s,b\}\) and consider any partition \(\theta\) of the interval \([a,b]\) containing \(\sigma\). We then define \(\theta_{\leq s}\) and \(\theta_{\geq s}\) to be the set of points in \(\theta\) which are less than or equal to \(s\) and greater or equal to \(s\) respectively. The two sets \(\theta_{\leq s}\) and \(\theta_{\geq s}\) are then partitions of \([a,s]\) and \([s,b]\); from the definition it is immediate to see that
        \[
            V_{\Omega, \theta}(\gamma) = V_{\Omega, \theta_{\leq s}}(\gamma\vert_{ [a,s]}) + V_{\Omega, \theta_{\geq s}}(\gamma\vert_{ [s,b]}) \geq L_\Omega^\tau (\gamma\vert_{ [a,s]}) + L_\Omega^\tau (\gamma\vert_{ [s,b]}).
        \]
        Since this holds for any \(\theta\) containing \(\sigma\), by \Cref{lemma:variations_decreasing} taking the infimum over all such \(\theta\) yields the inequality \say{\(\geq\)}, which concludes the proof.
        \item This last point is immediate, since if \(\phi \colon [c,d] \to [a,b]\) is an increasing Lipschitz homeomorphism, any partition \(\sigma\) of \([a,b]\) yields a partition \(\phi^{-1}(\sigma)\) of \([c,d]\) with the same variation, and for any partition \(\eta\) of \([c,d]\) we get the natural partition \(\phi(\eta)\) of \([a,b]\), which again has the same variation.
    \end{enumerate}
\end{proof}
\begin{remark}
    Given a Lorentzian pre-length space $\Xll$ and a conformal factor \(\Omega\), in general we cannot expect \(L_\Omega^\tau \) to satisfy more properties than the standard Lorentzian length functional $L^{\tau}$ induced by the time separation function $\tau$ (see \Cref{definition:tau-length}). Indeed, from \Cref{definition:conformal-factor-variation-length} it follows that the choice \(\Omega \equiv 1\) yields $L_\Omega^\tau =L^{\tau}$, therefore any property of $L_\Omega^\tau $ must also hold for $L^\tau$, unless we introduce additional assumptions on the conformal factor \(\Omega\) or on the space itself.
\end{remark}

We now show that the conformal length of causal curves is positive if and only if their $\tau$-length is positive, and that the length-maximizing property is preserved under conformal transformations for curves of zero length. This result will be key later to prove that the causal character of curves is invariant under conformal transformations.

\begin{proposition}\label{prop:length_causal_curves} 
    Let \( \Xll\) be a Lorentzian pre-length space, \( \Omega:X\rightarrow(0,+\infty) \) a conformal factor, and \(\gamma:I\rightarrow X\) a causal curve. Then, 
    \begin{enumerate}[label=(\roman*)]
        \item \label{item:length_are_both_positive}\(L_\Omega^\tau (\gamma) > 0 \iff L^\tau(\gamma) > 0\);
        \item \label{item:length_are_both_null}if \(L^\tau(\gamma) = 0\) and $\gamma$ is length maximizing between its endpoints, then it is also length maximizing with respect to \(L_\Omega^\tau \).
    \end{enumerate}
\end{proposition}
\begin{proof}
    If we denote by \(m\) and \(M\) the minimum and the maximum respectively of the function $\Omega\circ \gamma$, it is immediate to see from the definition that
    \[
        m L^\tau(\gamma) \leq L_\Omega^\tau(\gamma) \leq M L^\tau(\gamma).
    \]
    From these previous inequalities we get that \( L_\Omega^\tau (\gamma) > 0\) if and only if \( L^\tau(\gamma) > 0\), as both \(m\) and \(M\) are positive, which proves the first item. If we now assume that $\gamma$ is length maximizing with respect to \(L^\tau\) and of zero length, it must also be length maximizing with respect to \(L_\Omega^\tau \), otherwise there would be a causal curve $\alpha$ joining the endpoints of $\gamma$ with greater \(L_\Omega^\tau \)-length, and this would imply that $L_\Omega^\tau (\alpha) > 0$. By the previous item, we know then that \(L^\tau(\alpha) > 0\), which then contradicts \(\gamma\) being \(L^\tau\)-maximal.
\end{proof}

\subsection{A class of spaces closed under conformal transformations}\label{A class of spaces closed under conformal transformations}

Any reasonable definition of conformal transformation must satisfy the property that the class of spaces under consideration is closed under conformal change, i.e., that conformally-transformed spaces remain in the same class as the original one. This requires, in particular, providing a notion of a \say{conformal time separation function}, and showing that it indeed fulfils the properties of a time separation function in the conformally-transformed space. As we will see in \Cref{Conformal transformations as an equivalence relation}, closedness of conformal transformations holds 
for the class of intrinsic, quasi-strongly causal Lorentzian pre-length spaces (cf. \Cref{defintion:quasi_strong_causality}). The motivation for assuming intrinsicness is due to \Cref{definition:conformal-factor-variation-length} strongly relying on causal curves, which makes it necessary to relate them with time separation functions in order to capture properties of the latter. Moreover, this property turns out to be necessary for causality relations (and thus for quasi-strong causality) to be preserved under conformal transformations (cf.\ \Cref{prop:character_causal_curves}). Quasi-strong causality, on the other hand, is key to demonstrate that $\tau_{\Omega}$ is a time separation function (or, more precisely, that it is lower semi-continuous, cf.\ \Cref{proposition:tau-lower-semi-continuous}), and that the conformal length $L_\Omega^\tau$ and the length induced from $\tau_{\Omega}$ actually coincide (cf.\ \Cref{lemma:strongly_causal_implies_intrinsic}).

Let us introduce a notion of conformal time separation function $\tau_{\Omega}$ given by the supremum of the conformal length over all causal curves between any two points. We define the map $\tau_{\Omega}: X\times X\rightarrow\mathbb{R}$ as
\[
    \conftau(x,y) \coloneqq 
        \sup \Big \{ \{0\} \cup \big\{ L_\Omega^\tau (\gamma) \ \colon \ \gamma \ \text{is a causal curve from \(x\) to \(y\)}  \big\} \Big\}. 
\]
Observe that, in general, \(\tau_\Omega\) is \textit{not} lower semi-continuous, hence not a time separation function in the sense of \Cref{definition:LpLS}. However, we can construct a causal structure from \(\tau_\Omega\) by defining the causal relations
\[
    x \ll_\Omega  y\ \ :\Longleftrightarrow \ \conftau(x,y) > 0  
    \quad \text{and} \quad 
    x \leq_{\Omega} y \ :\Longleftrightarrow \ x \leq y. 
\]
Remarkably, conformal transformations preserve the causal structure; moreover, $\tau_{\Omega}$ verifies the reverse triangle inequality and is intrinsic, meaning that it is given by the supremum over the length functional it induces. These results are proven in the following three propositions.

\begin{proposition}\label{prop:character_causal_curves} 
    Let $\Xll$ be an intrinsic Lorentzian pre-length space, \( \Omega:X\rightarrow(0,+\infty) \) a conformal factor, and \( x,y\in X \) be any two points. Then, 
    \[
        \tau_{\Omega}(x,y) > 0 \quad \iff \quad \tau(x,y)>0 \ ; \quad\qquad 
        x \leq_\Omega y, \ x \not\ll_\Omega  y \quad \iff \quad x \leq y, \ x \not \ll y.
    \]
\end{proposition}
\begin{proof}
    
    Taking the supremum over all causal curves with endpoints $x$ and $y$ on item \textit{\ref{item:length_are_both_positive}} of \Cref{prop:length_causal_curves} yields that $\tau(x,y)>0$ if and only if $\tau_{\Omega}(x,y)>0$. This, together with the definitions of $\ll_\Omega ,\leq_{\Omega}$, immediately proves the second claim.
\end{proof} 

\begin{proposition}\label{prop:reverse_triangle_inequality}
    Let $\Xll$ be a Lorentzian pre-length space and $\Omega$ a conformal factor. Then \(\conftau\) satisfies the reverse triangle inequality, namely
    \[
        \conftau(x,z) \geq \conftau(x,y) + \conftau(y,z),\qquad x,y,z\in X, \quad  x \leq_\Omega y \leq_\Omega z.
    \]
\end{proposition}
\begin{proof}
    For any \(\varepsilon>0\), we can consider causal curves \(\gamma\) from \(x\) to \(y\) and \(\alpha\) from \(y\) to \(z\) such that
    \[
        \conftau(x,y)\leq L_\Omega^\tau (\gamma) + \varepsilon \ \land \ \conftau(y,z)\leq L_\Omega^\tau (\alpha) + \varepsilon.
    \]
    Concatenating the two curves yields a causal curve \(\gamma\alpha\) from \(x\) to \(z\). By item \textit{\ref{item:additivity}} of \Cref{prop:basic_properties_conformal_length}, we get 
    \[
        \conftau(x,z) \geq L_\Omega^\tau (\gamma\alpha) = L_\Omega^\tau (\gamma) + L_\Omega^\tau (\alpha) \geq \conftau(x,y) + \conftau(y,z) - 2\varepsilon.
    \]
    Since this holds for any \(\varepsilon>0\), we get the desired result.
\end{proof}

\begin{proposition}\label{prop:tau_omega_intrinsic}
    Let $\Xll$ be a Lorentzian pre-length space and $\Omega$ a conformal factor, so that we may define \(\conftau\). For any causal curve \(\gamma \colon [a,b] \to X\) the length is given by
    \[
        L^{\tau_\Omega}(\gamma) = \inf\left\{ \sum_{i=0}^{m-1} \conftau(\gamma(t_i), \gamma(t_{i+1})) \ \colon \ \left\{t_i\right\}_{i=0}^{m} \ \text{partition of} \ [a,b] \right\}.
    \]
    Consider then the map \(\hat{\tau}_\Omega\) given by
    \[
        \hat{\tau}_\Omega(x,y) \coloneqq 
        \sup \Big \{ \{0\} \cup \{ L^{\conftau} (\gamma) \ \colon \ \gamma \ \text{is a causal curve from \(x\) to \(y\)} \}  \Big \},
    \]
    then \(\conftau = \hat{\tau}_\Omega\).
\end{proposition}
\begin{proof}
    If \(p \not \leq q\) there is nothing to do. Let us thus assume that \(p \leq q\) and  consider any causal curve \(\gamma \colon [a,b] \to X\) joining them. For any partition of \([a,b]\) we see that
    \[
        \sum_{i=0}^{m-1}L_{\Omega}^\tau(\gamma\vert_{ [t_i, t_{i+1}]}) \leq \sum_{i=0}^{m-1} \tau_\Omega (\gamma(t_i), \gamma(t_{i+1})) \leq \tau_\Omega(p,q),
    \]
    where the left inequality follows from the definition of \(\tau_\Omega\) as a supremum of \(L_\Omega^\tau\)-lengths and the right inequality from the reverse triangle inequality shown in \Cref{prop:reverse_triangle_inequality}. We apply additivity of the \(L_{\Omega}^\tau\)-length and get that
    \[
        L_\Omega^\tau(\gamma) \leq \sum_{i=0}^{m-1} \tau_\Omega (\gamma(t_i), \gamma(t_{i+1})) \leq \tau_\Omega(p,q).
    \]
    Taking the infimum over all partitions of \([a,b]\) we have
    \[
        L_\Omega^\tau(\gamma) \leq L^{\tau_\Omega}(\gamma) \leq \tau_\Omega(p,q),
    \]
    and then taking the supremum over all such curves \(\gamma\) we finally get
    \[
        \tau_\Omega(p,q) \leq \hat{\tau}_\Omega(p,q) \leq \tau_\Omega(p,q),
    \]
    which proves that \(\tau_\Omega(p,q) = \hat{\tau}_\Omega(p,q)\).
\end{proof}

Quasi-strong causality of $\Xll$ means that, for any point \(x \in X\) and any of its neighbourhoods \(U\), we can find a smaller neighbourhood \(V \subset U\) such that any causal curve with endpoints in \(V\) lies in \(U\) (cf.\ \Cref{defintion:quasi_strong_causality}). In the next lemma, we show that under the assumption of quasi-strong causality one can establish some local upper and lower bounds on the conformal length $L_\Omega^\tau$ of causal curves between close enough points. If, in addition, $\Xll$ is intrinsic, then local boundedness of $L_\Omega^\tau $ by \(L^\tau\) translates into local boundedness of $\tau_{\Omega}$ by \(\tau\). This in turn allows us to demonstrate that the conformal factor at a fixed point $p\in X$ is given by the quotient $\frac{\tau_{\Omega}(p,p_n)}{\tau(p,p_n)}$ as $n\to\infty$, where $p_n$ is a suitable sequence of points converging to $p$. 

\begin{lemma}\label{lemma:local_uniform_control_conformal_length}
    Consider an intrinsic, quasi-strongly causal Lorentzian pre-length space $\Xll$, a conformal factor \(\Omega\) and a point \(p \in X\).  
    Let \(\{p_n\}_n\) be a sequence in \(J^+(p)\) (or in \(J^-(p)\)) converging to \(p\).
    Then, for any \(\varepsilon>0\) there exists \(N \in \N\) such that for every \(n \geq N\) and every causal curve \(\gamma_n\) joining \(p\) and \(p_n\), the following inequalities hold
    \begin{align}
        \label{eq:local_uniform_control_conformal_length}
        (\Omega(p) -\varepsilon)L^{\tau}(\gamma_n) &\leq L_\Omega^\tau (\gamma_n) \leq (\Omega(p) +\varepsilon)L^{\tau}(\gamma_n),\\
        \label{equation:tau_ratio_equals_omega}
        (\Omega(p) -\varepsilon)\tau(p, p_n)&\leq \tau_\Omega(p, p_n) \leq (\Omega(p) +\varepsilon)\tau(p, p_n).
    \end{align}
    In particular, if we suppose that eventually in \(n\) we have 
    \(
        \tau(p,p_n) \in (0, +\infty),
    \)
    we get
    \[
        \lim_{n \to \infty} \frac{\tau_\Omega(p,p_n)}{\tau(p,p_n)} = \Omega(p).
    \]
    We get the analogous conclusion for the case $p_n\in J^-(p)$ for all $n\in\N$.
\end{lemma}
\begin{proof}
    We just consider the case where the sequence \(\{p_n\}_n\) is contained in \(J^+(p)\) --- the other case being completely analogous.
    
    Using continuity of \(\Omega\) at \(p\) we can find an open ball \(U\) around \(p\) such that the oscillation of \(\Omega\) is bounded by \(\varepsilon\) on \(U\). We then employ quasi-strong causality to find a smaller neighbourhood \(V \subset U\) such that any causal curve with endpoints in \(V\) stays in \(U\). Then there exists an \(N \in \N\) such that for every \(n\geq N\) we have \(p_n \in V\). If we now take a curve \(\gamma_n\) from \(p\) to \(p_n\), we see that for any partition \(\sigma\) of its domain, we have that
    \[
        (\Omega(p) -\varepsilon)V_{\sigma}(\gamma_n) \leq V_{\Omega,\sigma}(\gamma_n) \leq (\Omega(p) +\varepsilon)V_{\sigma}(\gamma_n) ,
    \]
    since the terms featuring \(\Omega\) can be bounded above and below by \(\Omega(x)\) up to an error \(\varepsilon\), as \(\gamma_n\) is entirely contained in \(U\) by quasi-strong causality. Taking the infimum over all partitions \(\sigma\) then yields \eqref{eq:local_uniform_control_conformal_length}.
    Now, for any \(\varepsilon>0\) eventually in \(n\) we have that for any causal curve \(\gamma\) between \(p\) and \(p_n\) it holds that
    \[
        (\Omega(p) -\varepsilon)L^{\tau}(\gamma) \leq L_\Omega^\tau (\gamma) \leq (\Omega(p) +\varepsilon)L^{\tau}(\gamma).
    \]
    Taking the supremum over all such curves \(\gamma\) then yields \eqref{equation:tau_ratio_equals_omega}.
    
    If we now consider the additional assumption that \(\tau(p,p_n)\in(0,\infty)\), we can divide the previous inequalities by \(\tau(p,p_n)\) getting
    \[
        (\Omega(p) -\varepsilon)\leq \frac{\tau_\Omega(p, p_n)}{\tau(p, p_n)} \leq (\Omega(p) +\varepsilon)\qquad\Longleftrightarrow\qquad 
        \abs{\frac{\tau_\Omega(p, p_n)}{\tau(p, p_n)} - \Omega(p)} \leq \varepsilon.
    \]
    Since this inequality holds for any \(\varepsilon>0\) and for \(n\) sufficiently large, this proves the last claim.
\end{proof}

Having the upper and lower bounds of \Cref{lemma:local_uniform_control_conformal_length} at our disposal, we can now prove that intrinsicness and quasi-strong causality of $\Xll$ imply lower semicontinuity of $\tau_{\Omega}$.

\begin{proposition}\label{proposition:tau-lower-semi-continuous}
    Suppose \((X,\ll\,\leq,\tau,d)\) is an intrinsic and quasi-strongly causal Lorentzian length space and fix any conformal factor \(\Omega\). Then  \(\tau_\Omega\) is lower semicontinuous.
\end{proposition}
\begin{proof}
    Let \(p,q \in X\) be any two points. If \(\tau(p,q)=0\) (which happens if and only if $\tau_\Omega(p,q)=0$ by \Cref{prop:character_causal_curves}), then lower semicontinuity follows immediately. We therefore assume \(\tau(p,q) > 0\), and consider two arbitrary sequences of points \(p_n\) and \(q_n\) converging to \(p\) and \(q\) respectively.

    If \(p = q\), the reverse triangle inequality for $\tau$ implies that \(\tau(p,p) = +\infty\). Then \Cref{prop:length_causal_curves} ensures that \(\tau_{\Omega}(p,p)\neq0\), hence \(\tau_\Omega(p,p) = +\infty \) by the reverse triangle inequality for $\tau_{\Omega}$. We now fix any \(\varepsilon \in (0, \Omega(p))\) and find a neighbourhood \(U\) of \(p\) where the oscillation of \(\Omega\) is bounded by \(\varepsilon\), as well as the corresponding quasi-strong causality neighbourhood \(V\subseteq U\) (cf.\  \Cref{defintion:quasi_strong_causality}). 

    Eventually in \(n\), both sequences \(p_n\) and \(q_n\) are contained in \(V\), and lower semicontinuity of $\tau$ yields
    \[
        \liminf_{n \to +\infty} \tau(p_n,q_n) \geq \tau(p,p) = +\infty,
    \]
    which implies that eventually in \(n\) we have \(p_n \ll q_n\). Therefore, for sufficiently large \(n\) we can find causal curves \(\gamma_n\) between \(p_n\) and \(q_n\) which are contained in \(U\) with \(L^\tau(\gamma_n) \to +\infty\) as \(n \to +\infty\). Estimating the \(L^\tau_\Omega\)-length as in \Cref{lemma:local_uniform_control_conformal_length} we see that
    \[
        \tau_\Omega(p_n,q_n) \geq L_\Omega^\tau(\gamma_n) \geq (\Omega(p) - \varepsilon)L^\tau(\gamma_n).
    \]
    Taking the \(\liminf\) as \(n \to +\infty\) we see that
    \[
        \liminf_{n \to +\infty} \tau_\Omega(p_n,q_n) \geq +\infty = \tau_\Omega(p,p),
    \]
    which proves lower semicontinuity.
    
    Now, if \(p \neq q\), as \(\tau(p,q) > 0\) we get a causal curve \(\gamma\colon[a,b]\rightarrow X\) from \(p\) to \(q\) with positive \(\tau\)-length.

    We define
    \[
        \begin{split}
            t_1 &\coloneqq \sup \{t \in [a,b] \ \colon \ p \not \ll \gamma(t)\}, \\
            t_2 &\coloneqq \inf \{ t \in [a,b] \ \colon \ \gamma(t) \not \ll q\},
        \end{split}
    \]
    and set \(\widetilde{p} \coloneqq \gamma(t_1)\), \(\widetilde{q} \coloneqq \gamma(t_2)\); if the sets are empty, we set \(t_1 = a\) and \(t_2 = b\), so that \(\widetilde{p} = p\) and \(\widetilde{q}=q\) respectively. We now notice the following facts:
    \begin{itemize}
        \item \(L_\Omega^\tau(\gamma\vert_{[t_1,t_2]}) = L^\tau_\Omega(\gamma)\); indeed, this is trivially true if \(t_1 =a\) and \(t_2 = b\). If \(t_1 \neq a\) (or \(t_2 \neq b\)), we see that by continuity of \(\gamma\) and by the fact that the complement of \(I^+(p)\) (\(I^-(q)\), respectively) is a closed set we must have \(p \not \ll \widetilde{p}\) (\(\widetilde{q} \not \ll q\), respectively). In particular then the \(L^\tau\)-length of the curve \(\gamma\vert_{ [a, t_1]}\) (\(\gamma\vert_{ [t_2,b]}\), respectively) vanishes and therefore so does the \(L^\tau_\Omega\)-length, by \Cref{prop:length_causal_curves}. By additivity of the length we see then that the length of \(\gamma\vert_{[t_1,t_2]}\) coincides with the length of \(\gamma\).
        \item \(t_1 < t_2\); this is trivially true when \(t_1 = a\) and \(t_2 = b\) as \(a<b\). If \(t_1 \neq a\) (\(t_2 \neq b\), respectively), assuming by contradiction that \(t_1 \geq t_2\) gives in particular that \(t_2 \neq b\) (\(t_1 \neq a\), respectively): if \(t_1 = b\) (\(t_2 = a\) respectively) then the sup (inf) exists and \(p \not \ll \gamma(b) = q\) (\(p = \gamma(a) \not \ll q\)), since the complement of \(I^+(p)\) (\(I^-(q)\)) is closed; on the other hand we assumed \(p \ll q\), which is a contradiction. As we noticed before in such case we have \(p \not \ll \widetilde{p}\) and \(\widetilde{q} \not \ll q\). We then estimate the length of \(\gamma\) as 
        \[
            L^\tau(\gamma) \leq \tau(p, \widetilde{p}) + \tau(\widetilde{p},q) =0,
        \]
        as the first term vanishes since \(p \not \ll \widetilde{p}\) and the second one vanishes as the assumption that \(t_1 \geq t_2\) implies that \(\gamma(t_1) \not \ll q\) by the push-up property; this contradicts the fact that \(\gamma\) has positive length.
    \end{itemize}
     For \(h \in (0, \frac{t_2-t_1}{2})\) we define \(p^+(h) \coloneqq \gamma(t_1+h)\) and \(q^-(h) \coloneqq \gamma(t_2-h)\). Notice that \(p^+(h) \in I^+(p)\) and \(q^-(h) \in I^-(q)\) since \(t_1+h > t_1\) and \(t_2 -h < t_2\). In particular \(I^-(p^+(h))\) and \(I^+(q^-(h))\) are open neighbourhoods of \(p\) and \(q\) respectively. Therefore for a fixed \(h\) the two sequences are eventually contained in them, i.e., \(p_n \ll p^+(h)\) and \(q_n \ll q^-(h)\) for sufficiently large \(n\). We fix \(\varepsilon>0\) with \(\varepsilon < \min(\Omega(\widetilde{p}), \Omega(\widetilde{q}))\) and consider neighbourhoods at \(\widetilde{p}\) and \(\widetilde{q}\) such that the oscillation of \(\Omega\) is bounded by \(\varepsilon\). We then employ quasi-strong causality to find smaller neighbourhoods \(V_1\) and \(V_2\) which satisfy the quasi-strong causality condition. Lastly, we take \(h \leq h(\varepsilon)\) small enough so that \(p^+(h)\) and \(q^-(h)\) are in \(V_1\) and \(V_2\) respectively.
    
    As \(p^+(h) \leq q^-(h)\), we can apply the reverse triangle inequality to see that
    \[
        \tau_{\Omega}(p_n,q_n) \geq \tau_{\Omega}(p_n, p^+(h)) + \tau_{\Omega}(p^+(h),q^-(h)) + \tau_{\Omega}(q^-(h), q_n) \geq \tau_\Omega(p^+(h), q^-(h)).
    \]
    Notice that such estimate holds whenever \(n\) is big enough, hence taking the \(\liminf\) yields
    \[
        \liminf_{n \to +\infty} \tau_{\Omega}(p_n,q_n) \geq \tau_\Omega(p^+(h), q^-(h)).
    \]
    As \(V_1\) and \(V_2\) are open neighbourhoods of \(p^+(h(\varepsilon))\) and \(q^-(h(\varepsilon))\) respectively, we can find two points \(p_1\) and \(p_2\) along \(\gamma\) so that \(p_i \in V_i\) for \(i=1,2\) and for every \(h \leq h(\varepsilon)\) it holds \(p^+(h) \leq p_1 \leq p_2 \leq q^-(h)\). Applying the triangle inequality then yields
    \[
        \liminf_{n \to +\infty} \tau_{\Omega}(p_n,q_n) \geq \tau_\Omega(p^+(h), q^-(h)) \geq \tau_\Omega(p^+(h), p_1) + \tau_\Omega(p_1, p_2) + \tau_\Omega(p_2, q^-(h)).
    \]
    As \(p^+(h), p_1 \in V_1\) and \(p_2,q^-(h) \in V_2\) every causal curve between each pair is contained in a neighbourhood where the oscillation of \(\Omega\) is bounded by \(\varepsilon\). Analogously to what is done in \Cref{lemma:local_uniform_control_conformal_length}, we can estimate \(\tau_\Omega\) by
    \[
        \tau_\Omega(p^+(h), p_1) + \tau_\Omega(p_1, p_2) + \tau_\Omega(p_2, q^-(h)) \geq (\Omega(\widetilde{p}) - \varepsilon) \tau(p^+(h), p_1) + \tau_\Omega(p_1, p_2) + (\Omega(\widetilde{q}) - \varepsilon)\tau(p_2, q^-(h)).
    \]
    Notice that this estimate holds for any \(h \leq h(\varepsilon)\). We can then employ lower semicontinuity of \(\tau\) sending \(h \to 0\) to see that
    \[
        \liminf_{n \to +\infty}\tau_\Omega(p_n,q_n) \geq\liminf_{h \to 0^+} \tau_\Omega(p^+(h), q^-(h)) \geq (\Omega(\widetilde{p}) - \varepsilon) \tau(\widetilde{p}, p_1) + \tau_\Omega(p_1, p_2) + (\Omega(\widetilde{q}) - \varepsilon) \tau(p_2, \widetilde{q}).
    \]
    We can once again estimate as in \Cref{lemma:local_uniform_control_conformal_length} to get
    \[
        \begin{split}
            \liminf_{n \to +\infty}\tau_\Omega(p_n,q_n) &\geq \frac{\Omega(\widetilde{p}) - \varepsilon}{\Omega(\widetilde{p}) + \varepsilon} \tau(\widetilde{p}, p_1) + \tau_\Omega(p_1, p_2) + \frac{\Omega(\widetilde{q}) - \varepsilon}{\Omega (\widetilde{q}) + \varepsilon}\tau_\Omega(p_2, \widetilde{q})\\
            &\geq \min \left( \frac{\Omega(\widetilde{p}) - \varepsilon}{\Omega(\widetilde{p}) + \varepsilon}, \frac{\Omega(\widetilde{q}) - \varepsilon}{\Omega (\widetilde{q}) + \varepsilon}, 1 \right) (\tau_\Omega(\widetilde{p},p_1) + \tau_\Omega(p_1,p_2) + \tau_\Omega(p_2, \widetilde{q}))
        \end{split}
    \]
    Each term of the sum on the right-hand-side can be estimated from below by the \(L_\Omega^\tau\)-length of \(\gamma\) restricted to the corresponding interval. Using then additivity of the length we get
    \[
        \liminf_{n \to +\infty}\tau_\Omega(p_n,q_n) \geq \min \left( \frac{\Omega(\widetilde{p}) - \varepsilon}{\Omega(\widetilde{p}) + \varepsilon}, \frac{\Omega(\widetilde{q}) - \varepsilon}{\Omega (\widetilde{q}) + \varepsilon}, 1 \right) L_\Omega^\tau(\gamma\vert_{ [t_1,t_2]})
    \]
    Sending \(\varepsilon \to 0^+\) yields
    \[
        \liminf_{n \to +\infty}\tau_\Omega(p_n,q_n) \geq L_\Omega^\tau(\gamma\vert_{ [t_1,t_2]}) = L_\Omega^\tau(\gamma).
    \]
    Notice that this final estimate holds trivially if \(\gamma\) has zero length, so it holds for any causal curve between \(p\) and \(q\). Taking the supremum over all such curves then gives
    \[
        \liminf_{n \to +\infty}\tau_\Omega(p_n,q_n) \geq \tau_\Omega(p,q),
    \]
    which proves lower semicontinuity.
\end{proof}

The previous results establish that \((X,d,\ll,\leq,\tau_{\Omega})\) is a Lorentzian pre-length space, and justify referring to $\tau_{\Omega}$ as \textit{conformal time separation function}. It still remains to prove, however, that the class of intrinsic, quasi-strongly causal Lorenzian pre-length spaces is closed under conformal changes, i.e.\ to check whether the space $\confXll$ is intrinsic and quasi-strongly causal as well. Clearly, both spaces have the same topology and, by \Cref{prop:character_causal_curves}, they also have the same causality relations. Consequently, the assumption of quasi-strong causality of $\Xll$ automatically implies quasi-strong causality of $\confXll$. The intrinsicness of \(\tau_\Omega\), on the other hand, was proven in \Cref{prop:tau_omega_intrinsic}. However, this result can actually be strengthened with the addition of quasi-strong causality by proving that the length induced by \(\tau_\Omega\) coincides with the conformal length \(L_\Omega^\tau\). This is done in the following lemma. 

\begin{lemma}\label{lemma:strongly_causal_implies_intrinsic}
Consider a quasi-strongly causal intrinsic Lorentzian pre-length space $(X,d,\ll,\le,\tau)$ and a conformal factor \(\Omega\). 
For any causal curve $\gamma:[a,b]\to X$, it holds that
\[
    L_\Omega^\tau (\gamma)=L^{\tau_{\Omega}}(\gamma).
\]
In particular, this implies the conclusion of \Cref{prop:tau_omega_intrinsic}.
\end{lemma}
\begin{proof}
We first prove that $L_\Omega^\tau (\gamma)\leq L_{\tau_{\Omega}}(\gamma)$. For any partition \(\sigma\) of \([a,b]\) we see that
\[
    \begin{split}
        \sum_{i=1}^{m-1}\tau_{\Omega}\big(\gamma(t_i),\gamma(t_{i+1})\big)
        &\geq \sum_{i=1}^{m-1}L_\Omega^\tau  \left(\gamma\vert_{[t_i,t_{i+1}]}\right) 
        =L_\Omega^\tau (\gamma),
    \end{split}
\]
as \(\tau_\Omega(\gamma(t_i), \gamma(t_{i+1}))\) is defined as the supremum of the \(L_\Omega^\tau\)-length of curves between the two points and thus can be estimated from below by the \(L_\Omega^\tau\)-length of \(\gamma\vert_{[t_i,t_{i+1}]}\). Taking the infimum over all partitions \(\sigma\) immediately yields $L^{\tau_{\Omega}}(\gamma)\geq L_\Omega^\tau (\gamma)$.

We now prove $L^{\tau_{\Omega}}(\gamma)\leq L_\Omega^\tau (\gamma)$. Consider any positive real number $\varepsilon$. Since $\Omega$ is continuous, for all $t\in[a,b]$ there exists an open ball $U_t = B(\gamma(t),\delta_t)$ centered at $\gamma(t)$ with radius $\delta_t>0$ such that, for all $x\in U_t$, we have $\big\vert\Omega\big(\gamma(t)\big)-\Omega\big(x\big)\big\vert\leq \varepsilon$. By quasi-strong causality, there exists an open set $V_t\subset U_t$ containing $\gamma(t)$, such that any causal curve with endpoints in $V_t$ is contained in $U_t$. This is true for any $t$, so $\gamma$ is entirely contained in $\bigcup_{t\in[a,b]} V_t\subset X$. Moreover, since $\gamma([a,b])$ is compact, there exists a finite collection $\{V_{t_j}\}_{j=1}^{k}, k\in\mathbb{N}$ whose union contains $\gamma$. Lebesgue's number lemma then ensures that there is $\delta>0$ such that 
\begin{align*}
\vert t-t'\vert<\delta\quad\Longrightarrow\quad \gamma\vert_{[t,t']}\subset V_{t_j}\textup{ for some }t_j.
\end{align*}
Now, take any partition $\sigma$ of $\gamma$ with modulus $\vert\sigma\vert<\delta$, and fix some $\widetilde{\varepsilon}>0$. Then for all $i \in \{1,\dots,m-1\}$ there exists a causal curve $\eta_i$ from $\gamma(t_i)$ to $\gamma(t_{i+1})$ such that 
\[
    L_\Omega^\tau (\eta_i)+\widetilde{\varepsilon}\geq \tau_{\Omega}\left(\gamma(t_i),\gamma(t_{i+1})\right).
\] 
The endpoints of $\eta_i$ and $\gamma\vert_{[t_i,t_{i+1}]}$ are the same and are contained in $V_{t_j}$ for some $t_j\in[a,b]$, which (by quasi-strong causality) means $\eta_i\subset U_{t_j}$. Then,
\[
    \begin{split}
        V_{\Omega,\sigma}(\gamma) & = \sum_{i=0}^{m-1} \max_{t \in [t_i, t_{i+1}]} \Omega\left(\gamma(t)\right) \, \tau\big(\gamma(t_i), \gamma(t_{i+1})\big)
        \\
        &
        \geq \sum_{i=0}^{m-1} \max_{t \in [t_i, t_{i+1}]} \Omega\big(\gamma(t)\big) \, L^{\tau}\big(\eta_i\big) \geq \sum_{i=0}^{m-1}   \frac{\max_{t \in [t_i, t_{i+1}]}\Omega\big(\gamma(t)\big) }{\max_{s\in [t_i,t_{i+1}]} \Omega(\eta_i(s)}L_\Omega^\tau (\eta_i),
    \end{split}
\]
where in the last inequality we used the same estimate as in the proof of \Cref{prop:length_causal_curves} to recover \(L_\Omega^\tau \) from \(L^\tau\). We now recall that \(\eta_i\) are \(\widetilde{\varepsilon}\)-geodesics for \(L_\Omega^\tau \), so that
\[
        V_{\Omega,\sigma}(\gamma) \geq \sum_{i=0}^{m-1}   \frac{\max_{t \in [t_i, t_{i+1}]}\Omega\big(\gamma(t)\big) }{\max_{s\in [t_i,t_{i+1}]} \Omega(\eta_i(s))}\Big(\tau_{\Omega}\big(\gamma(t_i),\gamma(t_{i+1})\big)-\widetilde{\varepsilon}\Big).
\]
We then use the fact, that  for each summand, the curves are contained in \(U_{t_j}\) and thus the oscillation of \(\Omega\) is bounded by \(2\varepsilon\); this allows us to replace the maximum term in the denominator with the same one in the numerator up to a \(2\varepsilon\) contribution:
\[
    \begin{split}
        V_{\Omega,\sigma}(\gamma) &\geq \sum_{i=0}^{m-1}   \frac{\max_{t \in [t_i, t_{i+1}]} \Omega \big(\gamma(t)\big) }{\max_{t \in [t_i, t_{i+1}]} \Omega\big(\gamma(t)\big)+2\varepsilon}\Big(\tau_{\Omega}\big(\gamma(t_i),\gamma(t_{i+1})\big)-\widetilde{\varepsilon}\Big)\\
        &= \sum_{i=0}^{m-1} \left( 1 -\frac{2\varepsilon}{\max_{t \in [t_i, t_{i+1}]} \Omega \big(\gamma(t)\big)+2\varepsilon} \right) \Big(\tau_{\Omega}\big(\gamma(t_i),\gamma(t_{i+1})\big)-\widetilde{\varepsilon}\Big).
    \end{split}
\]
We now define $m:=\min_{t\in[a,b]}\Omega\big(\gamma(t)\big)$. Since nothing depends on $\eta_i$ in the inequality above, we can let $\widetilde{\varepsilon} \to0$ and find
\[
    \begin{split}
        V_{\Omega,\sigma}(\gamma) &\geq \sum_{i=0}^{m-1}  \left( 1-\frac{2\varepsilon}{\max_{t \in [t_i, t_{i+1}]}\Omega\big(\gamma(t)\big)+2\varepsilon}\right) \tau_{\Omega}\big(\gamma(t_i),\gamma(t_{i+1})\big)
        \\
        &\geq \left( 1-\frac{2\varepsilon}{m +2\varepsilon}\right) \sum_{i=0}^{m-1}  \tau_{\Omega}\big(\gamma(t_i),\gamma(t_{i+1})\big)
        \geq \left( 1-\frac{2\varepsilon}{m +2\varepsilon}\right)  L^{\tau_{\Omega}}\big(\gamma\big).
    \end{split}
\]
Taking the infimum over all partitions $\sigma$ yields $L_\Omega^\tau (\gamma)\geq \left( 1-\frac{2\varepsilon}{m +2\varepsilon}\right)  L^{\tau_{\Omega}} (\gamma)$, which upon letting $\varepsilon\to0$ becomes $L_\Omega^\tau (\gamma)\geq L^{\tau_{\Omega}} (\gamma)$.

The intrinsicness of \(\tau_\Omega\) is now immediate, as we proved that the length functionals \(L^\tau_\Omega\) and \(L^{\tau_\Omega}\) agree: since the causal relations \(\leq\) and \(\leq_\Omega\) coincide (as we discussed at the beginning of this section), the causal curves between any two given points are the same and thus taking the supremum with respect to either length will give the same value.
\end{proof}

We conclude this section by proving a key property of conformal length, namely that consecutive conformal changes are equivalent to a single change by the product of the conformal factors. In this sense, conformal changes behave analogously to their classical counterparts in what regards to conformal length. This result will be essential later in \Cref{Conformal transformations as an equivalence relation} to demonstrate that conformal transformations define an equivalence relation for the class of intrinsic, quasi-strongly causal Lorentzian pre-length spaces. 

\begin{proposition}\label{prop:transitivity_conformal_length}
    Let $\Xll$ be an intrinsic, quasi-strongly causal Lorentzian pre-length space, and consider two conformal factors \(\Omega, \Omega'\). Then $L_{\Omega'}^{\tau_\Omega}= L_{\Omega\cdot \Omega'}^\tau$. In particular,  $L_{1/\Omega}^{\tau_\Omega}=L^\tau$ and $(\tau_\Omega)_{1/\Omega}=\tau$.
\end{proposition}
\begin{remark}
    By \(L_{\Omega'}^{\tau_\Omega}\) we mean the conformal length functional obtained by applying the conformal factor \(\Omega'\) to the time separation function \(\tau_\Omega\), i.e., for a given causal curve \(\gamma\) we have
    \[
        L_{\Omega'}^{\tau_{\Omega}} = \inf \left\{ \sum_{i=0}^m \max_{t \in [t_i, t_{i+1}]} \Omega'(\gamma(t)) \, \tau_\Omega(\gamma(t_i), \gamma(t_{i+1})) \ \colon \ \{t_i\}_{i=0}^m \ \text{partition of} \ [a,b]\right\}.
    \]
\end{remark}
\begin{proof}
Let $\gamma:[a,b]\to X$ be any causal curve. We take $\varepsilon > 0$ and consider $\delta>0$ small enough so that $\Omega\circ\gamma$, $\Omega'\circ \gamma$, and $\left(\Omega\cdot\Omega'\right)\circ \gamma$ are uniformly continuous with modulus of continuity $(\varepsilon, \delta)$. We also take $\varepsilon$ small enough so that $2\varepsilon$ is less than the minimum value of all the conformal factors along the curve. We take a partition $\sigma$ with modulus $\abs{\sigma} < \delta$ and for each interval $[t_i, t_{i+1}]$ of the partition we define
     \begin{align*}
         M_{i} &\coloneqq \max_{[t_i, t_{i+1}]} \Omega \circ \gamma, & M_{i}' &\coloneqq \max_{[t_i, t_{i+1}]} \Omega' \circ \gamma, &  M_{i}'' &\coloneqq \max_{[t_i, t_{i+1}]} (\Omega\cdot\Omega') \circ \gamma. & 
     \end{align*}
     Notice that $M_i''$ is not necessarily equal to the product of $M_i$ with $M_i'$, but just less than or equal; however, we can bound it from below by the product of the minimum of $\Omega$ over $[t_i, t_{i+1}]$ with $M_i'$, which in turn is bounded from below by $(M_i - \varepsilon)M_i'$ by uniform continuity. Therefore
     \begin{equation}\label{eq:transitivity_product_estimate}
        M_iM_i' - \varepsilon M \leq (M_i - \varepsilon)M_i' \leq M_i'' \leq M_i M_i' \implies \abs{M_i'' - M_i'M_i} \leq \varepsilon M, 
     \end{equation}
     where $M$ is the maximum of all of the conformal factors along the curve. 

 We now consider any partition $\theta=\{s_j\}_{j=0}^r$ of the interval $[t_i, t_{i+1}]$. It is then immediate to see, as done in \Cref{prop:length_causal_curves}, that
    \[
        (M_i - \varepsilon) \sum_{j=0}^{r-1} \tau(\gamma(s_j), \gamma(s_{j+1})) \leq \sum_{j=0}^{r-1} \max_{s\in [s_j, s_{j+1}]} \Omega\left(\gamma(s)\right)\tau(\gamma(s_j), \gamma(s_{j+1})) \leq  (M_i + \varepsilon) \sum_{j=0}^{r-1} \tau(\gamma(s_j), \gamma(s_{j+1})),
    \]
    which means that
    \[
        (M_i - \varepsilon)V_\theta({\gamma}\vert_{[t_i, t_{i+1}]}) \leq V_{\Omega, \theta}({\gamma}\vert_{[t_i, t_{i+1}]}) \leq (M_i + \varepsilon)V_\theta({\gamma}\vert_{[t_i, t_{i+1}]}),
    \]
    so that taking the infimum over all partitions $\theta$ yields
    \begin{equation}\label{eq:estimating_L_omega}
       (M_i - \varepsilon) L^\tau({\gamma}\vert_{ [t_i, t_{i+1}]}) \leq L_\Omega^\tau  ({\gamma}\vert_{ [t_i, t_{i+1}]}) \leq (M_i + \varepsilon) L^\tau({\gamma}\vert_{ [t_i, t_{i+1}]}).
    \end{equation}

In a completely analogous way, we also get the inequalities 
\begin{equation}\label{eq:estimating_L_omega_omega'}
    (M_i' - \varepsilon) L_\Omega^\tau ({\gamma}\vert_{ [t_i, t_{i+1}]}) \leq L_{\Omega'}^{\tau_\Omega}  ({\gamma}\vert_{ [t_i, t_{i+1}]}) \leq (M_i' + \varepsilon) L_\Omega^\tau ({\gamma}\vert_{ [t_i, t_{i+1}]}),
\end{equation}
\begin{equation}\label{eq:estimating_L_omega_times_omega'}
    (M_i'' - \varepsilon) L^\tau({\gamma}\vert_{ [t_i, t_{i+1}]}) \leq L_{\Omega\cdot\Omega'}^\tau ({\gamma}\vert_{ [t_i, t_{i+1}]}) \leq (M_i'' + \varepsilon) L^{\tau}({\gamma}\vert_{ [t_i, t_{i+1}]}),
\end{equation}
where in \eqref{eq:estimating_L_omega_omega'} we have used that $L_\Omega^\tau =L^{\tau_{\Omega}}$ (by \Cref{lemma:strongly_causal_implies_intrinsic}). Equations \eqref{eq:estimating_L_omega}-\eqref{eq:estimating_L_omega_omega'} can be rewritten as 
     \[
        \begin{split}
            \frac{L_\Omega^\tau ({\gamma}\vert_{ [t_i, t_{i+1}]})}{M_{i} + \varepsilon} &\leq L^\tau({\gamma}\vert_{ [t_i, t_{i+1}]}) \leq \frac{L_\Omega^\tau ({\gamma}\vert_{ [t_i, t_{i+1}]})}{M_{i} - \varepsilon},\\
            \frac{L_{\Omega'}^{\tau_\Omega}({\gamma}\vert_{ [t_i, t_{i+1}]})}{M_{i}' + \varepsilon} &\leq L_\Omega^\tau ({\gamma}\vert_{ [t_i, t_{i+1}]}) \leq \frac{L_{\Omega'}^{\tau_\Omega}({\gamma}\vert_{ [t_i, t_{i+1}]})}{M_{i}' - \varepsilon},
        \end{split}
     \]
     which can then be used in \eqref{eq:estimating_L_omega_times_omega'} to give the estimate
     \[
        \frac{M_i'' - \varepsilon}{(M_i' + \varepsilon)(M_i + \varepsilon)} L_{\Omega'}^{\tau_\Omega}({\gamma}\vert_{ [t_i, t_{i+1}]})\leq L_{\Omega\cdot\Omega'}^\tau({\gamma}\vert_{ [t_i, t_{i+1}]}) \leq \frac{M_i'' + \varepsilon}{(M_i' - \varepsilon)(M_i -\varepsilon)} L_{\Omega'}^{\tau_\Omega}({\gamma}\vert_{ [t_i, t_{i+1}]}).
     \]
    Thanks to \eqref{eq:transitivity_product_estimate} we can replace the denominators by $M_i'' \pm O(\varepsilon)$, where \(O(\varepsilon)\) does not depend on \(i\). The fractions can then be simplified to
     \[
        (1 - O(\varepsilon)) L_{\Omega'}^{\tau_\Omega}({\gamma}\vert_{ [t_i, t_{i+1}]})\leq L_{\Omega\cdot\Omega'}^\tau ({\gamma}\vert_{ [t_i, t_{i+1}]}) \leq (1+ O(\varepsilon))L_{\Omega'}^{\tau_\Omega}({\gamma}\vert_{ [t_i, t_{i+1}]}).
     \]
     Summing over $i$ then gives
     \[
        (1 - O(\varepsilon)) L_{\Omega'}^{\tau_\Omega}({\gamma})\leq L_{\Omega\cdot\Omega'}^{\tau} ({\gamma}) \leq (1+ O(\varepsilon))L_{\Omega'}^{\tau_\Omega}({\gamma}),
     \]
     which implies, letting $\varepsilon \to 0^+$, that $L_{\Omega\cdot\Omega'}(\gamma) = L_{\Omega'}^{\tau_\Omega}(\gamma)$, thus concluding the proof. The special case $L^{\tau_\Omega}_{1/\Omega}=L^\tau$ then follows immediately.
\end{proof}

\subsection{Uniqueness of the conformal time separation function}\label{Uniqueness of the conformal time separation function}

Given an intrinsic, quasi-strongly causal Lorentzian pre-length space \(\Xll\), a natural question is whether two different conformal factors can induce the same time separation function. Under certain assumption of finiteness of the time separation function $\tau$, the following proposition addresses this issue.
\begin{proposition}
    Suppose that for a quasi-strongly causal, intrinsic Lorentzian pre-length space $\Xll$ and for two conformal factors \(\Omega, \Omega'\) we have that \(\tau_{\Omega} = \tau_{\Omega'}\). Then for every point \(p\) such that there exists a sequence \(p_n\) in \(I^+(p)\) with \(p_n \to p\) and \(\tau(p,p_n) < +\infty\), we have \(\Omega(p) = \Omega'(p)\). An analogous statement holds for $p_n\in I^-(p)$.
\end{proposition}
\begin{proof}
    From \Cref{lemma:local_uniform_control_conformal_length} it follows
    \[
        \Omega(p) = \lim_{n \to +\infty} \frac{\tau_{\Omega}(p,p_n)}{\tau(p,p_n)} = \lim_{n \to +\infty} \frac{\tau_{\Omega'}(p,p_n)}{\tau(p,p_n)} = \Omega'(p).
    \] 
\end{proof}
\begin{remark} \label{remark:infinity_line_example}
    The previous result is false without the assumption on the finiteness of \(\tau\) along a sequence converging to the point under consideration: if we take \(X = \R\) and define the function
    \[
        \tau(x,y) \coloneqq
        \begin{dcases}
            + \infty \ \text{if} \ x < y\\
            0 \qquad \!\!\! \text{else}
        \end{dcases}
    \]
    and set the relations \(x \ll y \iff x<_{\R}y\) and \(x \leq y \iff x \leq_{\R}y\), it is easy to see that the space \((X, \ll, \leq, d_\R, \tau)\) is an intrinsic, globally hyperbolic Lorentzian pre-length space but for any choice of conformal factor \(\Omega\) we see that \(\tau = \tau_\Omega = +\infty\).
\end{remark}

\subsection{Conformal transformation as equivalence relation}\label{Conformal transformations as an equivalence relation}

The purpose of this section is twofold. We first introduce the notion of conformally-related Lorentzian pre-length spaces, and then prove that conformal transformations within the class of intrinsic, quasi-strongly causal Lorentzian pre-length spaces define an equivalence relation. In particular, this ensures that such class of spaces is closed under conformal changes.

\begin{definition}\label{def:conf_rel_LLS}
    (Synthetic conformal relation). Two intrinsic, quasi-strongly causal Lorentzian pre-length spaces $\Xll$, $\Xtildell$ are said to be conformally related if there exists a homeomorphism $\iota:X\longrightarrow\widetilde{X}$ and a conformal factor $\Omega:X\longrightarrow(0,+\infty)$, such that \(\iota^*(\widetilde{\ll}) =  \ll\), \(\iota^*(\widetilde{\leq}) = \leq\) and $\iota^*\widetilde{\tau}=\tau_{\Omega}$. In that case, we write $X \sim_{\iota,\Omega} \widetilde{X}$.
\end{definition}

\begin{remark}\label{remark:notation_iota_star}
    For any two points $p,q\in X$, the notation \(\iota^*(\widetilde{\ll}) = \ll\), \(\iota^*(\widetilde{\leq}) = \leq\) and 
    \(\iota^*\widetilde{\tau}=\tau_{\Omega}\) of \Cref{def:conf_rel_LLS} mean 
    \[
    p\ll q \ \Longleftrightarrow \ \iota(p) \  \widetilde{\ll} \  \iota(q),\quad 
    p\leq q \  \Longleftrightarrow  \ \iota(p) \ \widetilde{\leq}  \  \iota(q),\ \text{ and }\ 
    \widetilde{\tau}\big(\iota(p),\iota(q)\big) = \tau_{\Omega}\big(p,q\big), \text { respectively.}
    \]
\end{remark}

\begin{remark}\label{remark:commute_pullback_and_conformal_change}
    If \(\iota\) is an homemorphism preserving the causal relation and \(\omega\) a conformal factor on \(X\), it is straightforward to see that \((\iota^*\widetilde{\tau})_\omega = \iota^*(\widetilde{\tau}_{\omega \circ \iota^{-1}})\). Indeed, for \(p,q \in X\) we have that
    \[
        (\iota^*\widetilde{\tau})_\omega(p,q) = \sup \left\{ L_\omega^{i^*\widetilde{\tau}} (\gamma) \ \colon \ \gamma \ \text{causal from} \ p \ \text{to} \ q\right\},
    \]
    where
    \[
        \begin{split}
            L_\omega ^{i^*\widetilde{\tau}} (\gamma) &= \inf_{\sigma \ \text{partition}} \left\{ \sum_{i=0}^{m-1} \max_{t \in [t_i, t_{i+1}]} \omega(\gamma(t)) \, (i^*\widetilde{\tau})(\gamma(t_i), \gamma(t_{i+1}))\right\}\\
            &= \inf_{\sigma \ \text{partition}} \left\{ \sum_{i=0}^{m-1} \max_{t \in [t_i, t_{i+1}]} (\omega \circ \iota^{-1}) (\iota\circ\gamma)(t) \, \widetilde{\tau}(\iota \circ \gamma(t_i), \iota\circ\gamma(t_{i+1}))\right\}\\
            &= \widetilde{L}_{\omega \circ \iota^{-1}}(\iota \circ \gamma).
        \end{split}
    \]
    Any causal curve \(\gamma\) from \(p\) to \(q\) induces a causal curve \(\iota\circ \gamma\) from \(\iota(p)\) to \(\iota(q)\) and viceversa as \(\iota\) has a continuous inverse. This immediately implies that taking the supremum over all \(\gamma\) from \(p\) to \(q\) gives
    \[
        (\iota^*\widetilde{\tau})_\omega(p,q) = \widetilde{\tau}_{\omega \circ \iota^{-1}}(\iota(p), \iota(q)) = \iota^*(\widetilde{\tau}_{\omega \circ \iota^{-1}})(p,q).
    \]
\end{remark}
\begin{theorem}
    $\sim_{\iota,\Omega}$ is an equivalence relation.
\end{theorem}    
\begin{proof}
Reflexivity immediately follows by taking $\iota= \operatorname{id}_X$ and $\Omega\equiv 1$. For symmetry, if \(X \sim_{\iota, \Omega} \widetilde{X}\)  we can set \(j = \iota^{-1}\) and \(\Omega' \coloneqq \frac{1}{\Omega \circ j}\). It is then clear that \(j\) is an homeomorphism which preserves the causal and timelike relations. Moreover, using $\iota^*\widetilde{\tau}=\tau_{\Omega}$ together with \Cref{prop:transitivity_conformal_length}, we get \(\tau = (\iota^*\widetilde{\tau})_{1/\Omega}\) and from the previous remark we then get \(\tau = \iota^*(\widetilde{\tau}_{1/\Omega \circ j}) = \iota^*\widetilde{\tau}_{\Omega'}\). Then we see easily that
\[
    j^*\tau = j^*\big(\iota^*\widetilde{\tau}_{\Omega'}\big) = \widetilde{\tau}_{\Omega'},
\]
which proves symmetry.
Regarding transitivity, if \(X \sim_{\iota, \Omega}X'\) and \(X' \sim_{\iota', \Omega'}X''\) we define the map \(\iota'' := \iota' \circ \iota\) and the conformal factor \(\Omega'' := \Omega \cdot (\Omega' \circ \iota)\). Again, it is immediate to see that \(\iota''\) is an homeomorphism that preserves the causal relations. Using \Cref{remark:commute_pullback_and_conformal_change}, we rewrite the identity \(\iota'^*\tau'' = \tau'_{\Omega'}\) as \(\tau' = \iota'^*(\tau''_{1/\Omega' \circ {\iota'^{-1}}})\). This can be plugged in the identity \(\iota^*\tau' = \tau_\Omega\) to obtain \(\iota^*\iota'^* \big(\tau''_{1/\Omega' \circ \iota'^{-1}}\big) = \tau_\Omega\). On noticing that \(i^* \circ i'^* = i''^*\) we get
\[
    \begin{split}
        \tau_\Omega = i''^{*}(\tau''_{1/\Omega' \circ \iota'^{-1}}) = (\iota''^* \tau'')_{1/\Omega' \circ \iota'^{-1} \circ \iota''} = (\iota''^*\tau'')_{1/\Omega' \circ \iota}.
    \end{split}
\]
Applying \Cref{prop:transitivity_conformal_length} we see that
\(
    \iota''^*\tau'' = (\tau_\Omega)_{\Omega' \circ \iota} = \tau_{\Omega \cdot \Omega' \circ \iota} = \tau_{\Omega''},
\)
which concludes the proof.
\end{proof}

We conclude this section by showing that conformal transformations preserve angles between timelike curves of the same time orientation. We recall the definition of angle in the proof of \Cref{thm-con-ang-pre} below. For more details see \cite{BS:23}. Moreover, we need the notion of \emph{local causal closedness}, see \cite[Def.\ 3.4]{Kunzinger_2018}: a Lorentzian pre-length space $\Xll$ is locally causally closed if for any \(x \in X\) there exists a neighbourhood \(U\) such that, for any pair of sequences \(\{p_n\}_n\) and \(\{q_n\}_n\) in \(U\) converging to points \(p,q \in \overline{U}\) with \(p_n \leq q_n\), we have \(p \leq q\).

\begin{lemma}[Constructing null related sequences along two timelike curves]\label{lem-nul-seq}
Let $\Xll$ be a chronological, locally causally closed Lorentzian pre-length space and $\alpha,\beta\colon[0,b]\rightarrow X$ be two future directed timelike curves with $p:=\alpha(0)=\beta(0)$. Then there exist sequences $\{t_k\}_k\searrow 0$ and $\{s_k\}_k \searrow 0$ in $[0,b]$ such that $\alpha(s_k)\leq \beta(t_k)$ with $\tau(\alpha(s_k),\beta(t_k))=0$, that is, \(\alpha(s_k)\) and \(\beta(t_k)\) are null related.
\end{lemma}
\begin{proof}
    Consider the neighbourhood \(U\) of \(p\) given by the locally-causally-closed condition. Without loss of generality, since both \(\alpha\) and \(\beta\) are continuous, we can assume that the image of both of the curves is contained in \(U\), otherwise we just work with the restriction of the curves to an interval where both of them are in \(U\).
    
    Fix any \(k \in \N\) and consider the point \(q_k=\beta(b/k)\). Since the curve \(\beta\) is timelike, we have that \(p \in I^-(q_k)\), which is an open set. By continuity of \(\alpha\) then there exists some time \(s_k \in (0,b/k)\) such that \(a_k \coloneqq \alpha(s_k) \in I^-(q_k)\). Equivalently, we have that \(q_k \in I^+(a_k)\). Since \(p \ll a_k\), we have that \(p \not \in I^+(a_k)\) as the space is chronological. We can then define
    \[
        t_k \coloneqq \sup \left\{ t \in [0,b/k] \ \colon \ \beta(t) \not \in I^+(a_k) \right\}\,.
    \]
    Notice that \(t_k \in (0,b/k)\): firstly, since \(p \neq a_k\), we can take a small neighbourhood of \(p\) which is disjoint from \(a_k\) and contained in \(I^-(a_k)\). Using continuity of \(\beta\) at 0, the curve stays in the neighbourhood for a positive time and thus as the space is chronological none of these points belong to \(I^+(a_k)\), hence \(t_k >0\). Analogously, since \(\beta(b/k) = q_k \in I^+(a_k)\), which is an open set, by continuity the curve \(\beta\) must stay in such neighbourhood for all times sufficiently close to \(b/k\) and thus \(t_k < b/k\). By the same continuity argument, along with the fact that \(t_k\) is defined as a supremum, shows that \(b_k \coloneqq \beta(t_k) \not \in I^+(a_k) \). By local causal closedness, since for every \(t \in (t_k, b/k)\) we have \(a_k \ll \beta(t)\), we get that \(a_k \leq b_k\). Since \(a_k \not \ll b_k\), we have that \(\tau(a(s_k), b(t_k)) = 0\).

    To conclude, notice that both \(t_k\) and \(s_k\) are contained in \((0,b/k)\), hence the sequences \(\{t_k\}_k\) and \(\{s_k\}_k\) converge to 0 as \(k \to +\infty\); up to extracting subsequences, we may assume that both sequences are decreasing.
\end{proof}

\begin{theorem}[Conformal maps are angle preserving]\label{thm-con-ang-pre}
  Let $\Xll$ be an intrinsic, quasi-strongly causal, locally causally closed Lorentzian pre-length space, and let \(\Omega\) be a conformal factor. Let $\alpha,\beta\colon[0,b)\rightarrow X$ be two timelike curves of the same time orientation with $\alpha(0)=\beta(0)=:p$ and assume that \(\tau\) is finite between \(p\) and the endpoints of the curves. Then, if the angle between $\alpha$ and $\beta$ exists in $[0,\infty]$  with respect to both $\tau$ and $\tau_\Omega$, respectively, they agree, i.e.\ $\ma_p^\Omega(\alpha,\beta) = \ma_p(\alpha,\beta)$.
\end{theorem}
\begin{proof}
First, note that the timelike and causal relations and the time orientation of the curves are unchanged. Without loss of generality we assume $\alpha, \beta$ to be future directed (the past directed case works completely analogous), hence the sign $\sigma$ of the angle is $-1$.

By assumption we have
 \begin{align}
  \ma_p(\alpha,\beta) = \lim_{(s,t)\in A_0;\,s,t\searrow 0} \tilde\ma_p(\alpha(s),\beta(t))\,,\\
  \ma_p^\Omega(\alpha,\beta) = \lim_{(s,t)\in A_0;\,s,t\searrow 0} \tilde\ma_p^\Omega(\alpha(s),\beta(t))\,,
 \end{align}
where $A_0:=\{(s,t)\in(0,b)^2\colon \alpha(s)\leq\beta(t)$  or $\beta(t)\leq\alpha(s)\}$ and $\tilde\ma_p^\Omega(\alpha(s),\beta(t))$ is the comparison angle of the triangle with side lengths $\tau_\Omega(p,\alpha(s))$, $\tau_\Omega(p,\beta(t))$ and  $\max(\tau_\Omega(\alpha(s),\beta(t)),\tau_\Omega(\beta(t),\alpha(s)))$.
By finiteness of \(\tau\) and the hyperbolic law of cosines (see e.g.\ \cite[Lem.\ 2.4]{BS:23}) we have for $(s,t)\in A_0$ that
\begin{align*}
 \sigma\cosh(\tilde\ma_p(\alpha(s),\beta(t)))= \frac{a^2 + b^2 - c^2}{2 a b} &= \frac{1}{2} \frac{a}{b} + \frac{1}{2} \frac{b-c}{a} \left(1+\frac{c}{b}\right) \,,
\end{align*}
where $a,b,c$ are the side lengths with respect to $\tau$ of the timelike triangle $(p,\alpha(s),\beta(t))$.

Let $(s_k, t_k)_k$ be a sequence given by \Cref{lem-nul-seq}, hence 
\begin{align*}
 \ma_p(\alpha,\beta)  = \lim_{k\to\infty} \tilde\ma_p(\alpha(s_k),\beta(t_k))\,.
\end{align*}
We write $a_k:=\tau(p,\alpha(s_k))$, $b_k:=\tau(p,\beta(t_k))$, $x_k:= \frac{a_k}{b_k}$, $y_k:= \frac{b_k}{a_k}=\frac{1}{x_k}$ and notice that \(c_k \coloneqq \tau(a_k,b_k)\), which vanishes by \Cref{lem-nul-seq}. Hence $-\cosh(\ma_p(\alpha,\beta))=\lim_{k\to\infty}\Bigl(\frac{1}{2}x_k + \frac{1}{2} y_k\Bigr)$ with $x_k,y_k\geq 0$. We also denote by $\tilde a_k, \tilde b_k, \tilde c_k,  \tilde x_k, \tilde y_k$ the corresponding side-lengths and their fractions with respect to $\tau_\Omega$; notice that \(\tilde{c}_k = \tau_\Omega(\tilde a_k, \tilde b_k) = 0\) by \Cref{lemma:local_uniform_control_conformal_length}.

By construction we have for all $k\in\N$ that $a_k\leq b_k$, hence $x_k = \frac{a_k}{b_k} \leq 1$. By going over to further subsequences, this allows us to assume without loss of generality that $x_k\to x\in[0,1]$ and hence $y_k = \frac{1}{x_k} \to \frac{1}{x}=:y$, with the understanding that $\frac{1}{0}=+\infty$. At this point we observe that since everything is finite and positive we can multiply and divide to get
\begin{equation}
 \tilde x_k = \frac{\tilde a_k}{\tilde b_k}= \frac{a_k \frac{\tilde a_k}{a_k}}{b_k\frac{\tilde b_k}{b_k}} = x_k \frac{\tilde a_k}{a_k}\frac{b_k}{\tilde b_k}\,,
\end{equation}
and use that $\frac{\tilde a_k}{a_k}\to \Omega(p)$ and $\frac{b_k}{\tilde b_k}\to \frac{1}{\Omega(p)}$ by  \Cref{lemma:local_uniform_control_conformal_length}. Thus, we have $\tilde x_k\to x$ and analogously that $\tilde y_k \to y$.

Finally, this allows us to conclude that
\begin{align*}
\ma_p^\Omega(\alpha,\beta) = \lim_{k\to\infty}  \Bigl(\frac{1}{2} \tilde x_k  + \frac{1}{2} \tilde y_k \Bigr) = \ma_p(\alpha,\beta)\,.
\end{align*}
\end{proof}

If we have two conformally related spaces $\Xll$ and $\Xtildell$, angles in $X$ are the same as in $\widetilde{X}$ by pulling back the conformal factor and applying \Cref{thm-con-ang-pre}:
\begin{corollary}[Conformal equivalence preserves angles]
 Let $\Xll$, $\Xtildell$ be intrinsic, quasi-strongly causal and locally causally closed Lorentzian pre-length spaces that are conformally related via $\iota\colon X\rightarrow\widehat{X}$. Let $\alpha,\beta\colon[0,b)\rightarrow X$ be two timelike curves of the same time orientation with $\alpha(0)=\beta(0)=:p$ and assume that \(\tau\) is finite between \(p\) and the endpoints of the curves. Then, if the angles between $\alpha$ and $\beta$ and between $\iota\circ\alpha$ and $\iota\circ\beta$ exist in $[0,\infty]$, they agree, i.e.\ $\ma_{\iota(p)}^{\widetilde{X}}(\iota\circ\alpha,\iota\circ\beta) = \ma_p^X(\alpha,\beta)$.
\end{corollary}

\section{Applications}\label{sectionapplications}

Let us now focus on applications of the synthetic notion of conformal transformation introduced above. In particular, in this section we first establish its connection with the corresponding notion on smooth Lorentzian manifolds, proving that they are equivalent. We then show invariance of causality conditions under conformal changes, and give a characterisation of global hyperbolicity via finiteness of all time separation functions in the conformal class. Finally, we describe how the Lorentzian Hausdorff measures behave under conformal transformations.

\subsection{Compatibility  for smooth spacetimes}\label{Connection with smooth spacetimes}

As we see next, given a spacetime and a conformal transformation, the Lorentzian length of the conformal metric agrees with the synthetic conformal length $L^{\tau}_{\Omega}$ of the Lorentzian (pre-)length space induced from the spacetime. We recall that quasi-strong causality in the synthetic setting is equivalent to strong causality on smooth spacetimes (cf.\ \Cref{defintion:quasi_strong_causality}, \Cref{remark:quasi-strong-is-smooth-strong}).
\begin{lemma}
    Let $(M,g)$ be a strongly causal spacetime. Consider a smooth scalar function $\Omega:M\longrightarrow(0,+\infty)$ and a causal curve $\gamma: [a,b]\to M$ in \(M\). Define the metric \(\widetilde{g} := \Omega^2g\). 
    Then,
    \begin{equation}
        L_\Omega^\tau (\gamma)=L_{\widetilde g}(\gamma) = \int_a^b \Omega(\gamma(t))\sqrt{-g(\dot{\gamma}(t),\dot{\gamma}(t))}dt.
    \end{equation}
    In particular, the time separation functions agree, i.e., $\tau_\Omega = \tau_{\Omega^2 g}$.
\end{lemma}
\begin{proof}
 The proof is analogous to the one in \cite[Prop.\ 2.32]{Kunzinger_2018}. Since we are assuming (quasi-)strong causality, the notions of causal curve induced by the Lorentzian pre-length structure and the one induced by the smooth structure coincide \cite[Lemma 2.21]{Kunzinger_2018}. For any partition $\sigma = \{t_i\}_{i=0}^m$ of the interval $[a,b]$, by additivity of $L_{\widetilde{g}}$ we have (recall     \eqref{def:g-length})
 \[
    \begin{split}
        L_{\widetilde{g}}(\gamma) &= \sum_{i=0}^{m-1} \int_{t_i}^{t_{i+1}} \Omega(\gamma(t)) \sqrt{-g(\dot{\gamma}(t), \dot{\gamma}(t))} \, dt 
        \leq \sum_{i=0}^{m-1} \max_{t \in [t_i, t_{i+1}]} \Omega(\gamma(t)) \int_{t_i}^{t_{i+1}} \ \sqrt{-g(\dot{\gamma}(t), \dot{\gamma}(t))} \, dt \\
        &=  \sum_{i=0}^{m-1} \max_{t \in [t_i, t_{i+1}]} \Omega(\gamma(t)) L_g(\gamma\vert_{ [t_i,t_{i+1}]})\leq \sum_{i=0}^{m-1} \max_{t \in [t_i, t_{i+1}]} \Omega(\gamma(t)) \, \tau(\gamma(t_i), \gamma(t_{i+1})) 
    = V_{\Omega,\sigma}(\gamma),
    \end{split}
 \]
 where for the last inequality we have used that  $L_g(\gamma\vert_{[t_i, t_{i+1}]}) \leq \tau(\gamma(t_i), \gamma(t_{i+1}))$ for any $i \in \{0, \dots, m-1\}$ (cf.\ \eqref{def:time-sep-funct-spacetimes}).
 Taking then the infimum over all partitions $\sigma$ immediately yields that $L_{\widetilde{g}}(\gamma) \leq L_\Omega^\tau (\gamma)$. Now, at any point of differentiability of $\gamma$, it holds \cite[Prop.\ 2.32]{Kunzinger_2018}
 \[
    \lim_{h \to0^+} \frac{\tau(\gamma(t), \gamma(t+h))}{h} = \abs{\dot{\gamma}(t)}_{g}.
 \]
 By continuity of $\Omega \circ \gamma$ it then follows that 
 \[
    \lim_{h \to0^+} \max_{s \in[t,t+h]} \Omega(\gamma(s)) \,\frac{\tau(\gamma(t), \gamma(t+h))}{h} = \Omega(\gamma(t)) \, \abs{\dot{\gamma}(t)}_{g} = \abs{\dot{\gamma}}_{\widetilde{g}}.
 \]
 We now notice that
 \(
    \max_{s \in[t,t+h]} \Omega(\gamma(s)) \, \tau(\gamma(t), \gamma(t+h)) \geq L_\Omega^\tau (\gamma\vert_{ [t,t+h]}),
 \)
 since the left hand side is just $V_{\Omega,\sigma}(\gamma\vert_{ [t,t+h]})$ for $\sigma = \{t,t+h\}$. This, together with the previous estimate $L_\Omega^\tau (\gamma) \geq L_{\widetilde{g}}(\gamma)$, gives
 \[
    \max_{s \in[t,t+h]} \Omega(\gamma(s)) \,\frac{\tau(\gamma(t), \gamma(t+h))}{h} \geq \frac{L^\tau_{ \Omega}(\gamma\vert_{ [t,t+h]})}{h} \geq \frac{L_{\widetilde{g}}(\gamma\vert_{ [t,t+h]})}{h} = \frac{1}{h}\int_t^{t+h} \abs{\dot{\gamma}(s)}_{\widetilde{g}} \, ds.
 \]
 Both the initial and the last term of this chain of inequalities converge to $\abs{\dot{\gamma}(t)}_{\widetilde{g}}$ as $h \to 0^+$, therefore
 \[
    \lim_{h\to 0^+} \frac{L_\Omega^\tau (\gamma\vert_{ [t,t+h]})}{h} = \abs{\dot{\gamma}(t)}_{\widetilde{g}}.
 \]
 Thus, 
 the function $f(t) \coloneqq L_\Omega^\tau (\gamma\vert_{ [a,t]})$ admits right derivatives almost everywhere. Indeed, for almost every $t \in [a,b]$ we have 
 \[
    \lim_{h \to 0^+} \frac{L_\Omega^\tau (\gamma\vert_{ [a,t+h]}) - L_\Omega^\tau (\gamma\vert_{ [a,t]})}{h} 
    = \lim_{h \to 0^+} \frac{L_\Omega^\tau (\gamma\vert_{ [t,t+h]})}{h} = \abs{\dot{\gamma}(t)}_{\widetilde{g}},
 \]
 where in the first equality we used additivity of the length. Notice that the function $q(t) \coloneqq L_{\widetilde{g}}(\gamma\vert_{[a,t]})$ is absolutely continuous (as the primitive of a continuous function) and satisfies $f(0)=q(0)=0$. To conclude, we prove that $f$ is absolutely continuous, so that we may apply the Fundamental Theorem of Calculus. First we note that, for any subinterval $[a_i,b_i]$ in $[a,b]$, $L_\Omega^\tau (\gamma\vert_{ [a_i,b_i]}) \leq M \cdot L_{g}(\gamma\vert_{ [a_i,b_i]})$, where $M>0$ is the maximum of $\Omega \circ \gamma$ over the whole interval $[a,b]$. Indeed, for any partition $\sigma$ of the interval $[a_i,b_i]$, one gets
 \[
    \begin{split}
            V_{\Omega,\sigma}(\gamma\vert_{[a_i,b_i]}) &= \sum_{j=0}^{m-1} \max_{t \in [t_j, t_{j+1}]} \Omega(\gamma(t)) \, \tau(\gamma(t_j), \gamma(t_{j+1})) 
            \leq M \sum_{j=0}^{m-1} \tau(\gamma(t_j), \gamma(t_{j+1}))  = M\cdot V_\sigma(\gamma\vert_{ [a_i,b_i]}),
    \end{split}
 \]
 so that taking the infimum over $\sigma$ yields $L_\Omega^\tau (\gamma\vert_{ [a_i,b_i]}) \leq M L_{g}(\gamma\vert_{ [a_i,b_i]})$. On the other hand, the function $h(t) \coloneqq L_g(\gamma\vert_{ [a,t]})$ is absolutely continuous, hence for any $\varepsilon>0$ there exists $\delta>0$ such that, for any disjoint family $\{[a_i,b_i]\}_{i=0}^m$ of intervals in $[a,b]$, it holds
 \[
    \sum_{i=0}^m \abs{b_i-a_i} < \delta \implies \sum_{i=0}^m \abs{h(b_i)-h(a_i)} < \frac{\varepsilon}{M}.
 \]
 For any such family $\{[a_i,b_i]\}_{i=0}^m$, we get
 \[
    \begin{split}
        \sum_{i=0}^m \abs{f(b_i)-f(a_i)} &= \sum_{i=0}^m \abs{L_\Omega^\tau (\gamma\vert_{ [a,b_i]}) - L_{ \Omega}(\gamma\vert_{ [a,a_i]})} 
        = \sum_{i=0}^m L_\Omega^\tau (\gamma\vert_{ [a_i,b_i]}) \\
        &\leq \sum_{i=0}^m M \, L_{g}({\gamma}\vert_{ [a_i,b_i]}) 
        =M \sum_{i=0}^m \abs{h(b_i)-h(a_i)} < \varepsilon.
    \end{split}
 \]
 Therefore, $f$ is absolutely continuous; in particular, it is differentiable almost everywhere, so its derivative is equal to its right-sided derivative, i.e. $\abs{\dot{\gamma}(t)}_{\widetilde{g}}$. Applying the Fundamental Theorem of Calculus then gives
 \[
    L_\Omega^\tau (\gamma) = f(b) = \int_a^b \frac{d}{dt} f(t) \, dt = \int_a^b \abs{\dot{\gamma}(t)}_{\widetilde{g}} \, dt = \int_a^b \frac{d}{dt} q(t) \, dt = q(b) = L_{\widetilde{g}}(\gamma),
 \]
 which proves the theorem.
\end{proof}

\subsection{Causality conditions}\label{subsec-cc}

Originally presented in \cite[Section 2.7]{Kunzinger_2018} and later extended in \cite{ACS:20}, the causality conditions for Lorentzian pre-length spaces provide a synthetic counterpart of the causal ladder of smooth spacetimes. We next briefly review these conditions, and prove that they are preserved under conformal transformations. We also discuss forward completeness, a concept introduced in \cite{BBCGMORS:24}, also inspired by the considerations by Gigli \cite{Gig:25}, which is the Lorentzian analogue of metric completeness. Note that forward completeness does not fit into the causal ladder since it is implied by global hyperbolicity, but does not entail causal simplicity.

\begin{definition}\label{def:causality_conditions}
  A Lorentzian pre-length space \(\Xll\) is
  \begin{enumerate}[label=(\roman*)]
    \item {\em chronological} if the relation \(\ll\) is irreflexive, that is, \(\forall x \in X \ x \not \ll x\);
    \item {\em causal} if the relation \(\leq\) is a partial order, that is, \(\forall x,y \in X \ x \leq y \ \land \ y \leq x \implies x= y\);
    \item {\em non-totally imprisoning} if for any compact set $K$ in $X$, there exists a constant $C > 0$ which bounds from above the $d$-length of any causal curve $\gamma$ contained in $K$;
    \item {\em distinguishing} if for any \(x,y \in X\), whenever \(I^+(x) = I^+(y)\) or \(I^-(x) = I^-(y)\) hold we have \(x=y\);
    \item {\em strongly causal} if the family $\{ I^+(x) \cap I^-(y) \mid x, y \in X \}$ is a subbase of the topology defined by \(d\); this condition implies quasi-strong causality as stated in \Cref{defintion:quasi_strong_causality};
    \item {\em stably causal} if the smallest closed relation containing \(\leq\) is antisymmetric.
    \item {\em causally continuous} if it is distinguishing and for any \(x,y \in X\) we have that
    \(I^+(x) \subset I^+(y) \implies I^-(y) \subset I^-(x)\), and \(I^-(x) \subset I^-(y) \implies I^+(y) \subset I^+(x)\).
    \item {\em causally simple} if it is causal and for any \(x \in X\) the sets \(J^+(x)\) and \(J^-(x)\) are closed;
    \item {\em globally hyperbolic} if it is non-totally imprisoning and the sets $J^+(x) \cap J^-(y)$ are compact for all $x,y\in X$.
    \item[(*)] \emph{forward complete} if every bounded and monotonically increasing sequence converges, i.e., every sequence $(x_n)_{n\in\N}$ in $X$, with $x_n\leq x_{n+1}\leq z$ (for all $n\in\N$ and some $z\in X$) has a limit in $X$.
  \end{enumerate}
\end{definition}
\begin{remark}
    In Theorem 3.20 of \cite{ACS:20} it is shown that, for Lorentzian length spaces subject to certain technical assumptions, each of the conditions (i)--(ix) above implies the preceding one. It is also proven that 
    quasi-strong causality and strong causality are equivalent. 
\end{remark}
Observe that none of the conditions (i)--(ix), (*) involves the time separation function, but only the causal relations. In particular, the lower semicontinuity of $\tau$ plays no role, so one could state the same definitions for \say{Lorentzian pre-length spaces} not necessarily endowed with a lower semicontinuous time separation function. This allows us to make a slight abuse of notation in the following proposition, and refer to the conformally transformed space as Lorentzian pre-length space (despite the fact that lower semicontinuity of $\tau_{\Omega}$ requires quasi-strong causality, cf.\ \Cref{proposition:tau-lower-semi-continuous}). We do, however, assume intrinsicness, as otherwise the causal relations are not preserved under conformal changes (cf.\ \Cref{prop:character_causal_curves}). 
\begin{proposition}[Invariance of causality conditions under conformal transformations]\label{prop:conf-cc}
    Consider an intrinsic Lorentzian pre-length space \(\Xll\), a conformal factor \(\Omega \colon X \to (0, +\infty)\) and the conformally transformed Lorentzian pre-length space \(\confXll\). If any of the causality conditions in \Cref{def:causality_conditions} holds in \(\Xll\), it likewise holds  in \(\confXll\).
\end{proposition}
\begin{proof}
    The proof is straightforward as the conformally transformed space has the same causal and timelike relations as the starting space by \Cref{prop:character_causal_curves}. Moreover, the conditions depend only on the topology of the space, the metric and on the causal structure. Since conformal changes do not alter any of these, it immediately follows that they also hold for the final space.
\end{proof}

If we now consider conformally related Lorentzian pre-length spaces (\Cref{def:conf_rel_LLS}) then the invariance of non-total imprisonment does not hold any more. 
However, the following results still apply to (i), (ii),(iv)--(ix), (*).

\begin{proposition}
Let $\Xll$ and $\Xtildell$ be two intrinsic and quasi-strongly causal Lorentzian pre-length spaces that are conformally related, i.e, $X \sim_{\iota,\Omega} \widetilde{X}$. Then, any of the causality conditions (i), (ii), (iv)--(viii), (*) of \Cref{def:causality_conditions} holds if and only it holds for $\Xtildell$. Furthermore, if $X$ and $\widetilde{X}$ are Lorentzian length spaces (cf.\ \cite{Kunzinger_2018}), the same holds for global hyperbolicity (point (ix))\footnote{Actually, it suffices that both spaces are causally path-connected, $d$- and $\widetilde{d}$-compatible, and that through any point there is a timelike curve.}.
\end{proposition}
\begin{proof}
    The proof of points (i), (ii), (iv), (vi)--(viii), (*) is the same as in the proof of \Cref{prop:conf-cc}. For strong causality we observe that $\iota(x)\,\widetilde\ll\,\iota(y)$ if and only if $x\ll y$, hence $I(\iota(x),\iota(y)) = \iota(I(x,y))$ for all $x,y\in X$. Moreover, as \(\iota\) is an homeomorphism, it is immediate to see that it maps subbases of the starting space to subbases of the target. Hence, if we assume that $X$ is strongly causal, we immediately have that the chronological diamonds of \(\widetilde{X}\), being the image of the chronological diamonds of \(X\), make a subbase of the topology of \(\widetilde{X}\) and thus \(\widetilde{X}\) is strongly causal.

    Finally, for Lorentzian length spaces global hyperbolicity does not depend on the metric inducing the topology, hence is conformally invariant, see \cite[Thm.\ 3.7, Cor.\ 3.8]{Min:23}.
\end{proof}

\subsection{Global hyperbolicity and finiteness of Lorentzian distance}\label{Global hyperbolicity and finiteness of Lorentzian distance}

In the classical smooth setting, a strongly causal spacetime $(M,g)$ is globally hyperbolic if and only if the time separation function $\widetilde{\tau}$ induced from any metric $\widetilde{g}$ in the conformal class $[g]$ of $g$ is everywhere finite (see e.g.\ \cite[Theorem 4.30]{beem2017}). A similar result in the synthetic setting, however, is not to be expected without additional assumptions on the space. Indeed, the Lorentzian pre-length space constructed in \Cref{remark:infinity_line_example} is globally hyperbolic and intrinsic, but no choice of conformal factor can make the time separation function finite, as any compact causal curve of infinite length with respect to a time separation function will be of infinite length with respect to all conformal time separation functions. This pathology arises from the fact that individual causal curves on a compact interval can have infinite length, which cannot happen in the spacetime case. Therefore, it is natural to exclude this from the beginning. To relate global hyperbolicity with the finiteness of $\tau_{\Omega}$ for all conformal factors, 
we consider an intrinsic, quasi-strongly causal Lorentzian pre-length space $\Xll$ satisfying the following additional assumptions:
\begin{enumerate}[label=(\roman*)]
    \item \label{item:assumption_tau_finite} the time separation function \(\tau\) is finite;
    \item \label{item:assumption_non_empty_past} every point of $X$ has a non-empty chronological past, i.e., \(I^-(x) \neq \emptyset\) for all \(x \in X\);
    \item \label{item:assumption_causally_path_connected} $\Xll$ is {\em causally path connected} (cf.\ \Cref{definition:path-connectedness}).
\end{enumerate}
Let us check if these assumptions are indeed necessary. First, the example in \Cref{remark:infinity_line_example} satisfies \ref{item:assumption_non_empty_past} and \ref{item:assumption_causally_path_connected} but not \ref{item:assumption_tau_finite} and, as mentioned before, no conformal time separation function is finite. 
A more well-behaved example which satisfies the same assumptions is the following. Consider the real line \(\R\) with the natural causal/timelike structure given by the ordering of real numbers, and the map \(\tau \colon \R^2 \to \R\) defined as
\[
    \tau(x,y) \coloneqq 
    \begin{dcases}
        \int_x^y \frac{1}{\abs{t}} \, dt \ \text{if} \ x < y, \\
        0 \qquad\qquad \, \text{else}.
    \end{dcases}
\]
This function is indeed a time separation function which makes the space intrinsic, globally hyperbolic and quasi-strongly causal. However, the finiteness of all conformal time separation functions fails because all the non-constant causal curves through the origin have infinite \(\tau\)-length and thus will keep having infinite \(\tau\)-length after a conformal change.

As an example which satisfies assumptions \ref{item:assumption_tau_finite} and \ref{item:assumption_causally_path_connected} but not \ref{item:assumption_non_empty_past}, let us consider the subset of \(\R^2\) defined by (see  \Cref{figure:no_chronological_past})
\[
    X \coloneqq \bigg\{ \left(t,-\abs{t} +1\right): \ t\in [-1,1]\backslash\{0\}\bigg\} \cup \bigcup_{n \in \N} \left\{ \left(t,\abs{t} + 1-\frac{1}{n}\right) \ \colon \ t \in \left( -\frac{1}{2n}, \frac{1}{2n} \right) \right\}.
\]
We define the causal relation as follows: we say that  \(p \leq q\) if there is a causal curve in Minkowski space joining \(p\) to \(q\) that is entirely contained in \(X\).
We set the chronological relation to be empty and \(\tau \equiv 0\). This space is then an intrinsic, quasi-strongly causal Lorentzian pre-length space which satisfies \ref{item:assumption_tau_finite} and \ref{item:assumption_causally_path_connected}, but not \ref{item:assumption_non_empty_past}. The space is not globally hyperbolic as the diamond \(J((-1,0),(1,0))\) equals the whole space, which is not compact. However, for every conformal change \(\Omega\) it is immediate to see that \(\tau_\Omega \equiv 0\). Such example is interesting as it can be built as a subspace of the Minkowski plane by \say{intrinsifying} the standard time-separation function.

To build an example which satisfies \ref{item:assumption_tau_finite} and \ref{item:assumption_non_empty_past} but not \ref{item:assumption_causally_path_connected}, it is sufficient to consider the previous example and add a vertical segment \say{at the bottom} of the space (see \Cref{figure:non_causally_path_connected}):
\[
    Y = X \cup \left\{ (t,0) \ \colon \ t \in (-2,-1)  \right\}. 
\]

We define the causal relation just as in the previous case, and say that points \(p,q\) are timelike related if they are causally related with \(p \neq q\), and \(p\) lies in the (open) segment we added to the space \(X\). 
As time separation function $\tau$, we consider the Minkowskian length of any causal curve from \(p\) to \(q\); notice that this length is independent of the choice of curve. It is then immediate to see that this space is intrinsic, quasi-strongly causal and not globally hyperbolic. Moreover, every point has a non-empty chronological past and \(\tau\) is finite, but the space is not causally path-connected. For any conformal factor \(\Omega\), the corresponding \(\tau_\Omega\) is still finite as every causal curve between two points \(p\leq q\) will have \(L_\Omega^\tau \)-length bounded from above by \(M \tau(p,q)\), where \(M\) is the maximum of the function \(\Omega\) on the segment from \(p\) to \((0,0)\) (if \(p\) is not on the segment, then \(\tau_\Omega(p,q) = 0\)).

\begin{figure}[ht]
    \begin{minipage}{.52\textwidth}
        \centering
        \begin{tikzpicture}[scale=5.75, decoration={markings, 
                            mark= at position 0.52 with {\arrow{stealth}}},
                            ] 
            \draw[line width=1pt ,postaction ={decorate}]  (1,0) -- (0.4,0.6) ;
            \draw[line width=1pt ,postaction ={decorate}] (0.4,-0.6) -- (1,0);
        
            \filldraw[fill=white, draw=black, line width=0.8pt] (1,0) circle (0.45pt);
            \filldraw[fill=black, draw=black, line width=0.8pt] (0.4,0.6) circle (0.45pt);
            \filldraw[fill=black, draw=black, line width=0.8pt] (0.4,-0.6) circle (0.45pt);
            \draw[line width=1pt ,postaction ={decorate}] ({1-1/(0.25*((2.5)^2))},{0}) -- ({1-1/(0.25*2*((2.5)^2))},{1/(0.25*2*(2.5^2))});
            \draw[line width=1pt ,postaction ={decorate}] ({1-1/(0.25*(3^2))},{0}) -- ({1-1/(0.25*2*(3^2))},{1/(0.25*2*(3^2))});
            \draw[line width=1pt ,postaction ={decorate}] ({1-1/(0.25*(4^2))},{0}) -- ({1-1/(0.25*2*(4^2))},{1/(0.25*2*(4^2))});
            \draw[line width=1pt ,postaction ={decorate}] ({1-1/(0.25*(5^2))},{0}) -- ({1-1/(0.25*2*(5^2))},{1/(0.25*2*(5^2))});
            \draw[line width=1pt ,postaction ={decorate}] ({1-1/(0.25*(6^2))},{0}) -- ({1-1/(0.25*2*(6^2))},{1/(0.25*2*(6^2))});
            \draw[line width=1pt] ({1-1/(0.25*(7^2))},{0}) -- ({1-1/(0.25*2*(7^2))},{1/(0.25*2*(7^2))});
            \draw[line width=1pt] ({1-1/(0.25*(8^2))},{0}) -- ({1-1/(0.25*2*(8^2))},{1/(0.25*2*(8^2))});
            \draw[line width=1pt] ({1-1/(0.25*(9^2))},{0}) -- ({1-1/(0.25*2*(9^2))},{1/(0.25*2*(9^2))});
            \draw[line width=1pt] ({1-1/(0.25*(11^2))},{0}) -- ({1-1/(0.25*2*(11^2))},{1/(0.25*2*(11^2))});
            \draw[line width=1pt ,postaction ={decorate}] ({1-1/(0.25*2*(2.5^2))},{-1/(0.25*2*(2.5^2))}) -- ({1-1/(0.25*(2.5^2))},{0});
            \draw[line width=1pt ,postaction ={decorate}]  ({1-1/(0.25*2*(3^2))},{-1/(0.25*2*(3^2))}) -- ({1-1/(0.25*(3^2))},{0});
            \draw[line width=1pt ,postaction ={decorate}]  ({1-1/(0.25*2*(4^2))},{-1/(0.25*2*(4^2))})-- ({1-1/(0.25*(4^2))},{0});
            \draw[line width=1pt ,postaction ={decorate}] ({1-1/(0.25*2*(5^2))},{-1/(0.25*2*(5^2))}) -- ({1-1/(0.25*(5^2))},{0});
            \draw[line width=1pt ,postaction ={decorate}] ({1-1/(0.25*2*(6^2))},{-1/(0.25*2*(6^2))}) -- ({1-1/(0.25*(6^2))},{0});
            \draw[line width=1pt] ({1-1/(0.25*(7^2))},{0}) -- ({1-1/(0.25*2*(7^2))},{-1/(0.25*2*(7^2))});
            \draw[line width=1pt] ({1-1/(0.25*(8^2))},{0}) -- ({1-1/(0.25*2*(8^2))},{-1/(0.25*2*(8^2))});
            \draw[line width=1pt] ({1-1/(0.25*(9^2))},{0}) -- ({1-1/(0.25*2*(9^2))},{-1/(0.25*2*(9^2))});
            \draw[line width=1pt] ({1-1/(0.25*(11^2))},{0}) -- ({1-1/(0.25*2*(11^2))},{-1/(0.25*2*(11^2))});
        \end{tikzpicture}
        \vspace{1.18cm}
        \caption{The space \(X\) is made of points with empty chronological past.}
        \label{figure:no_chronological_past}
    \end{minipage}
    \quad
    \begin{minipage}{.45\textwidth}
        \centering
        \begin{tikzpicture}[scale=5.75, decoration={markings, 
                            mark= at position 0.52 with {\arrow{stealth}}},
                            ]
            \draw[line width=1pt, ,postaction ={decorate}]  (1,0) -- (0.4,0.6);
            \draw[line width=1pt,postaction ={decorate}] (0.4,-0.6) -- (1,0);
            \draw[line width=1pt,postaction ={decorate}] (0.4,-0.8) -- (0.4,-0.6);
        
            \filldraw[fill=white, draw=black, line width=0.8pt] (1,0) circle (0.45pt);
            \filldraw[fill=black, draw=black, line width=0.8pt] (0.4,0.6) circle (0.45pt);
            \filldraw[fill=white, draw=black, line width=0.8pt] (0.4,-0.8) circle (0.45pt);
            \draw[line width=1pt ,postaction ={decorate}] ({1-1/(0.25*((2.5)^2))},{0}) -- ({1-1/(0.25*2*((2.5)^2))},{1/(0.25*2*(2.5^2))});
            \draw[line width=1pt ,postaction ={decorate}] ({1-1/(0.25*(3^2))},{0}) -- ({1-1/(0.25*2*(3^2))},{1/(0.25*2*(3^2))});
            \draw[line width=1pt ,postaction ={decorate}] ({1-1/(0.25*(4^2))},{0}) -- ({1-1/(0.25*2*(4^2))},{1/(0.25*2*(4^2))});
            \draw[line width=1pt ,postaction ={decorate}] ({1-1/(0.25*(5^2))},{0}) -- ({1-1/(0.25*2*(5^2))},{1/(0.25*2*(5^2))});
            \draw[line width=1pt ,postaction ={decorate}] ({1-1/(0.25*(6^2))},{0}) -- ({1-1/(0.25*2*(6^2))},{1/(0.25*2*(6^2))});
            \draw[line width=1pt] ({1-1/(0.25*(7^2))},{0}) -- ({1-1/(0.25*2*(7^2))},{1/(0.25*2*(7^2))});
            \draw[line width=1pt] ({1-1/(0.25*(8^2))},{0}) -- ({1-1/(0.25*2*(8^2))},{1/(0.25*2*(8^2))});
            \draw[line width=1pt] ({1-1/(0.25*(9^2))},{0}) -- ({1-1/(0.25*2*(9^2))},{1/(0.25*2*(9^2))});
            \draw[line width=1pt] ({1-1/(0.25*(11^2))},{0}) -- ({1-1/(0.25*2*(11^2))},{1/(0.25*2*(11^2))});
            \draw[line width=1pt ,postaction ={decorate}] ({1-1/(0.25*2*(2.5^2))},{-1/(0.25*2*(2.5^2))}) -- ({1-1/(0.25*(2.5^2))},{0});
            \draw[line width=1pt ,postaction ={decorate}]  ({1-1/(0.25*2*(3^2))},{-1/(0.25*2*(3^2))}) -- ({1-1/(0.25*(3^2))},{0});
            \draw[line width=1pt ,postaction ={decorate}]  ({1-1/(0.25*2*(4^2))},{-1/(0.25*2*(4^2))})-- ({1-1/(0.25*(4^2))},{0});
            \draw[line width=1pt ,postaction ={decorate}] ({1-1/(0.25*2*(5^2))},{-1/(0.25*2*(5^2))}) -- ({1-1/(0.25*(5^2))},{0});
            \draw[line width=1pt ,postaction ={decorate}] ({1-1/(0.25*2*(6^2))},{-1/(0.25*2*(6^2))}) -- ({1-1/(0.25*(6^2))},{0});
            \draw[line width=1pt] ({1-1/(0.25*(7^2))},{0}) -- ({1-1/(0.25*2*(7^2))},{-1/(0.25*2*(7^2))});
            \draw[line width=1pt] ({1-1/(0.25*(8^2))},{0}) -- ({1-1/(0.25*2*(8^2))},{-1/(0.25*2*(8^2))});
            \draw[line width=1pt] ({1-1/(0.25*(9^2))},{0}) -- ({1-1/(0.25*2*(9^2))},{-1/(0.25*2*(9^2))});
            \draw[line width=1pt] ({1-1/(0.25*(11^2))},{0}) -- ({1-1/(0.25*2*(11^2))},{-1/(0.25*2*(11^2))});
        
        \end{tikzpicture}
        \caption{The space \(Y\) is not causally path connected.}
        \label{figure:non_causally_path_connected}
    \end{minipage}
\end{figure}
The assumptions \ref{item:assumption_tau_finite}--\ref{item:assumption_causally_path_connected} are therefore necessary to relate compactness of diamonds to finiteness of all conformal time separation functions and vice versa. This relation is established in the next theorem.
\begin{theorem}\label{theorem:cpt-diamonds-iff-tau-omega-finite}
    Let \(\Xll\) be an intrinsic and quasi-strongly causal Lorentzian pre-length space satisfying the assumptions \ref{item:assumption_tau_finite}-\ref{item:assumption_causally_path_connected}. Then, the following statements are equivalent:
    \begin{enumerate}
        \item All causal diamonds of $X$ are compact;
        \item For any choice of conformal factor $\Omega$, $\tau_{\Omega}$ is finite.
    \end{enumerate}
\end{theorem}
\begin{proof}
    We begin by showing \say{\(\implies\)}. Let us take any pair of points \(p \leq q\) and consider the diamond \(J(p,q)\). As the space is globally hyperbolic, this diamond is compact. We consider any conformal factor \(\Omega\) and denote by \(M\) the maximum value it attains on the diamond, which exists by compactness along with continuity of \(\Omega\). For any causal curve \(\gamma\) joining \(p\) to \(q\), which will then be contained in \(J(p,q)\), we can estimate its \(L_\Omega^\tau \)-length analogously to what we did in \Cref{prop:length_causal_curves} by
    \[
        L_\Omega^\tau (\gamma) \leq M L^\tau(\gamma) \leq M \tau(p,q).
    \]
    Notice that here \(M\) is independent of the curve, so by taking the supremum over all such curves \(\gamma\) we see that
    \[
        \tau_\Omega(p,q) \leq M \tau(p,q) < +\infty,
    \]
    as we assumed our initial time separation function to be finite.

    For \say{\(\impliedby\)}, we proceed by contradiction. Suppose that in the initial space there exist two points \(p \leq q\) such that their diamond \(J(p,q)\) is not compact. Then there must exist a sequence of points \(\{p_n\}_n\) in \(J(p,q)\) such that none of its subsequences converges; without loss of generality we can assume that the sequence is injective, so that \(n \neq m \implies p_n \neq p_m\). First, we take a point \(p^- \in I^-(p)\), so that by the push-up property of Lorentzian pre-length spaces we have \(p^- \ll p_n\) for all \(n \in \N\). We see that for every \(n \in \N\) we must have
    \[
        r(n) \coloneqq \inf_{m \neq n} d(p_n,p_m) >0,
    \]
    as otherwise we would be able to extract a subsequence converging to \(p_n\). We now see that the balls defined as
    \[
        B_n \coloneqq B\left( p_n, \frac{r(n)}{2} \right)
    \]
    must be mutally disjoint; indeed, if \(\exists z \in B_n \cap B_m\) for \(n \neq m\) by the triangle inequality we would have
    \[
        d(p_n,p_m) \leq d(p_n,z) + d(z,p_m) < \frac{r(n)}{2} + \frac{r(m)}{2} \leq \max(r(n),r(m)),
    \]
    but we also know that \(d(p_n,p_m)\) is greater or equal than both \(r(n)\) and \(r(m)\), therefore we would get \(d(p_n,p_m) \geq \max(r(n),r(m))\), which contradicts the previous chain of inequalities as one of them is a strict inequality.

    The aim now is to build a conformal factor which takes high values inside each \(B_n\) but is constant outside of them. To this end, for any \(B_n\) we look at a smaller ball \(U_n \coloneqq B(p_n, r(n)/4)\) and employ quasi-strong causality to get the existence of a smaller neighbourhood \(V_n \subset U_n\) such that any causal curve with endpoints in \(V_n\) stays in \(U_n\). By causal path connectedness, as \(p^- \ll p_n\), there exists a timelike curve \(\gamma_n\) joining \(p^-\) to \(p_n\). By continuity, we can consider a point \(s_n\) on this curve which is contained in \(V_n\); as \(\gamma_n\) is timelike, \(s_n \ll p_n\), so in particular \(\tau(s_n,p_n) > 0\). Since \(\tau\) is intrinsic we can find a causal curve \(\sigma_n\) joining \(s_n\) to \(p_n\) with positive length \(L^\tau(\sigma_n)\) and by quasi-strong causality this curve is entirely contained in \(U_n\). The concatenation of the curve \(\gamma_n\) until the point \(s_n\) and the curve \(\sigma_n\) produces a curve \(\eta_n\) from \(p^-\) to \(p_n\) which has positive length in a neighbourhood of the point \(p_n\).

    We define the map \(\Omega_n \colon X \to (0, +\infty)\) by
    \[
        \Omega_n(x) \coloneqq \frac{2n}{r(n)L^\tau(\sigma_n)}\operatorname{dist}(x, (B_n)^c) + 1
    \]
    and notice that this map satisfies the following properties:
    \begin{itemize}
        \item \(\Omega_n\) is continuous and in fact Lipschitz with constant equal to \(2n/(r(n)L^\tau(\sigma_n))\);
        \item \(\Omega_n \equiv 1\) outside of \(B_n\);
        \item \(\Omega_n(x) \geq \frac{n}{L^n(\sigma_n)} +1\) for every \(x \in U_n\).
    \end{itemize}
    We then define the map \(\Omega = \prod_{n \in \N} \Omega_n\) and see that such product is well defined since at every point at most only one of the factors is not equal to 1; in fact such product is continuous. Moreover, for any \(n\in \N\) it holds that \(\Omega \geq \Omega_n\), so that measuring the length of the curve \(\eta_n\) yields
    \[
        \tau_\Omega(p^-,p_n) \geq L_\Omega^\tau (\eta_n) \geq L_\Omega^\tau (\sigma_n) \geq L^\tau_{\Omega_n}(\sigma_n).
    \]
     As \(\sigma_n \subset U_n\), we estimate the last term in the same way as in the proof of \Cref{prop:length_causal_curves} by bounding \(\Omega_n\) from below with its minimum over \(U_n\), which we know to be greater than \(\frac{n}{L(\sigma_n)} +1\):
    \[
        L^\tau_{\Omega_n}(\sigma_n) \geq L^\tau(\sigma_n) \left( \frac{n}{L(\sigma_n)} +1 \right) \geq n.
    \]
    This shows that
    \[
        \tau_\Omega(p^-, p_n) \geq n.
    \]
    As \(p_n \ll q\), we see that
    \[
        \tau_\Omega(p^-,q) \geq \tau(p^-,p_n) + \tau(p_n, q) \geq n,
    \]
    and since this inequality holds for every \(n \in \N\) we must have \(\tau_\Omega(p^-,q) = + \infty\), which shows that for this choice of conformal factor the time separation function is not finite and therefore contradicts our assumption.
\end{proof}
\begin{remark}
    The assumption \ref{item:assumption_tau_finite} of \(\tau\) being finite is always satisfied by globally hyperbolic Lorentzian length spaces as defined in \cite{Kunzinger_2018}, and is only needed to prove one of the two implications (namely \say{\(\implies\)}), as \say{\(\impliedby\)} already takes all conformal time separation functions to be finite.
\end{remark}

According to \Cref{def:causality_conditions}, global hyperbolicity requires not only compactness of causal diamonds but also non-total imprisonment of $\Xll$. The latter must indeed be assumed, as it is possible to construct intrinsic, quasi-strongly causal Lorentzian pre-length spaces satisfying \ref{item:assumption_tau_finite}–\ref{item:assumption_causally_path_connected} that nevertheless fail to be non-totally imprisoning. 

To illustrate this, consider the space \(W \subset \R^2\) of \Cref{Figure:example-totally-imprisoning}, defined by

\[
    W \coloneqq \left\{ (t,0) \ \colon t \in  [0,1/\pi)  \right\} \cup \left\{ (t, t\sin(1/t)) \colon \ t \in (- 1/\pi,0] \right\}, 
\]
with a time separation function defined as
\[
    \tau((t_1,x_1),(t_2,x_2)) \coloneqq \max(t_2-t_1,0).
\]
Define the causal relations as
\[
        (t_1,x_1) \leq (t_2,x_2) \,  :\Longleftrightarrow \,  t_2 \geq t_1, 
        \qquad\qquad 
        (t_1,x_1) \ll (t_2,x_2) \, :\Longleftrightarrow \, t_2 > t_1, 
\]
and use the euclidean distance of \(\R^2\) restricted to \(W\) as distance function \(d\). Then, \((W,d,\ll,\leq,\tau)\) is a Lorentzian pre-length space that satisfies all the assumptions of \Cref{theorem:cpt-diamonds-iff-tau-omega-finite} but is not non-totally imprisoning, as the curve given by 
\(s \mapsto (s, s \sin(1/s))\) for \(s\in [-1/2\pi,0]\) is a future-directed causal curve contained in a compact set which has infinite Euclidean length.

\begin{figure}[ht]
  \centering
\begin{tikzpicture}
  \begin{axis}[
      xlabel={}, ylabel={},
      xmin=-0.3, xmax=0.3,
      ymin=-0.4, ymax=0.4,
      height=7cm,
      axis lines=middle,
      samples=1000,
      domain= -0.0001:0.318,
      xtick=\empty,
      ytick=\empty
      ]
    \addplot[line width=1.05pt,smooth,black] ({x*sin(deg(1/x))},-x);
    \draw[line width=1.05pt, black] (axis cs:0,0) -- (axis cs:0,{1/pi});
    \filldraw[fill = white, draw = black] (axis cs:0, {1/pi}) circle (2pt);
    \filldraw[fill = white, draw = black] (axis cs:0, {-1/(pi)}) circle (2pt);
    \node[anchor=west, black] at (axis cs:0.25,0.02) {$x$};
    \node[anchor=south, black] at (axis cs:0.02,0.3) {$t$};
    \node at (axis cs:0.25,0.25) {$\mathbb{R}^2$};
  \end{axis}
\end{tikzpicture}

  \caption{The space \(W\) fails the non-total imprisonment condition.}
  \label{Figure:example-totally-imprisoning}
\end{figure}
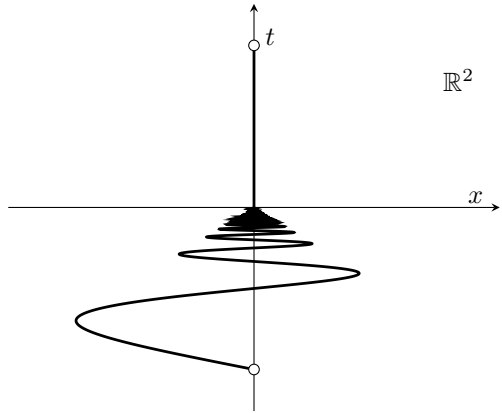

Consequently, we must consider non-total imprisoning spaces to obtain the following result, analogous to the classical one. 
\begin{theorem}
Let $\Xll$ be an intrinsic, non-totally imprisoning, and quasi-strongly causal Lorentzian pre-length space satisfying the assumptions  \ref{item:assumption_tau_finite}-\ref{item:assumption_causally_path_connected}. The following statements are equivalent:
\begin{enumerate}
    \item $\Xll$ is globally hyperbolic;
    \item For any choice of conformal factor $\Omega$, $\tau_{\Omega}$ is finite.
\end{enumerate}
\end{theorem}
\begin{proof}
Suppose that $\tau_{\Omega}$ is finite for any choice of conformal factor. Then by \Cref{theorem:cpt-diamonds-iff-tau-omega-finite} all diamonds in $\Xll$ are compact. Since $\Xll$ is by assumption non-totally imprisoning, it is globally hyperbolic. On the other hand, if $\Xll$ is globally hyperbolic, then all its diamonds are compact, and again by \Cref{theorem:cpt-diamonds-iff-tau-omega-finite} we have that $\tau_{\Omega}$ is finite for any choice of conformal factor.
\end{proof}

\subsection{Lorentzian Hausdorff measures}\label{Lorentzian Hausdorff measure}

In the recent work \cite{McCann_2022}, the authors introduce an analogue of the Hausdorff measure of metric spaces in the setting of Lorentzian (pre-)length spaces. The corresponding notion is that of Lorentzian Hausdorff measure, and consists of a one-parameter family of volume measures defined on a Lorentzian (pre-)length space. This important generalization motivates studying how these measures behave under conformal transformations, which is the aim of the present section. We begin by reviewing the corresponding definition.
\begin{definition}
    Consider a Lorentzian pre-length space \(\Xll\) and any Borel set \(E \subset X\). For a fixed positive value \(\delta>0\)  we define the set of {\em \(\delta\)-covers by diamonds} of \(E\) by
    \[
        \mathcal{J}_\delta(E) \coloneqq \left\{ \{J(p_i,q_i)\}_{i \in \N} \ \colon \bigcup_{i \in \N} J(p_i,q_i) \supset E, \ \operatorname{diam}^d(J(p_i,q_i)) < \delta \right\}.
    \]
    We then fix \(s>0\) and define the  {\em\(s\)-dimensional \(\delta\)-Lorentzian Hausdorff pre-measure of \(E\)} by
    \[
        \mathcal{H}^s_{\delta,\tau}(E) \coloneqq \inf \left\{ \sum_{i \in \mathbb{N}} \omega_s \tau(p_i,q_i)^s \ \colon \ \{J(p_i,q_i)\}_{i \in \N} \subset \mathcal{J}_\delta(E)\right\},
    \]
    where $\omega_s:=\frac{\pi^{(s-1)/2}}{s \Gamma((s+1)/2)2^{s-1}}$ is a normalization constant and $\Gamma$ is Euler's gamma function. Finally, we define the {\em \(s\)-dimensional Lorentzian Hausdorff measure of \(E\)} as
    \[
        \mathcal{H}_\tau^s(E) \coloneqq \lim_{\delta \to 0^+} \mathcal{H}_{\delta,\tau}^s(E).
    \]
    The value \(s\) is usually called {\em dimension exponent} of the Lorentzian Hausdorff measure.
\end{definition}
In \cite{McCann_2022} it is shown that this is indeed a measure on the Borel sets of \(X\). Moreover, it is also shown (Theorem 4.8) that for smooth (and large classes of continuous) spacetimes of dimension \(n\) with their natural structure of Lorentzian length spaces the measure \(\mathcal{H}_\tau^n\) coincides with the natural volume measure induced by the Lorentzian metric.

Our objective now is to describe the relation between the measures \(\mathcal{H}_\tau^s\) and \(\mathcal{H}_{\tau_\Omega}^s\) for all \(s > 0\). We will always assume the Lorentzian pre-length space to be intrinsic. In addition, for most results, we will assume that our space is separable; this assumption is very mild and allows us to apply standard tools of measure theory to study this problem.

The key ingredient that will allow us to study the two measures is given by the following lemma. We point out that at this level we need neither separability nor quasi-strong causality of the space.
\begin{lemma} \label{lemma:hausdorff_measure_control}
    Consider an intrinsic Lorentzian pre-length space \(\Xll\) and $\Omega\colon X\rightarrow(0,\infty)$ a conformal factor. For any point \(x \in X\), any Borel set \(E\), any dimension exponent \(s>0\) and any \(\varepsilon > 0\) there exists a positive value \(r(x,\varepsilon)>0\) such that for any \(r \in (0,r(x,\varepsilon))\) it holds that
    \[
        (\Omega(x)^s - \varepsilon)\mathcal{H}^s_\tau(E \cap B(x,r)) \leq \mathcal{H}_{\tau_\Omega}^s(E \cap B(x,r)) \leq (\Omega(x)^s+ \varepsilon) \mathcal{H}^s_\tau(E \cap B(x,r)).
    \]
\end{lemma}
\begin{proof}
    We fix any \(\varepsilon'>0\). By continuity of \(\Omega\) we can find a ball \(B(x,r')\) around \(x\) such that for any \(z \in B(x,r')\) we have \(\abs{\Omega(x) - \Omega(z)} < \varepsilon'\). Next, we consider any radius \(r'' < r'/2\) and work with the ball \(B(x,r'')\).

    Consider any \(\delta \in (0, r'')\) and take any \(\delta\)-cover \(\{J(p_i,q_i)\}_i \subset \mathcal{J}_\delta(E \cap B(x,r''))\). It is not too restrictive to assume that all the diamonds \(J(p_i,q_i)\) intersect \(E \cap B(x,r'')\), as otherwise we can just discard them and get a better estimate of the measure of the set. In turn this yields that all of the diamonds are contained in the set \(B(x,r')\): indeed, since each diamond has diameter less than \(\delta<r''\) and intersects the ball \(B(x,r'')\), the distance of each point of any diamond from \(x\) can be estimated by the triangle inequality by \(2r'' < r'\), so that the diamonds are contained in \(B(x,r')\). For every diamond \(J(p_i,q_i)\), we can estimate the conformal length of any causal curve between \(p_i\) and \(q_i\) as in \eqref{eq:local_uniform_control_conformal_length} as the diamonds are contained in \(B(x,r')\) and hence get that for all \(i \in \N\)
    \[
        (\Omega(x) - \varepsilon')\tau(p_i,q_i)\leq\tau_\Omega(p_i,q_i) \leq (\Omega(x)+\varepsilon') \tau(p_i,q_i).
    \]
    Raising the previous chain of inequalities to the \(s\)-th power and summing over \(i\) yields
    \[
        (\Omega(x) - \varepsilon')^s \sum_i\tau(p_i,q_i)^s\leq \sum_i\tau_\Omega(p_i,q_i)^s \leq (\Omega(x)+\varepsilon')^s \sum_i\tau(p_i,q_i)^s.
    \]
    Taking the infimum over all such \(\delta\)-covers we see that
    \[
        (\Omega(x)-\varepsilon')^s \, \mathcal{H}^s_{\delta,\tau}(E \cap B(x,r'')) \leq \mathcal{H}^s_{\delta,\tau_\Omega}(E \cap B(x,r'')) \leq (\Omega(x)+\varepsilon')^s  \,\mathcal{H}^s_{\delta,\tau}(E \cap B(x,r'')).
    \]
    Finally, sending \(\delta \to 0^+\) gives
    \[
        (\Omega(x)-\varepsilon')^s \, \mathcal{H}^s_{\tau}(E \cap B(x,r'')) \leq \mathcal{H}^s_{\tau_\Omega}(E \cap B(x,r'')) \leq (\Omega(x)+\varepsilon')^s \,\mathcal{H}^s_{\tau}(E \cap B(x,r'')).
    \]
    We stress that this inequality holds for any \(\varepsilon'>0\) and that the dependency on \(x\) and \(\varepsilon'\) is only through the fact that \(r'' < r'/2\), where \(r'\) depends on \(x\) and \(\varepsilon'\).

    To get the desired result, we notice that both factors \((\Omega(x) \pm \varepsilon')^s\) are continuous in \(\varepsilon'\) and thus for a given \(\varepsilon>0\) we can shrink \(\varepsilon'\) to a sufficiently small value so that
    \[
        \Omega(x)^s - \varepsilon \leq (\Omega(x) - \varepsilon')^s \leq (\Omega(x)+\varepsilon')^s \leq \Omega(x)^s + \varepsilon.
    \]
    The claim immediately follows by then choosing as \(r(x,\varepsilon) = r'(x,\varepsilon')/2\).
\end{proof}

With this technical lemma at hand, we can move towards our first step of determining the relation between \(\mathcal{H}^s_\tau\) and \(\mathcal{H}^s_{\tau_\Omega}\). 
\begin{proposition}\label{prop:hausdorff_absolutely_continuous}
    If \(\Xll\) is a separable, intrinsic Lorentzian pre-length space and \(\Omega\) is a conformal factor on \(X\), then for any choice of \(s>0\) the measure \(\mathcal{H}^s_{\tau_\Omega}\) is absolutely continuous with respect to \(\mathcal{H}^s_\tau\) and viceversa. In particular, we have that \(\mathcal{H}^s_\tau(X) = 0 \iff \mathcal{H}^s_{\tau_\Omega}(X) = 0\).
\end{proposition}
\begin{proof}
    We start by proving \(\mathcal{H}^s_{\tau_\Omega} \ll \mathcal{H}^s_\tau\). We consider any Borel set \(E\) such that \(\mathcal{H}^s_\tau(E) = 0\) and we prove that \(\mathcal{H}^s_{\tau_\Omega}(E) = 0\). For any point \(x \in E\), we consider a sufficiently small ball \(B(x,r(x))\) such that \Cref{lemma:hausdorff_measure_control} applies with \(\varepsilon = \Omega(x)^s/2\), so that we get the upper bound given by
    \[
        \mathcal{H}_{\tau_\Omega}^s(E \cap B(x,r)) \leq \frac{3\Omega(x)^s}{2} \mathcal{H}^s_\tau(E \cap B(x,r)) \leq \frac{3\Omega(x)^s}{2} \mathcal{H}^s_\tau(E) = 0,
    \]
    where we used monotonicity of the measure along with the fact that \(E\) has zero \(\mathcal{H}^s_\tau\)-measure. The family of such balls \(\{B(x,r(x))\}_{x \in E}\) is an open cover of \(E\) and since \(X\) is a separable metric space we can apply Lindel\"of's property to extract a countable subcover \(\{B(x_i,r_i)\}_{i \in \N}\) of \(E\); in particular, the family \(\{E \cap B(x_i,r_i)\}_{i \in \N}\) still is a countable cover of \(E\). We now can use \(\sigma\)-subadditivity of the measure \(\mathcal{H}^s_{\tau_\Omega}\) to see that
        \[
          \mathcal{H}^s_{\tau_\Omega}(E) \leq \sum_i \mathcal{H}^s_{\tau_\Omega}(E \cap B(x_i,r_i)) = 0,
    \]
    as each term on the right-hand side vanishes.

    The proof that \(\mathcal{H}^s_\tau \ll \mathcal{H}^s_{\tau_\Omega}\) follows word by word the previous one, but using the lower bound given by \Cref{lemma:hausdorff_measure_control} instead of the upper bound.
\end{proof}
\begin{remark}
    The previous proposition is an indication that the Lorentzian Hausdorff dimension of the space as defined in \cite{McCann_2022} should be a conformal invariant, since we have that
    \[
        \inf \{s>0\ \colon \ \mathcal{H}_\tau^s(X) = 0\} = \inf \{s>0\ \colon \ \mathcal{H}_{\tau_\Omega}^s(X) = 0\}.
    \]
    The dimension is defined as a similar infimum, over the values for which the measure of the space is finite, and finiteness is not preserved in general under conformal changes. In \cite{McCann_2022}, it is shown that with the additional assumption of {\em local \(d\)-uniformity}, the two infimums coincid. It is not clear, however, that such a condition is preserved under conformal changes.
\end{remark}

With a similar reasoning, we can also show that one of the two measures is \(\sigma\)-finite whenever the other one is. 
\begin{proposition}
    If \(\Xll\) is a separable, intrinsic Lorentzian pre-length space and \(\Omega\) is a conformal factor on \(X\), then for any choice of \(s>0\) the measure \(\mathcal{H}^s_{\tau_\Omega}\) is \(\sigma\)-finite if and only if \(\mathcal{H}^s_\tau\) is.
\end{proposition}
\begin{proof}
    We only prove that \(\mathcal{H}^s_{\tau_\Omega}\) is \(\sigma\)-finite when \(\mathcal{H}^s_\tau\) is, as the other implication is completely analogous. We take a \(\sigma\)-finite cover \(\{X_i\}_{i \in \N}\) of \(X\) for the measure \(\mathcal{H}^s_\tau\). For each \(i \in \N\), we employ \Cref{lemma:hausdorff_measure_control} at each point \(x\) of \(E=X_i\), with \(\varepsilon = \Omega(x)^s/2\), so that we get the upper estimate
    \[
         \mathcal{H}^s_{\tau_\Omega} (X_i \cap B(x,r(x))) \leq \frac{3\Omega(x)^s}{2} \mathcal{H}^s_\tau (X_i \cap B(x,r(x))) \leq \frac{3\Omega(x)^s}{2} \mathcal{H}^s_\tau (X_i) < +\infty.
    \]
    By separability, we can extract a countable subcover \(\{B(x_n^i,r_n^i)\}_{n\in \N}\) of \(X_i\), which by the previous estimate is made of finite \(\mathcal{H}^s_{\tau_\Omega}\)-measure sets. Then the family \(\{B(x_n^i,r_n^i)\}_{i,n \in \N}\) is a cover of \(X\) which satisfies the \(\sigma\)-finiteness condition for the measure \(\mathcal{H}^s_{\tau_\Omega}\), thus concluding the proof.
\end{proof}
We are now ready to state the main result of this section. We will additionally assume that the measure \(\mathcal{H}^s_\tau\) is \(\sigma\)-finite. In the smooth setting, this is the case when one chooses the manifold dimension as exponent, since in this case the measure agrees with the volume measure up to normalization (cf.\ \cite{McCann_2022}).
\begin{theorem}\label{theorem:Lorentzian-Hausdorff-measures}
    Suppose that \(\Xll\) is an intrinsic, separable Lorentzian pre-length space and assume that there exists a \(s>0\) such that \(\mathcal{H}^s_\tau\) is \(\sigma\)-finite. Then, for any conformal factor \(\Omega\) and any Borel set \(E\), we have that
    \[
        \mathcal{H}^s_{\tau_\Omega}(E) = \int_{E} \Omega(x)^s \, d\mathcal{H}^s_\tau(x).
    \]
\end{theorem}
\begin{proof}
    By \Cref{prop:hausdorff_absolutely_continuous} we know that \(\mathcal{H}^s_{\tau_\Omega}\) is absolutely continuous with respect to \(\mathcal{H}^s_\tau\); moreover we assumed the latter to be \(\sigma\)-finite. Hence we can apply Radon-Nikodym's theorem to obtain the existence of a measurable map \(f \colon X \to [0, + \infty]\) such that for any Borel set \(E\) we have
    \[
        \mathcal{H}^s_{\tau_\Omega}(E) = \int_{E} f(x) \, d\mathcal{H}^s_\tau(x).
    \]
    To conclude the proof, we only need to show that \(f(x) = \Omega(x)^s\) for \(\mathcal{H}^s_\tau\)-almost every \(x \in X\); we will argue by contradiction. If the set \(\{f \neq \Omega^s\}\) has positive measure, one of the sets \(\{f > \Omega^s\}\), \(\{f < \Omega^s\}\) must have positive measure; without loss of generality we can assume that this is the case for \(P \coloneqq \{f > \Omega^s\}\), the other case being totally analogous. We can write the set \(P\) as
    \[
        P = \bigcup_{n \in \N, n\geq 1} \left\{f > \Omega^s + \frac{1}{n} \right\},
    \]
    so that by \(\sigma\)-subadditivity we see that
    \[
        0<\mathcal{H}^s_\tau(P) \leq \sum_{n=1}^{+\infty} \mathcal{H}^s_\tau \left( \left\{f > \Omega^s + \frac{1}{n} \right\} \right),
    \]
    and hence there exists an integer \(n \in N\) such that the set \(E \coloneqq \left\{f > \Omega^s + 1/n\right\}\) has positive measure.
    
    We now apply \Cref{lemma:hausdorff_measure_control} to every point \(x \in E\) with \(\varepsilon = 1/2n\) to get an open cover of \(E\) consisting of balls and extract a countable subcover by separability; we can take the balls small enough so that the oscillation of \(\Omega^s\) on them is bounded by \(1/2n\). The same argument of subadditivity shows that there exists at least one ball \(B(x,r)\) such that \(E \cap B(x,r)\) has positive measure. We then see that
    \[
        \begin{split}
            \left(\Omega(x)^s + \frac{1}{2n} \right) \mathcal{H}^s_{\tau}(E \cap B(x,r)) &\geq \mathcal{H}^s_{\tau_\Omega}(E \cap B(x,r))\\ 
            &= \int_{E \cap B(x,r)} f(y) \; d\mathcal{H}^s_{\tau}(y) \\
            & \geq \int_{E \cap B(x,r)} \left(\Omega(y)^s + \frac{1}{n} \right) \; d\mathcal{H}^s_{\tau}(y).
        \end{split}
    \]
    For \(y \in E \cap B(x,r)\) it holds that \(\Omega(y)^s \geq \Omega(x)^s - 1/2n\) as the oscillation of \(\Omega^s\) was bounded by \(1/2n\). We can then resume the previous chain of inequalities with
    \[
        \begin{split}
            \left(\Omega(x)^s + \frac{1}{2n} \right) \mathcal{H}^s_{\tau}(E \cap B(x,r)) 
            & \geq \int_{E \cap B(x,r)} \left(\Omega(x)^s + \frac{1}{2n} \right) \; d\mathcal{H}^s_{\tau}(y)\\
            &= \left( \Omega(x)^s +\frac{1}{2n} \right) \mathcal{H}^s_\tau(E\cap B(x,r)).
        \end{split}
    \]
    As the last term of the chain is equal to the first one, all the inequalities must have been equalities; in particular we would get that
    \[
        \int_{E \cap B(x,r)} f(y) \, d\mathcal{H}^s_{\tau}(y) 
             = \int_{E \cap B(x,r)} \left(\Omega(y)^s + \frac{1}{n} \right) \, d\mathcal{H}^s_{\tau}(y),
    \]
    which combined with the fact that on every point of \(E\) the function \(f\) is greater than \(\Omega^s +1/n\) implies that \(f = \Omega^s+1/n\) for \(\mathcal{H}^s_\tau\)-almost every point of \(E \cap B(x,r)\). As this equality is never realized for any point of \(E\) this would yield that \(E \cap B(x,r)\) has zero \(\mathcal{H}^s_\tau\) measure, contradicting the fact that it has positive measure.

    We conclude that \(f = \Omega^s\) \(\mathcal{H}^s_\tau\)-almost everywhere and thus the proof is complete.
\end{proof}

%% file: metric_definition.tex
\section{Conformal Length on Metric Spaces}\label{Conformal Length on Metric Spaces}

As mentioned before, the approach we have followed to study conformal transformations in the synthetic Lorentzian setting can also be applied to metric and length spaces. This is the purpose of the present section. Specifically, we introduce the notions of conformal variation, length, and distance in the metric setting, and establish a number of results involving them. We then obtain a generalization of Theorem 1 in \cite{nomizu61}, which states that any Riemannian manifold $(M,g)$ admits a conformally related metric that makes $M$ Cauchy-complete. Finally, we prove a result analogous to \Cref{theorem:Lorentzian-Hausdorff-measures} for (conformal) Hausdorff measures.     

\subsection{Definition and basic properties}\label{basic_properties_metric_case}

Let us start by defining a conformal length in metric spaces, and show that it allows one to get a length structure. This issue has already been discussed in \cite{Han2019}, where a notion of conformal distance is defined as follows. Given a metric space \((X,d)\) and a conformal factor \(\Omega \colon X \to (0,+\infty)\), one defines a new distance by
\[
    \hat{d}_\Omega(x,y) = \inf\left\{ \int_0^1 \Omega(\gamma(t))v_\gamma(t) \, dt \ \colon \ \gamma \in \operatorname{AC}([0,1], X), \, \gamma(0) = x, \, \gamma(1)=y \right\},
\]
where \(v_\gamma\) is the metric speed featured in \Cref{def:metric_speed}. Such definition reflects the idea that the conformal distance should measure lengths as weighted integrals of the metric speed, where the weight is given by the conformal factor. One quickly realizes that, in order to give an analogous definition in the Lorentzian setting, it is necessary to have a notion of \say{causal speed}, which would require the additional structure of a \emph{measured} Lorentzian pre-length space. This approach is carried out in \cite{Bra:25c}. Instead, here we will work with a different notion of conformal distance 
by defining a length structure and showing it to be equivalent to the former one. As we will see, such a definition is the metric counterpart of the Lorentzian one.

If the starting space \((X,d)\) is a length space, the new metric will induce the same topology as the starting one, and the space will still be a length space. Hence, for length spaces, we will be able to define conformal equivalence between them in a completely analogous way to the one featured in \Cref{Conformal transformations as an equivalence relation}. As in the Lorentzian version, we start with the notion of conformal variation and length.
\begin{definition}[Metric conformal variation, conformal length]\label{conformal variation and length in ms}
    Let $(X,d)$ be a metric space, $\Omega\in C^{0}(X,(0, +\infty))$, and consider the set $A$ of all $C^0$-paths in $X$.\  For any path $\gamma:[a,b]\rightarrow X$ in $A$, and any partition $\sigma$ of $[a,b]$, we define the {\em conformal variation of $\gamma$ with respect to the conformal factor $\Omega$} as
    \[
        V_{\Omega,\sigma}^d(\gamma):=\sum_{i=0}^{m-1}\min_{t\in[t_i,t_{i+1}]}\Omega(\gamma(t)) d\big(\gamma(t_i),\gamma(t_{i+1})\big).
    \]
    The conformal length of $\gamma$ with respect to $\Omega$ is defined as
    \[
        L_{\Omega}^d(\gamma):=\sup \big\{V_{\Omega,\sigma}^d(\gamma):\sigma\textup{ is a partition of }[a,b]\big\}.
    \]
\end{definition}
\begin{remark}
    Comparing with \Cref{definition:conformal-factor-variation-length} one sees that this is the natural counterpart in the metric space setting, as the \(\max\) and \(\inf\) are replaced by \(\min\) and \(\sup\) respectively. This is so because metric geodesics are length {\em minimizers} and not maximizers.
\end{remark}
With a proof which is completely analogous to the one in \Cref{prop:basic_properties_conformal_length}, one gets that the conformal variation of a curve is {\em increasing} with respect to inclusions of partitions.
\begin{lemma}
\label{lemma:variations_increasing}
    The conformal variation \(V_{\Omega, \sigma}(\gamma)\) is increasing with respect to inclusions of partitions, namely, if \(\sigma\) and \(\eta\) are partitions of \([a,b]\) such that \(\sigma \subset \eta\), we have that \(V_{\Omega,\sigma}(\gamma) \leq V_{\Omega, \eta}(\gamma)\).\ Moreover, given any partition $\sigma$ of $[a,b]$, it holds
        \begin{align*}
            L_{\Omega}^d(\gamma)=\sup_{\substack{\eta \, \text{partition of} \, [a,b] \\ \sigma \subset \eta}} V_{\Omega, \eta}^d(\gamma).
        \end{align*}
\end{lemma}

Clearly, \( L^d_{\Omega}:A\longrightarrow[0,+\infty]\), so let us check whether it defines a length structure.
\begin{proposition}
Let $(X,d)$ be a metric space.\ The conformal length $L^d_{\Omega}$ with respect to a conformal factor $\Omega$, together with the set $A$ of all continuous paths in $X$, defines a length structure $(X,A,L^d_{\Omega})$, see \Cref{def-len-str} or \cite[\S 2.1]{Burago_2001} for details on length structures.
\end{proposition}
\begin{proof}
    Since $(X,d)$ is a metric space, the class of all continuous paths $A$ in $X$ with (monotonous surjective reparametrisations and) the variational length $L^d$ defines a length structure $(X, A, L^d )$.\ This means, in particular, that $A$ is closed under restrictions, concatenations, and linear reparametrisations.
    
    The additivity of the functional \(L^d_\Omega\) follows by arguing as in \Cref{prop:basic_properties_conformal_length}, by replacing every \(\min/\inf\) with \(\max/\sup\) and flipping all the inequalities.
    
    The next step is to check whether $L^d_{\Omega}$ depends continuously on the parameter of the path whenever the length of the path is finite, i.e.\ 
    \begin{align*}
     \lim_{t\rightarrow t_0} L_{\Omega}^d(\gamma\vert_{[a,t]})= L_{\Omega}^d(\gamma\vert_{[a,t_0]}),\qquad t_0,t\in[a,b],
    \end{align*}
    whenever \(L_\Omega(\gamma) < \infty\).
    Since \(\Omega \circ \gamma \) is a continuous function on \([a,b]\), it attains a maximum value \(M>0\) and a minimum \(m>0\). Similarly to \Cref{prop:length_causal_curves}, one gets that, for any \(s_1<s_2\) in \([a,b]\), on considering any variation of \(\gamma_{[s_1, s_2]}\) we have that
    \begin{equation} \label{eq:variation_estimate_above}
        \begin{split}
            mL^d(\gamma\vert_{ [s_1,s_2]}) \leq L_\Omega^d (\gamma\vert_{ [s_1,s_2]}) \leq M L^d(\gamma\vert_{ [s_1,s_2]}).
        \end{split}
    \end{equation}
    In particular, the length \(L^d(\gamma)\) of the curve with respect to the original metric is finite. On the other hand, $(X, A, L^d)$ being a length structure means that the induced length $L^d$ is continuous with respect to its parameter for curves of finite \(L^d\)-length. Therefore, for any $\varepsilon>0$ there exists a $\delta>0$ such that, if $\vert t-t_0\vert<\delta$ then $\left\vert L^d(\gamma\vert_{[a,t]})-L^d(\gamma\vert_{[a,t_0]})\right\vert<\varepsilon/M$. Thus, on denoting by \(s \coloneqq \min\{t_0,t\}, s'\coloneqq \max\{t_0,t\}\) we have
    \[
        \begin{split}
            \left\vert  L_{\Omega}^d(\gamma\vert_{[a,t]})- L_{\Omega}^d(\gamma\vert_{[a,t_0]})\right\vert&=\left\vert  L_{\Omega}^d(\gamma\vert_{[s,s']})\right\vert \\ 
            & \leq M \left\vert L^d(\gamma\vert_{[s,s']})\right\vert \\
            &= M \left\vert  L^{d}(\gamma\vert_{[a,t]})- L^{d}(\gamma\vert_{[a,t_0]})\right\vert <\varepsilon.
        \end{split}
    \]
    This proves that $L_{\Omega}^d$ continuously depends on the start/end-parameter of the curve.
    
    Now, conformal variations are obviously invariant under reparametrizations, hence so is $L^d_{\Omega}$.\ Indeed, if \(\phi \colon [c,d] \to [a,b]\) is a homeomorphism, then any partition \(\sigma\) of \([a,b]\) yields a partition \(\phi^{-1}(\sigma)\) of \([c,d]\) with the same variation.\ Conversely, for any partition \(\eta\) of \([c,d]\) we get the natural partition \(\phi(\eta)\) of \([a,b]\), which again has the same variation.
    
    Finally, let us check whether $L^d_{\Omega}$ is uniformly bounded away from zero for curves joining any point $x\in X$ and a point in the complement of any open neighbourhood $U_x\subset X$ of $x$. As \(U_x\) is a neighbourhood of \(x\), we can find a small radius \(\delta > 0\) such that \(\overline{B_d(x,\delta)}\) is contained in \(U_x\). By continuity of \(\Omega\), we can take an even smaller radius such that for any point \(z \in \overline{B_d(x,\delta)}\) it holds that
    \[
        m \coloneqq \frac{\Omega(x)}{2} \leq \Omega(z).
    \]
    As \((X,A,L^d )\)\ is a length structure, we get a constant \(C>0\) which bounds from below the \(d\)-length of any curve starting at \(x\) and ending outside \(B_d(x, \delta)\). We also notice that for any path \(\gamma\) contained in \( \overline{B_d(x,\delta)}\), reasoning similarly as we did in \eqref{eq:variation_estimate_above} it is straightforward to see that
    \[
        mL(\gamma)^d \leq L_\Omega^d(\gamma).
    \]
    We stress that in this case \(m\) is independent of the curve, but rather depends on the point \(x\). We now consider an admissible path $\gamma:[a,b]\rightarrow X$ from $\gamma(a)=x$ to $\gamma(b) \coloneqq y\in X\setminus U_x$ and define 
    \[
        T \coloneqq \sup\{ t \in [a,b] \  \colon \ \gamma(t') \in B_d(x,\delta) \quad \forall t'\in [a,t)\}.
    \]
    Notice that \(T\) is well defined as \(\gamma\) is continuous and \(\gamma(a) = x\); moreover, \(\gamma(T) \not \in B_d(x, \delta)\) again by continuity. We now see that
    \[
        \begin{split}
            L_\Omega^d(\gamma) \geq L_\Omega^d(\gamma\vert_{ [a,T]}) \geq m L^d(\gamma\vert_{[a,T]}) \geq mC.
        \end{split}
    \]
    As this estimate holds for any \(\gamma\) starting at \(x\) and ending outside \(U_x\), the uniform lower bound is proven. 
\end{proof}

\subsection{The induced conformal distance \texorpdfstring{\(d_\Omega\)}{d\_{omega}}}\label{section5.2}
Since we have proven that \((X,A,L_{\Omega}^d)\) is a length structure and the space \(X\) is Hausdorff (as its topology is induced by \(d\)), we can define the length metric \(d_\Omega\) induced by the length functional \(L_\Omega^d\) as follows.
 \begin{definition}[Conformal metric]
    Let $(X,d)$ be a metric space and consider the conformal length structure $(X,A,L_{\Omega}^d)$ given by all continuous paths in $X$ and the conformal length $L_{\Omega}^d$ with respect to a conformal factor $\Omega$.\ We define the {\em conformal length metric} $d_{\Omega}$ as
    \[
        d_\Omega(x,y):=\inf\{L_{\Omega}^d(\gamma):\gamma\in A, \ \gamma(a)=x,\gamma(b)=y\},
    \]
    where $x,y\in X$.
\end{definition}
\begin{remark}
    At this level, it could be the case that \(d_\Omega(x,y) = + \infty\) for some points \(x, y\in X\). Nevertheless, one can still define a topology and talk about a (pseudo)-metric structure; we discuss such a possibility in \Cref{remark:topologies_do_not_agree}.
\end{remark}
In particular, we know that the topology induced by \(d_\Omega\) is finer than the one we started with. In the next proposition, we show that under the assumption that \((X,d)\) is a length space, the two topologies actually coincide.
\begin{proposition}\label{prop:same_topologies}
    If \((X,d)\) is a length space and \(\Omega \colon X \to (0, +\infty)\) is a continuous function, the topology induced by the metric \(d\) is finer than the topology induced by the metric \(d_\Omega\). In particular, the two topologies agree.
\end{proposition}
\begin{proof}
    It is sufficient to prove that for any ball \(B_\Omega(x,r)\) with respect to the metric \(d_\Omega\) we can find another radius \(\rho >0\) such that the ball \(B_d(x,\rho)\) with respect to the metric \(d\) is contained in \(B_\Omega(x,r)\). If \(B_\Omega(x,r)\) is a given ball with respect to \(d_\Omega\), we take \(\varepsilon \in (0,\Omega(x))\) (notice that \(\Omega>0\) by assumption). By continuity, there exists a radius \(R>0\) such that
    \[
        \Omega(B_d(x,R)) \subset (-\varepsilon + \Omega(x), \Omega(x) + \varepsilon),
    \]
    where we stress that the ball on the left-hand-side is with respect to the metric \(d\). We now consider a positive number \(\delta>0\) such that
    \[
        \delta < \min\left(\frac{R}{r} \,, \frac{r}{\Omega(x) +\varepsilon}\right),
    \]
    and we define \(\rho \coloneqq \delta r\). We claim that \(B_d(x,\rho)\) is contained in \(B_\Omega(x,r)\). Indeed, for any \(y \in B_d(x,\rho)\) we know that
    \[
        d(x,y) < \rho = \delta r < R.
    \]
    As \((X,d)\) is a length space, there exists a continuous curve \(\gamma\) (with respect to the topology induced by \(d\)) from \(x\) to \(y\) such that \(L^d(\gamma) <\rho<R\). This implies that \(\gamma \subset B_d(x,R)\), hence we can estimate the \(L_\Omega^d\)-length of \(\gamma\) in a similar way as we did in \eqref{eq:variation_estimate_above}, since on the set \(B_d(x,R)\) the function \(\Omega\) is bounded above by \(\Omega(x)+ \varepsilon\):
    \[
        L_\Omega^d(\gamma) \leq (\Omega(x)+\varepsilon)L^d(\gamma)<(\Omega(x)+\varepsilon)\rho< (\Omega(x)+\varepsilon) \frac{r}{\Omega(x) +\varepsilon} =r,
    \]
    hence \(d_\Omega(x,y) < r\), which means that \(y \in B_\Omega(x,r)\). As this holds for any \(y \in B_d(x,\rho)\), we have that \(B_d(x,\rho) \subset B_\Omega(x,r)\), which concludes the proof.
\end{proof}
\begin{remark}\label{remark:topologies_do_not_agree}
    The previous result is not true in general if \((X,d)\) is not a length space. Indeed, we can consider the curve \((x,x\sin(1/x))\) for \(x \in [0,1/\pi]\) in \(\R^2\), endow it with the Euclidean metric and regard it as a metric space \((X,d)\). If we consider the conformal factor \(\Omega \equiv 1\), we see that \(d_\Omega\) is the metric induced from the length structure of such space. Clearly the \(d_\Omega\) distance between any point and the origin equals \(+\infty\), whereas any other pair of points are at finite distance. This implies that the topology induced by \(d_\Omega\) sees the space \(X\) as a disconnected space, whose connected components are \(\{(0,0)\}\) and \(X \setminus \{(0,0)\}\). Hence the topology is not the same as the Euclidean one induced on \(X\).
\end{remark}
\begin{corollary}
    If \((X,d)\) is a length space and \(\Omega \colon X \to (0, +\infty)\) is a conformal factor, then the set of continuous paths with respect to \(d\) and the set of continuous paths with respect to \(d_\Omega\) coincide.
\end{corollary}
\begin{proof}
    By the previous proposition, we know that the induced topologies coincide and thus functions are continuous with respect to \(d\) if and only if they are continuous with respect to \(d_\Omega\).
\end{proof}
\begin{corollary}\label{coro:conformal_metrics_are_loc_lip}
    If \((X,d)\) is a length space and \(\Omega\) is a conformal factor on \(X\), then the metric \(d_\Omega\) is locally biLipschitz equivalent to \(d\). More precisely, for any \(\varepsilon>0\) and any point \(x \in X\) we can find a ball \(B_d(x,r)\) (or a ball \(B_\Omega(x,r)\)) such that for every \(y,z \in B_d(x,r)\) (or in \(B_\Omega(x,r)\)) it holds that
    \[
        \left(\Omega(x) - \varepsilon \right)d(y,z) \leq d_\Omega(y,z) \leq (\Omega(x) + \varepsilon)d(y,z).
    \]
\end{corollary}
\begin{proof}
    We choose any point \(x \in X\) and any \(\varepsilon > 0\); as the topologies induced by \(d\) and \(d_\Omega\) coincide, the function \(\Omega\) is continuous with respect to the topology induced by \(d_\Omega\). We can then consider a \(d_\Omega\)-ball \(B_\Omega(x,r)\) such that for every \(y\) in the ball it holds
    \[
        \abs{\Omega(x) - \Omega(y)} < \varepsilon.
    \]
    By \Cref{prop:same_topologies} we know we can find a \(d\)-ball \(B_d(x,s)\) entirely contained in \(B_\Omega(x,r)\). We again employ \Cref{prop:same_topologies} to find a \(d_\Omega\)-ball \(B_\Omega(x,r')\) contained in \(B(x,s)\) and we take it small enough that \(r' <r/4\). Once again, we take a ball \(B_d(x,s')\) contained in \(B_\Omega(x,r')\) and small enough so that \(s'< s/4\). We claim that on the neighbourhood \(U \coloneqq B_d(x,s')\) the two metrics are Lipschitz equivalent. Indeed, let us take \(y,z \in U\) and consider first an \(\varepsilon'\)-\(d\)-geodesic between them which we denote by \(\gamma_{\varepsilon'}^d\). Clearly \(d_\Omega(y,z) \leq L_\Omega^d(\gamma_{\varepsilon'}^d)\). Suppose that the geodesic leaves the ball \(B_d(x,s)\). We then consider a \(\varepsilon''\)-\(d\)-geodesic from \(x\) to \(y\), which we call \(\sigma_{\varepsilon''}^d\) and concatenate it with \(\gamma_{\varepsilon'}^d\). The concatenated curve starts at \(x\) and leaves \(B_d(x,s)\), so it has \(d\)-length bounded from below by \(s\):
    \[
        L^d(\sigma_{\varepsilon''}^d \gamma_{\varepsilon'}^d) \geq s.
    \]
    By additivity of the length and using that they are almost geodesics, we get that
    \[
        d(x,y) + d(y,z) + \varepsilon' + \varepsilon'' \geq s.
    \]
    On the other hand, the points \(y\) and \(z\) were taken to be inside \(B_d(x,s')\) and thus we can estimate the left hand side from above as
    \[
        3s' + \varepsilon' + \varepsilon'' \geq s.
    \]
    Our assumption is on the curve \(\gamma_{\varepsilon'}^d\) and not on the curve \(\sigma_{\varepsilon''}^d\), hence this estimate holds for any \(\varepsilon'' >0\). Sending it to zero we can rewrite the previous inequality as
    \[
        \varepsilon' \geq s-3s' > s',
    \]
    since we had \(4s'<s\). Hence we have shown that a necessary condition for the curve to leave the neighbourhood \(B_d(x,s)\) is that \(\varepsilon '> s'\). Conversely then, if we take any \(\varepsilon' < s'\), we get that any \(\varepsilon'\)-\(d\)-geodesic must lie inside \(B_d(x,s)\) and in particular inside \(B_\Omega(x,r)\). On such set, we can estimate the \(L_\Omega^d\)-length of curves as in \eqref{eq:variation_estimate_above} and see that
    \[
        d_\Omega(y,z) \leq L_\Omega^d(\gamma_{\varepsilon'}^d) \leq (\Omega(x) + \varepsilon)L^d(\gamma_{\varepsilon'}^d) \leq (\Omega(x) + \varepsilon) (d(x,y) + \varepsilon').
    \]
    As this holds for any \(\varepsilon'<s'\), sending it to zero we see that \(d_\Omega(y,z) \leq (\Omega(x) + \varepsilon) \cdot d(y,z)\).

    We can essentially repeat the same reasoning now, starting with a \(\varepsilon'\)-\(d_\Omega\)-geodesic between \(y\) and \(z\), which  we denote by \(\gamma_{\varepsilon'}^{d_\Omega}\). We can see analogously that if \(\varepsilon' <r'\) then the curve must be contained inside \(B_\Omega(x,r)\). In particular, we can still estimate as in \eqref{eq:variation_estimate_above} and see that
    \[
        d_\Omega(y,z) \geq -\varepsilon' + L_\Omega^d(\gamma_{\varepsilon'}^{d_\Omega}) \geq -\varepsilon' + (\Omega(x) - \varepsilon)L^d(\gamma_{\varepsilon'}^{d_\Omega}) \geq -\varepsilon' + (\Omega(x) - \varepsilon)d(y,z),
    \]
    so that sending \(\varepsilon' \to 0^+\) we get \(d_\Omega(y,z) \geq (\Omega(x) - \varepsilon)d(y,z)\).
    Overall, we found that
    \[
        (\Omega(x) - \varepsilon)d(y,z)\leq d_\Omega(y,z) \leq (\Omega(x) + \varepsilon)d(y,z)
    \]
    for every \(y,z \in U\) and thus the claim follows as \(U\) is open in both topologies and thus we can consider a smaller ball centered at \(x\) inside it with respect to both metrics.
\end{proof}

We now prove an important property of the length functional \(L_\Omega^d\), namely, that it is lower semicontinuous with respect to \(d_\Omega\)-uniform convergence.
\begin{proposition}\label{prop:length_lower_semicontinuous}
    Consider a length space \((X,d)\) and a conformal factor \(\Omega \colon X \to (0, +\infty)\). Consider a sequence of paths \(\{\gamma_n\}_n\) which converges uniformly with respect to the metric \(d_\Omega\) to a path \(\gamma\). Then it holds that
    \[
        L^d_\Omega(\gamma) \leq \liminf_{n \to +\infty} L^d_\Omega(\gamma_n).
    \]
\end{proposition}
\begin{proof}
    First, we notice that for any partition \(\sigma\) of \([a,b]\) applying triangle inequality yields
    \[
        \begin{split}
         &\abs{ V_{\Omega,\sigma}^d (\gamma) -V_{\Omega,\sigma}^d(\gamma_n)}=\\
            &\ =\abs{\sum_{i=1}^r \min_{t \in [t_i,t_{i+1}]}\Omega(\gamma(t)) \, d(\gamma(t_i),\gamma(t_{i+1})) -  \min_{t \in [t_i,t_{i+1}]} \Omega(\gamma_n(t)) \, d(\gamma_n(t_i), \gamma_n (t_{i+1}))} \\
            &\ =\Bigg\vert\sum_{i=1}^r \left(\min_{t \in [t_i,t_{i+1}]} \Omega(\gamma(t)) - \min_{t \in [t_i,t_{i+1}]} \Omega(\gamma_n(t)) \right) d(\gamma_n(t_i),\gamma_n(t_{i+1}))\ \\
            & \quad +\sum_{i=1}^r \min_{t \in [t_i,t_{i+1}]}\Omega(\gamma(t)) \left( d(\gamma(t_i),\gamma(t_{i+1}))- d(\gamma_n(t_{i}), \gamma_n(t_{i+1})) \right) \Bigg\vert\\
            & \ \leq \sum_{i=1}^r \Bigg\vert \min_{t \in [t_i,t_{i+1}]} \Omega(\gamma(t)) - \min_{t \in [t_i,t_{i+1}]} \Omega(\gamma_n(t)) \Bigg\vert d(\gamma_n(t_i),\gamma_n(t_{i+1}))\ \\
            &\quad +\sum_{i=1}^r \min_{t \in [t_i,t_{i+1}]}\Omega(\gamma(t)) \Bigg\vert d(\gamma(t_i),\gamma(t_{i+1}))- d(\gamma_n(t_{i}), \gamma_n(t_{i+1}))  \Bigg\vert. \\
        \end{split}
    \]
    The last sum obviously goes to zero as \(n \to +\infty\) since we have pointwise convergence. To estimate the other sum, we work on every \([t_i,t_{i+1}]\). For any \(t\) and any \(\varepsilon>0\), there exists a \(\delta_t>0\) such that \(\Omega(B_{\Omega}(\gamma(t),\delta_t)) \subset B_\R(\Omega(\gamma(t)), \varepsilon)\).
    By compactness we can find a finite subcover and thus a positive \(\rho>0\) such that for any point \(\gamma(t)\) the ball \(B_{\Omega}(\gamma(t),\rho)\) is contained in the finite union \(U\) of the balls \(B_{\Omega}(\gamma(t),\delta_t)\); we can take \(\rho\) to be smaller than the (finitely many) radii \(\delta_t\) of the balls defining \(U\). Then we can use uniform convergence of the sequence \(\gamma_{n \vert[t_i,t_{i+1}]}\) with respect to \(d_\Omega\) to get that eventually in \(n\) the \(C^0\) distance of the sequence from \(\gamma\vert_{[t_i,t_{i+1}]}\) is less than \(\rho\); in particular, the sequence is contained in \(U\). Let \(t\) be the point in \([t_i,t_{i+1}]\) where \(\Omega\circ \gamma_n\) attains its minimum. Then there exists a ball \(B_{\Omega}(\gamma(t'), \delta_{t'})\) of the finite covering \(U\) which contains \(\gamma_n(t)\). Hence we get
    \[
        \min_{s \in [t_i,t_{i+1}]} \Omega(\gamma_n(s))=\Omega(\gamma_n(t))  \geq \Omega(\gamma(t')) - \varepsilon \geq \min_{s\in [t_i,t_{i+1}]} \Omega(\gamma(s)) - \varepsilon.
    \]
    Conversely, let \(t\) be the point in \([t_i,t_{i+1}]\) where \(\Omega \circ \gamma\) takes its minimum. As \(\gamma(t) \in U\), there exists an open ball \(B_\Omega(\gamma(t'), \delta_{t'})\) of the finite subcover containing \(\gamma(t)\). Then the point \(\gamma_n(t')\) belongs to \(B_\Omega(\gamma(t'),\rho) \subset B_\Omega(\gamma(t'), \delta_{t'})\) as we took \(\rho\) to be smaller than the radii of the finite subcover. We then have
    \[
        \begin{split}
            \min_{s \in [t_i,t_{i+1}]} \Omega(\gamma(s)) &=\Omega(\gamma(t))  \geq \Omega(\gamma(t')) - \varepsilon \geq \Omega(\gamma_n(t')) - 2\varepsilon \\
            &\geq \min_{s\in [t_i,t_{i+1}]} \Omega(\gamma_n(s)) - 2\varepsilon.
        \end{split}
    \]
    Hence we have found that, for \(n\) big enough
    \[
        \abs{\min_{s \in [t_i,t_{i+1}]} \Omega(\gamma(s)) -\min_{s\in [t_i,t_{i+1}]} \Omega(\gamma_n(s))} \leq 2\varepsilon.
    \]
    Since the number of intervals \([t_i,t_{i+1}]\) is finite, we can find a common value \(N\) such that all of the bounds on each interval hold. Since the sequence is eventually in a neighbourhood of \(\gamma\) which has bounded diameter, the terms of the form \(d(\gamma_n(t_i), \gamma_n(t_{i+1}))\) are uniformly bounded and thus also the first sum is going to zero. This proves that
    \[
        V_{\Omega,\sigma}(\gamma_n)  \to V_{\Omega,\sigma}(\gamma) \ \text{as} \ n \to +\infty.
    \]
    So, for a fixed partition, the variations of \(\gamma_n\) converge to the variation of \(\gamma\). We then choose a partition \(\sigma\) such that, for \(\varepsilon>0\) we have \(L_\Omega^d(\gamma) - \varepsilon < V_{\Omega,\sigma}^d(\gamma)\); if \(L_\Omega^d(\gamma) = +\infty\), we can choose a partition so that \(V_{\Omega,\sigma}(\gamma) > M\) for a fixed \(M>0\) and the following reasoning applies analogously. Since we know that \(V_{\Omega, \sigma}^d(\gamma_n) \to V_{\Omega,\sigma}^d(\gamma)\), we have that eventually in \(n\) it holds that
    \[
        L_\Omega^d(\gamma) - \varepsilon < V_{\Omega, \sigma}^d(\gamma_n) \leq L_\Omega^d(\gamma_n).
    \]
    Taking the \(\liminf\) as \(n \to +\infty\), we get
    \[
        L_\Omega^d(\gamma) - \varepsilon \leq \liminf_{n \to +\infty} L_\Omega^d(\gamma_n).
    \]
    Since this holds for any \(\varepsilon>0\), this gives that 
    \[
        L_\Omega^d(\gamma) \leq \liminf_{n \to +\infty} L_\Omega^d(\gamma_n)
    \]
    which proves lower semicontinuity of \(L_\Omega^d\).
\end{proof}

Our next goal is to explore the relationship between $L_{\Omega}^d$ and the variational length associated to $d_{\Omega}$. 
To that aim, we now define $d_{\Omega}$-variation and the length $L^{d_{\Omega}}$ it induces. We then prove that the length functionals $L^d_{\Omega}$ and $L^{d_{\Omega}}$ actually agree on curves when the starting space is a length space.

\begin{definition}
Let $\gamma:[a,b]\rightarrow X$ be a path in $(X,d_{{\Omega}})$.\ We define the {\em $d_\Omega$-variation $V^{d_{\Omega}}_\sigma(\gamma)$ of $\gamma$ with respect to a partition $\sigma=\{t_i\}_{i=0}^{m}$} and the {\em variational length  $L^{d_{{\Omega}}}(\gamma)$ of $\gamma$} by
\begin{align*}
    V^{d_{\Omega}}_\sigma(\gamma)& \coloneqq \sum_{i=0}^{m-1}d_{\Omega}(\gamma(t_i),\gamma(t_{i+1})),\\
    L^{d_{{\Omega}}}(\gamma)&\coloneqq \sup\bigg\{ V^{d_{\Omega}}_\sigma(\gamma): \sigma\textup{ partition of }[a,b]\bigg\}.
\end{align*}
\end{definition}
\begin{proposition}\label{prop:conformal_length_and_induced_length_coincide}
    Let \((X,d)\) be a length space and consider a conformal factor \(\Omega \colon X \to (0, +\infty)\). Then, for every path \(\gamma\) we have that \(L_\Omega^d(\gamma) = L^{d_\Omega}(\gamma)\).
\end{proposition}
\begin{proof}
    One inequality is straightforward: for any partition \(\sigma\) of \([a,b]\) we have that
    \[
        V^{d_\Omega}_{\sigma}(\gamma) = \sum_{i=1}^m d_{\Omega}(\gamma(t_i),\gamma(t_{i+1})) \leq \sum_{i=1}^m L_{\Omega}^d(\gamma\vert_{[t_i, t_{i+1}]}) = L_{\Omega}^d(\gamma),
    \]
    hence \(L^{d_\Omega}(\gamma) \leq L_\Omega^d(\gamma)\).
    
    For the opposite inequality, as \(\gamma\) is continuous in both topologies it is in particular uniformly continuous with respect to \(d_\Omega\) and thus
    for any \(\varepsilon > 0\) there exists \(\delta>0\) such that for any partition \(\sigma=\{t_i\}_{i=1}^m\) with \(\abs{\sigma} <\delta\) we have
    \[
        d_\Omega(\gamma(t_i),\gamma(t_{i+1})) \leq \varepsilon \quad \forall i \in \{1, \dots, m-1\},
    \]
    For every interval \([t_i,t_{i+1}]\) we can find a curve \(\eta_i \colon[t_i,t_{i+1}] \to X\) with \(\eta_i(t_i) = \gamma(t_i)\) and \(\eta_i(t_{i+1}) = \gamma(t_{i+1})\) such that
    \[
        L_\Omega^d(\eta_i) \leq d_\Omega(\gamma(t_i),\gamma(t_{i+1})) + \frac{\varepsilon}{m}.
    \]
    Concatenating such curves yields a curve \(\eta_\varepsilon\) such that
    \begin{equation}\label{eq:lower_semicontinuity_dream}
        L_\Omega^d(\eta_\varepsilon) = \sum_{i=1}^{m-1} L_\Omega^d(\eta_i) \leq \sum_{i=1}^{m-1} \left( d_\Omega(\gamma(t_i),\gamma(t_{i+1})) + \frac{\varepsilon}{m} \right)\leq L^{d_\Omega}(\gamma) + \varepsilon.
    \end{equation}
    We notice that for any \(t \in [a,b]\) we can look at the interval \([t_i,t_{i+1}]\) containing it and using uniform continuity together with the fact that \(\gamma(t_i) = \eta_\varepsilon(t_i) = \eta_i(t_i)\) we see that
    \[
        \begin{split}
            d_\Omega(\gamma(t),\eta_\varepsilon(t)) &\leq d_\Omega(\gamma(t),\gamma(t_i)) + d_\Omega(\gamma(t_i),\eta_\varepsilon(t)) \\
            &\leq \varepsilon + L_\Omega^d(\eta_i) \leq \varepsilon + \frac{\varepsilon}{m} + d_\Omega(\gamma(t_i), \gamma(t_{i+1}))  < 3\varepsilon.
        \end{split}
    \]
    Hence the sequence of curves \(\{\eta_\varepsilon\}_\varepsilon\) converges uniformly with respect to the \(d_\Omega\) metric to \(\gamma\); we set \(\varepsilon = 1/n\) so that we have countably many curves. Applying \Cref{prop:length_lower_semicontinuous} we get
    \[
        L^d_\Omega(\gamma) \leq \liminf_{n \to +\infty} L_\Omega(\eta_{1/n}) \leq L{d_{\Omega}}(\gamma),
    \]
    which proves the missing inequality.
\end{proof}
\begin{corollary}\label{coro:conformally_changed_space_is_length_space}
    If \((X,d)\) is a metric space and \(\Omega \colon X \to (0, +\infty)\)  is a conformal factor, then the space \((X,d_\Omega)\) is a length space.
\end{corollary}

Analogously to the proof in \Cref{prop:transitivity_conformal_length}, by using \Cref{prop:conformal_length_and_induced_length_coincide} one can show that doing consecutive conformal changes of a given length metric is the same as doing a single change by the product of the conformal factors:
\begin{proposition}\label{prop:transitivity_metric_case}
    Let \((X,d)\) be a length space and \(\Omega,\Omega' \colon X \to (0, +\infty)\) be conformal factors. Then it holds that \(L^{d_\Omega}_{\Omega'} = L^d_{\Omega\cdot\Omega'}\). In particular, we have that \((d_\Omega)_{\Omega'} = d_{\Omega \cdot \Omega'}\).
\end{proposition}

By \Cref{prop:same_topologies}, we know that the topologies induced by \(d\) and \(d_\Omega\) agree. In particular, curves are continuous with respect to \(d\) if and only if they are continuous with respect to \(d_\Omega\). In the following lemma, we prove that the same holds true for absolutely continuous curves.
\begin{lemma}\label{lemma:equivalent_absolute_continuity}
    Let \((X,d)\) be a length space and \(\Omega \colon X \to (0, +\infty)\) a conformal factor on \(X\). Consider a curve \(\gamma \colon [a,b] \to X\). Then \(\gamma\) is absolutely continuous with respect to \(d\) if and only if it is absolutely continuous with respect to \(d_\Omega\).
\end{lemma}
\begin{proof}
    Suppose that \(\gamma\) is absolutely continuous with respect to \(d\) and denote by \(M \coloneqq \max \Omega \circ \gamma\). Since \((X,d)\) is a length space, we know that the function \(f(t) \coloneqq L(\gamma\vert_{ [a,t]})\) is absolutely continuous as well (notice that \(\gamma\) has finite length, being absolutely continuous). Hence, for any \(\varepsilon >0\) we can find \(\delta>0\) such that for any finite disjoint family of intervals \(\{[a_i,b_i]\}_{i=1}^n\) we have that
    \[
        \sum_{i=1}^n \abs{b_i-a_i} <\delta \implies \sum_{i=1}^n \abs{f(b_i) -f(a_i)}  = \sum_{i=1}^n L^d(\gamma\vert_{ [a_i,b_i]}) < \frac{\varepsilon}{M}.
    \]
    We now see that for any such family of intervals it holds that
    \[
        \sum_{i=1}^n d_{\Omega}(\gamma(a_i),\gamma(b_i)) \leq \sum_{i=1}^n L_\Omega^d(\gamma\vert_{ [a_i,b_i]}) \leq M \sum_{i=1}^n L^d(\gamma\vert_{ [a_i,b_i]}) < M \cdot \frac{\varepsilon}{M} = \varepsilon,
    \]
    which shows that the curve is absolutely continuous with respect to \(d_\Omega\). The other implication is now trivial, since by denoting \(\tilde{d} \coloneqq d_\Omega\) we have that \(\tilde{d}_{1/\Omega} = d\) by \Cref{prop:transitivity_metric_case}, thus we can apply the previous reasoning.
\end{proof}

With
\Cref{lemma:equivalent_absolute_continuity}
at hand, we can now demonstrate that the metric speed of curves with respect to the conformal distance is given by the product of the original metric speed of the curve by the conformal factor. This is completely analogous to what happens in the smooth setting when multiplying a metric tensor by a conformal factor. More importantly, this establishes an equivalence between 
the notion of conformal length given in \cite{Han2019} and the definition given in this work. 
\begin{proposition}
    Consider a length space \((X,d)\) and let \(\Omega\) be a conformal factor on \(X\). For any absolutely continuous curve \(\gamma\) (with respect to any of the metrics, see \Cref{lemma:equivalent_absolute_continuity}), we denote by \(v_\gamma(t)\) its metric speed with respect to \(d\) and by \(v_{\Omega,\gamma}(t)\) its metric speed with respect to \(d_\Omega\). Then it holds that
    \[
        v_{\Omega,\gamma}(t) = \Omega(\gamma(t))v_\gamma(t)
    \]
    for almost every \(t \in [a,b]\). In particular, we have that
    \[
        L_\Omega^d(\gamma\vert_{ [a,t]}) = \int_a^t \Omega(\gamma(t)) v_{\gamma}(t) \, dt.
    \]
\end{proposition}
\begin{proof}
    Let us consider any  time \(s \in [a,b)\) where both \(t\mapsto L(^d\gamma\vert_{[a,t]})\) and \(t \mapsto L_\Omega^d(\gamma\vert_{[a,t]})\) are differentiable. We know that for any \(t_1 < t_2\) it holds that
    \[
        \min_{t \in [t_1,t_2]} \Omega(\gamma(t)) \,  L^d(\gamma\vert_{[t_1,t_2]}) \leq L_\Omega^d(\gamma\vert_{[t_1,t_2]}) \leq \max_{t \in [t_1,t_2]} \Omega(\gamma(t))\, L^d(\gamma\vert_{[t_1,t_2]}).
    \]
    If there exists \(t>s\) such that \(L^d(\gamma\vert_{ [s,t]}) = 0\), then by the previous inequalities and the fact that \(\Omega > 0\) we have that \(L_\Omega^d(\gamma\vert_{ [s,t]}) = 0\) and so in both cases the metric speed vanishes and thus the result is trivial, so we assume that for every \(s'>s\) it holds that \(L^d(\gamma\vert_{ [s,s']}) > 0\). We can then take the previous chain of inequalities and write it as
    \[
        \min_{t \in [s,s']} \Omega(\gamma(t)) \leq \frac{L_\Omega^d(\gamma\vert_{[s,s']})}{L^d(\gamma\vert_{[s,s']})} \leq  \max_{t \in [s,s']} \Omega(\gamma(t)),
    \]
    so by sending \(s' \to s^+\) and using continuity of \(\Omega \circ \gamma\) we find that
    \begin{equation} \label{eq:infinitesimal_ratio_lengths}
        \lim_{s' \to s^+} \frac{L_\Omega^d(\gamma\vert_{[s,s']})}{L^d(\gamma\vert_{[s,s']})} = \Omega(\gamma(s)).
    \end{equation}
    From the classical theory of length spaces (cf. \cite{Burago_2001}) we know that at every point \(s\) of differentiability of \(t \mapsto L_\Omega^d(\gamma\vert_{[a,t]})\) it holds that
    \[
        {\frac{d}{dt}}\vert_{ t=s} L_\Omega^d(\gamma\vert_{ [a,t]}) = \lim_{t \to s} \frac{L_{\Omega}^d(\gamma\vert_{ [s,t]})}{t-s} = v_{\Omega,\gamma}(s).
    \]
    As we are at a point of differentiability, we can then compute \(v_{\Omega,\gamma}(s)\) using the right-sided derivative:
    \[
        v_{\Omega,\gamma}(s) = \lim_{s' \to s^+} \frac{L^d_\Omega(\gamma\vert_{[s,s']})}{s'-s} = \lim_{s' \to s^+} \frac{L_\Omega^d(\gamma\vert_{[s,s']})}{L^d(\gamma\vert_{[s,s']})} \cdot \frac{L^d(\gamma\vert_{[s,s']})}{s'-s} = \Omega(\gamma(s))v_\gamma(s),
    \]
    concluding the proof.
\end{proof}

\subsection{A generalization of Nomizu-Ozeki's theorem}\label{section 5.3}
In the smooth setting, an important result by Ozeki and Nomizu \cite{nomizu61} shows that every Riemannian manifold can be endowed with a Riemannian metric which is conformally related to the previous one and makes the space Cauchy-complete. Such result can be generalized: if a metric space is locally compact, then it admits a metric which makes it complete. Indeed, one can simply embed the space into its metric completion and the condition of being locally compact implies that such embedding is an open map. From the theory of Polish spaces, we get that an open subset of a Polish space is Polish: although we did not assume our space to be separable, the same proof can be applied to show that an open subset of a complete space is completely metrizable. It is worth mentioning, however, that this generalization does not retain the additional information that the new metric is conformally related to the previous one according to our notion of conformality. In the following theorem, we show that the same proof given in \cite{nomizu61} applies to our notion of conformal distance in the setting of locally compact length spaces, and the resulting complete space is conformally related to the original one.
\begin{theorem}\label{thm:make_a_space_complete}
    Suppose \((X,d)\) is a locally compact length space. Then there exists a metric \(d_\Omega\) which is conformally related to \(d\) such that the space \((X,d_\Omega)\) is Cauchy-complete.
\end{theorem}
\begin{proof}
    We consider the function \(\rho \colon X \to \R \cup \{+\infty\}\) given by
    \[
        \rho(x) \coloneqq \sup \{ r> 0 \ \colon\ \overline{B_r(x)} \ \text{is compact} \}.
    \]
    Notice that the local compactness assumption implies that \(\rho(x)>0\) for all \(x \in X\). If there exists a point such that \(\rho(x) = +\infty\) we would get that \(X\) itself is compact and thus complete, hence the trivial conformal change \(\Omega \equiv 1\) yields the conclusion.

    We therefore assume that \(\rho(x) < +\infty\) for all $x\in X$, so that \(\rho\) is actually a real-valued, positive function. In this case, \(\rho\) is a 1-Lipschitz function: if not, then there are \(x,y \in X\) with \(\rho(y) > \rho(x) +d(x,y)\). Then there exists a positive radius \(r > \rho(x) +d(x,y)\) such that \(\overline{B_r(y)}\) is compact; notice that this implies that
    \[
        r' := r-d(x,y) > \rho(x) > 0.
    \]
    We then consider the ball \(\overline{B_{r'}(x)}\). By the triangle inequality, such ball is contained in \(\overline{B_r(y)}\), which is compact and therefore it is compact as well; however we noticed that \(r' > \rho(x)\), which contradicts the fact that \(\rho(x)\) is the supremum of the radii for which closed balls at \(x\) are compact. Hence we have
    \[
        \abs{\rho(y) - \rho(x)} \leq d(x,y),
    \]
    showing that \(\rho\) is 1-Lipschitz. We then set \(\Omega \coloneqq 1/\rho\) and notice that it is a positive, continuous function on \(X\), so that we may use it as a conformal factor and consider the metric \(d_\Omega\); recall that by \Cref{prop:same_topologies}, the topology of the space does not change.
    
    We now aim to show that for every point \(x \in X\) the ball \(B \coloneqq B_{\Omega,1/3}(x)\) with respect to the metric \(d_\Omega\) is pre-compact. This would immediately imply that the space is complete. We will prove that \(B \subset\overline{B_{\rho(x)/2}(x)}\): for any \(y \in B\) we have that \(d_\Omega(x,y) < 1/3\); as \((X,d_\Omega)\) is a length space by \Cref{coro:conformally_changed_space_is_length_space}, there exists a curve \(\gamma\) from \(x\) to \(y\) with \(L^d_{\Omega}(\gamma) < 1/3\). As done several times in previous proofs, we can relate \(L^d_\Omega\) to \(L^d\) via
    \[
        \Omega(\gamma(t))L^d(\gamma) \leq L^d_\Omega(\gamma) < 1/3,
    \]
    where \(\gamma(t)\) is the point for which \(\Omega \circ \gamma\) attains its minimum. On recalling that \(\Omega = 1/\rho\) and that \(\rho\) is 1-Lipschitz we may write
    \[
        L^d(\gamma) < \frac{\rho(\gamma(t))}{3} \leq\frac{\rho(x)}{3} + \frac{d(x,\gamma(t))}{3} \leq \frac{\rho(x)}{3} + \frac{L^d(\gamma)}{3}.
    \]
    Hence we have that
    \[
        L^d(\gamma) < \frac{\rho(x)}{2} \implies d(x,y) < \frac{\rho(x)}{2},
    \]
    hence \(y \in \overline{B_{\rho(x)/2}(x)}\), thus \(B \subset \overline{B_{\rho(x)/2}(x)}\). As \(\rho(x)/2<\rho(x)\), the ball on the right-hand-side is compact and thus \(B\) is precompact, concluding the proof.
\end{proof}
\begin{remark}
    Applying the previous theorem together with the Hopf--Rinow theorem, we have that any length metric on a locally compact space can be conformally changed to make the space geodesically complete. In particular, we also get that between any two points we can find a length minimizer which joins them, so the space becomes strictly intrinsic.
\end{remark}
\subsection{Hausdorff measures}\label{section 5.4}
We conclude this section by showing that the same reasoning we applied in \Cref{Lorentzian Hausdorff measure} can be used in the metric case to prove a relation between the Hausdorff measures induced by conformally related metrics. We will often additionally assume that the space is separable. For a given dimensional exponent \(s>0\), we will denote by \(\mathcal{H}^s_d\) and \(\mathcal{H}^s_{d_\Omega}\) the \(s\)-dimensional Hausdorff measures with respect to \(d\) and \(d_\Omega\) respectively.

Let us first show that an analogue of \Cref{lemma:hausdorff_measure_control} holds.
\begin{lemma}
    Consider a length space \((X,d)\) and $\Omega\colon X \rightarrow(0,\infty)$ a conformal factor. For any point \(x \in X\), any Borel set \(E\), any dimension exponent \(s>0\) and any \(\varepsilon > 0\) there exists a positive value \(r(x,\varepsilon)>0\) such that for any \(r \in (0,r(x,\varepsilon))\) it holds that
    \[
        (\Omega(x)^s - \varepsilon)\mathcal{H}_d^s(E \cap B(x,r)) \leq \mathcal{H}_{d_\Omega}^s(E \cap B(x,r)) \leq (\Omega(x)^s+ \varepsilon) \mathcal{H}^s_d(E \cap B(x,r)).
    \]
\end{lemma}
\begin{proof}
    The proof goes analogously as in \Cref{lemma:hausdorff_measure_control}: for a given \(x \in X\), we fix \(\varepsilon'>0\) and employ \Cref{coro:conformal_metrics_are_loc_lip} to find an open ball \(B(x,r')\) around \(x\) such that for any \(y,z \in B(x,r')\) it holds that
    \[
        (\Omega(x)-\varepsilon)d(y,z) \leq d_\Omega(y,z) \leq (\Omega(x) + \varepsilon)d(y,z).
    \]
    We consider any radius \(r''<r'/2\) and apply the same reasoning as in \Cref{lemma:hausdorff_measure_control}; this time we employ the above estimate to bound the diameter of the covering sets \(\{J_i\}_i\) as
    \[
        (\Omega(x) - \varepsilon')\operatorname{diam}(J_i)\leq \operatorname{diam}_\Omega(J_i) \leq (\Omega(x)+\varepsilon') \operatorname{diam}(J_i).
    \]
    The rest of the proof is identical to \Cref{lemma:hausdorff_measure_control}.
\end{proof}

It is immediate to see that all the proofs after \Cref{lemma:hausdorff_measure_control} in \Cref{Lorentzian Hausdorff measure} hold for generic measures which satisfy the estimate featured in \Cref{lemma:hausdorff_measure_control}. As the previous proof shows that this is the case for \(\mathcal{H}^s_d\) and \(\mathcal{H}_{d_\Omega}^s\), all of the results apply here as well. We summarize them in the following theorem.
\begin{theorem}
    Consider a separable length space \((X,d)\) and a conformal factor \(\Omega \colon X \to (0, +\infty)\). Then for any \(s>0\) the measures \(\mathcal{H}^s_d\) and \(\mathcal{H}^s_{d_\Omega}\) are mutually absolutely continuous with respect to each other. In particular, one vanishes if and only if the other does as well. Moreover, one is \(\sigma\)-finite if and only if the other one is as well and in such case, for any borel set \(E\) we have that
    \[
        \mathcal{H}^s_{d_\Omega}(E) = \int_E \Omega(x)^s \, d\mathcal{H}^s_d(x).
    \]
\end{theorem}

In the metric case we also get an additional result which is not so trivial for Lorentzian pre-Length spaces, namely that the Hausdorff dimension of separable length spaces is invariant under conformal transformations.
\begin{proposition}
    Consider a separable length space \((X,d)\) and fix a conformal factor \(\Omega \colon X \to (0, +\infty)\). Then the Hausdorff dimension of \(X\) with respect to \(d\) equals the Hausdorff dimension with respect to \(d_\Omega\). Equivalently, the Hausdorff dimension of a separable length space is conformally invariant. 
\end{proposition}
\begin{proof}
    The claim immediately follows the previous theorem, in particular from the fact that the measure are mutually vanishing and from the fact that
    \[
        \dim_d(X) = \inf\{s>0 \ \colon \mathcal{H}^s_d(X) = 0\} = \inf\{s>0 \ \colon \mathcal{H}^s_{d_\Omega}(X) = 0\} = \dim_{d_\Omega}(X).
    \]
\end{proof}